\documentclass[11pt]{article}
 
\usepackage{amstext}
\usepackage{amsthm}
\usepackage{amssymb}
\usepackage{amsmath,enumerate}
\usepackage{cases}

\usepackage{comment}
\usepackage{thm-restate}
\usepackage{centernot}
\usepackage{url} 
\usepackage{subfigure}
\usepackage{enumitem}
\usepackage{mathrsfs}
\usepackage[dvipsnames]{xcolor}

\usepackage[normalem]{ulem}
\usepackage{setspace}

\usepackage[margin=1in]{geometry}

\newcommand{\eps}{\varepsilon}

\renewcommand{\v}{\textup{\textsf{v}}}

\newenvironment{proofclaim}[1][]%
{\noindent \emph{Proof.} {}{#1}{}}{\hfill
	$\Diamond$\vspace{1em}}

\newcommand{\thistheoremname}{}
\newtheorem{genericthm}[section]{\thistheoremname}

\newcommand{\overbar}[1]{\mkern 1.5mu\overline{\mkern-1.5mu#1\mkern-1.5mu}\mkern 1.5mu}

\setlistdepth{9}

\newlist{myEnumerate}{enumerate}{9}
\setlist[myEnumerate,1]{label=(\arabic*)}
\setlist[myEnumerate,2]{label=(\Roman*)}
\setlist[myEnumerate,3]{label=(\Alph*)}
\setlist[myEnumerate,4]{label=(\roman*)}
\setlist[myEnumerate,5]{label=(\alph*)}
\setlist[myEnumerate,6]{label=(\arabic*)}
\setlist[myEnumerate,7]{label=(\Roman*)}
\setlist[myEnumerate,8]{label=(\Alph*)}
\setlist[myEnumerate,9]{label=(\roman*)}

\def\P{\mathbb P}

\usepackage[pdftex]{graphicx} 

\usepackage[pdftex,pagebackref=false]{hyperref} 
\usepackage[noabbrev,capitalise]{cleveref}
\hypersetup{
    plainpages=false,       
    unicode=false,          
    pdftoolbar=true,        
    pdfmenubar=true,        
    pdffitwindow=false,     
    pdfstartview={FitH},    
    pdfnewwindow=true,      
    colorlinks=true,        
    linkcolor=blue,         
    citecolor=green,        
    filecolor=magenta,      
    urlcolor=cyan           
}

\theoremstyle{plain}
\newtheorem{thm}{Theorem}[section]
\newtheorem{lem}[thm]{Lemma}
\newtheorem{claim}[thm]{Claim}
\newtheorem{prop}[thm]{Proposition}
\newtheorem{cor}[thm]{Corollary}

\newtheorem{conj}[thm]{Conjecture}
\newtheorem{definition}[thm]{Definition}
\newtheorem{que}[thm]{Question}

\title{When are random regular triangle-free graphs bipartite?}

\author{Gregory DeCamillis\\ University of Waterloo\\ gdecamillis@uwaterloo.ca \and Pu Gao \\ University of Waterloo\\ pu.gao@uwaterloo.ca}

\date{}
\begin{document}

\maketitle

\begin{abstract}

We study the structure of random $d$-regular triangle-free graphs and show that a sharp phase transition occurs at $d=\frac{\sqrt 3}{2}\sqrt{n \log n}$. For smaller $d$, asymptotically almost surely the graph is non-bipartite, whereas for greater $d$, asymptotically almost surely the graph is bipartite.

\end{abstract}

\section{Introduction}
One of the fundamental problems in random graph theory is estimating the probability that a random graph has a property $P$ as the number of vertices goes to infinity. The most classical random graphs include $G_{n,p}$ and $G_{n,m}$, both of which are defined on vertex set $[n]$. In $G_{n,p}$, every pair of distinct vertices are adjacent independently with probability $p$, whereas $G_{n,m}$ denotes the random graph chosen uniformly from those having exactly $m$ edges.  Note that $G_{n,m}$ is distributed as $G_{n,p}$ conditioned on having $m$ edges, and it is often straightforward to transfer many a.a.s.\ properties between $G_{n,p}$ and $G_{n,m}$, with $m=p\binom{n}{2}$.  Conditioning on $G_{n,p}$ being $d$-regular yields another extensively studied random graph  $G_{n,d}$, a  graph chosen uniformly at random from the set of all $d$-regular graphs on $[n]$. Although $G_{n,p}$ and $G_{n,d}$ share many a.a.s.\ properties with $d=pn$, until the recent work on resolving the sandwich conjecture~\cite{behague2025proof, gao2025kimvussandwichconjecturetrue, gao2025sandwichingrandomregulargraphs}, there is no easy approach to translate results from $G_{n,p}$ to $G_{n,d}$. For instance, the Hamiltonicity of $G_{n,m}$ was completely determined in the early 1980s by Koml\'os and Szemer\'edi \cite{KOMLOS198355} whereas understanding the Hamiltonicity of $G_{n,d}$ involved a series of papers \cite{BOLLOBAS198397,COOPER_FRIEZE_REED_2002, FENNER1983301,  Krivelevich2001, Robinson1992, Robinson1994}, the last of which was published in 2002. In fact, unlike $G_{n,p}$ in which edges appear independently and standard probabilistic tools such as the Chernoff bounds, the Azuma--Hoeffding inequality, and Janson's inequality can be readily applied, the analysis of $G_{n,d}$ typically requires more technical enumerative methods, as these tools generally do not apply directly.

The study of subgraph containment is central to random graph theory. Given a graph $H$, what is the probability that $H$ appears as a subgraph in the random graph? What is the distribution of $X_H$, the number of copies of $H$ in the random graph? The classical results study the phase transition when $H$ starts to appear (e.g.\ in $G_{n,m}$ or $G_{n,d}$), and the limiting distribution of $X_H$ when $p$ or $d$ is above the emergence threshold. More recently, there has been considerable interest in understanding the extreme tails of  $X_H$~\cite{demarco2012,Dembo2018,  harel2022, JansonOleszkiewiczRucinski2004}, far away from its standard deviation. In particular, when $\mathbb{E} X_H$ is large, what is the probability that $X_H=0$? Estimates of this probability in $G_{n,m}$ and $G_{n,d}$ immediately yield estimations for the following enumeration problems.

\begin{que}
 \label{q:enumeration}   \begin{enumerate}
    \item[(a)]   How many $n$-vertex $m$-edge graphs are there that are $H$-free?
   
    \item[(b)] How many $n$-vertex $d$-regular graphs are there that are $H$-free?
   \end{enumerate}
\end{que}

Another direction of the research 
studies $G_{n,m}$ and $G_{n,d}$ conditioning on the event $X_H=0$.
\begin{que}\label{q:structure}
    What can be said about the structures of $G_{n,m}\mid \{X_H=0\}$ and $G_{n,d}\mid \{X_H=0\}$?
\end{que}

In extremal graph theory, an $H$-free graph on $[n]$ with the maximum number of possible edges has a certain structure. For instance, if $H=K_3$, which we call a triangle, then any triangle-free graph can have at most $n^2/4$ --- known as the Tur\'{a}n number for $K_3$--- edges and graphs with that many edges  must be  bipartite.
Question~\ref{q:structure} can be viewed as a probabilistic version of it. In the case $H=K_3$, Question~\ref{q:structure} asks when $G_{n,m}\mid \{X_H=0\}$ and $G_{n,d}\mid \{X_H=0\}$ become asymptotically almost surely (a.a.s.) bipartite.

In this paper, we focus on the case when $H=K_3$. We briefly review work in the literature around Questions~\ref{q:enumeration} and~\ref{q:structure} in this special case. More discussions on general $H$ will be given in \cref{sec: other}. We state  results in $G_{n,m}$ although some of the results we discuss were originally stated in $G_{n,p}$.  As discussed earlier, results can often be transferred straightforwardly between the two models (either by conditioning on the number of edges, or by coupling the two models, or by almost the same proof in the two models); we indicate in the survey whenever such a transfer is not immediate. To simplify the notation, we write
\[
T_{n,m}=G_{n,m}\mid\{X_{K_3}=0\},\quad T_{n,d}=G_{n,d}\mid\{X_{K_3}=0\}.
\]
See the beginning of \cref{sec: preleminaries} for our definitions of asymptotics.

\subsection{The irregular setting}\label{sec: irregular}

First consider $G_{n,m}$ and $T_{n,m}$. A phase transition in the way the triangles cluster in $G_{n,m}$ occurs when $m\approx n^{3/2}$, which naturally splits the range of $m$ into the following three regimes where $\eps>0$ is an arbitrarily small constant:\smallskip

\begin{center}
\begin{minipage}{10cm}
\raggedright
\noindent\textit{Sparse regime}: $m \leq n^{3/2 - \eps}$;

\noindent\textit{Intermediate regime}:  $n^{3/2 - \eps}<m < (1+\eps) \frac{\sqrt 3}{4}n\sqrt{n \log n}$;

\noindent\textit{Dense regime:} $m \ge (1+\eps) \frac{\sqrt 3}{4}n\sqrt{n \log n}$. 

\end{minipage}
\end{center}

\smallskip

 Work in the sparse regime was initiated by Erd\H{o}s and R\'enyi in their seminal paper on random graphs \cite{erdos1960evolution}, in which they showed that if $m\sim cn/2$ then the limiting distribution of $X=X_{K_3}$ converges to  Poisson with mean $c^3/6$ as $n\to \infty$.  A number of small improvements were made to this result \cite{Bollobs_1981, Karonski1983} with more careful analysis. The first major improvement came from Janson, {\L}uczak, and Ruci\'nski in \cite{janson1988exponential}, where they extended the previous result to $m=o(n^{6/5})$ by applying Janson's correlation inequalities to handle cases where $X$ has a higher variance.  This result was originally written in $G_{n,p}$, and was later extended to $G_{n,m}$ by Pr\"omel and Steger \cite{PROMEL1996311}. 

Soon after \cite{janson1988exponential}, Frieze  \cite{Frieze1992SmallSubgraphs} gave a different proof for the distribution of $X$ beyond the $m \sim cn$ setting. This proof relied on what would later be known as the switching method. Wormald described the switching method in more detail in \cite{wormald1996}, and used it to extend the estimation of $\P[X = 0]$ up to $m=o(n^{4/3})$. Finally,  by more refined switching arguments, Stark and Wormald~\cite{STARK_WORMALD_2018} obtained the asymptotic estimation of $\P[X = 0]$ covering all $m$ in  the sparse regime.
Interestingly, the expression of $\log \P[X = 0]$ is a convergent power series in $\hat p=m/\binom{n}{2}$, which no longer agrees with the Poisson distribution in~\cite{PROMEL1996311}.

As for the dense regime, Pr\"omel and Steger initiated the study of Question~\ref{q:structure} for $T_{n,m}$ in \cite{promel1996asymptotic}. This was fully resolved by Osthus, Pr\"omel, and Taraz in \cite{osthus2003densities}, where they showed the following.

\begin{thm}\label{thm: osthus 0-1}\cite{osthus2003densities}
            Let $G \sim T_{n,m}$. For all $\eps > 0$,
            \begin{numcases}{\P[G \text{ is bipartite}] \rightarrow} 
                    1 & if $\quad m = o(n)$\notag\\
                    0 &if $\quad n/2 \leq m \leq (1 - \eps)\frac{\sqrt 3}{4}n\sqrt{n\log n}$\notag\\
                    1 & if $\quad m \geq (1 + \eps)\frac{\sqrt 3}{4}n\sqrt{n\log n}$ \label{eq: second 1 statement}
                .
            \end{numcases}
        \end{thm}
The first phase transition in the theorem above corresponds to the appearance of the first cycle in $G_{n,m}$, which was already well understood. Note also that Theorem~\ref{thm: osthus 0-1} also answers Question~\ref{q:enumeration} (a) for $H=K_3$ in the dense regime.

This leaves Question~\ref{q:enumeration}(a) and Question~\ref{q:structure}
in the intermediate regime. {\L}uczak~\cite{luczak2000triangle} proved a stability result for $T_{n,m}$, given below, which is akin to a randomized version of Erd\H{o}s and Simonovits's stability theorem~\cite{ErdosSimonovits1966} for near extremal graphs.
This gives a partial answer to Question~\ref{q:structure} for $T_{n,m}$ in the intermediate regime.

\begin{definition}
    For any $0 \leq \delta \leq 1$, a graph $G$ is $\delta$-bipartite if $G$ can be made bipartite by deleting at most $\delta|E(G)|$ edges from $G$.
\end{definition}

\begin{thm}\label{thm: luc 13}\cite{luczak2000triangle}
    For every $\delta > 0$ there exists a constant $c = c(\delta)$ such that, for $m \geq cn^{3/2}$,
    $T_{n,m}$ is $\delta$-bipartite with probability $1 - o(1)$.
\end{thm}
As for Question~\ref{q:enumeration}(a), {\L}uczak~\cite{luczak2000triangle} obtained an estimation for $\P[X = 0]$ in $G_{n,m}$, given below,  in the intermediate regime, which follows as a corollary of \cref{thm: luc 13}.

\begin{thm}\label{thm: luczak estimate}\cite{luczak2000triangle}
    Let $G \sim G_{n,m}$. For every $\eps > 0$, there exists a constant $c = c(\eps)$ such that if $cn^{3/2} \leq m \leq n^2/c$, then 
    \[
    \left(\frac{1}{2} - \eps\right)^m \leq \P[X = 0] \leq \left(\frac{1}{2} + \eps\right)^{m}.
    \]
\end{thm}
An asymptotically accurate estimation of $\P[X = 0]$ is obtained in a recent breakthrough by Jenssen, Perkins, and Potukuchi~\cite{JENSSEN2025110499} when $m \geq \frac{13}{56}n\sqrt{n\log n}$. For smaller $m$ in the intermediate regime,
Jenssen, Perkins, Potukuchi, and Simkin~\cite{jenssen2024lower}
obtained an approximation of $\log \P[X=0]$ up to an additive error $o(n^{3/2})$. These two results improve Theorem~\ref{thm: luczak estimate}.

\subsection{The regular setting}\label{sec: regular}
Now we turn to $G_{n,d}$ and $T_{n,d}$. 
We say $d$ is in the sparse, intermediate, or the dense regime, if $m=dn/2$ is in the sparse, intermediate, or the dense regime respectively.
Again, $X=X_{K_3}$ denotes the number triangles.  Independently, Bollob\'as \cite{Bollobas1980} and Wormald \cite{Wormald1978,Wormald1981} proved that the joint distribution of the number of short cycles converges to independent Poisson when $d$ is constant, which immediately gives that $\P[X=0]\to \exp(-(d-1)^3/6)$ as $n\to \infty$ for constant $d\ge 2$. McKay, Wormald, and Wysocka~\cite{McKayWormaldWysocka2004} extended it to  $d=o(n^{1/5})$.  

The asymptotic estimation of $\P[X=0]$ in the sparse regime for $d=\Omega(n^{1/5})$ is open. The limiting distribution of $X$ was shown to converge to normal $d = O(n^{1/2})$ (see \cite{Gao2024,GaoWormald2008}) or $d \geq n/\log n$ (see \cite{Sah_Sawhney_2023}), but this only gives meaningful information when $X$ is within a constant number of standard deviations of its mean.

Since $G_{n,d}$ has the same distribution as $G_{n,m}$, with $m=dn/2$, conditioned on being regular, and it is easy to show that the probability that $G_{n,m}$ is regular is $\exp(-\Theta(n\log d))$, we can immediately translate any estimate of $\log\P_{G_{n,m}}[X=0]$ to an upper bound on $\log\P_{G_{n,d}}[X=0]$ with the same expression with an additional additive error $O(n\log d)$. For instance,
in the sparse regime, the estimation of $\log\P[X=0]$ in $G_{n,m}$ by Stark and Wormald~\cite{STARK_WORMALD_2018} immediately gives an upper bound on $\log\P[X=0]$  in $G_{n,d}$, by taking 
the power series of $\log_{G_{n,m}} \P[X=0]$ (with $m=dn/2$) in~\cite{STARK_WORMALD_2018}, plus an additional additive error $O(n\log d)$. Of course, terms in the power series of order  $O(n\log d)$ can be truncated.

Similarly, in the dense regime (but with $d=o(n)$), translating \cref{thm: luczak estimate} into the regular setting gives the following, where the additive error $O(n\log d)$ is already absorbed by $o(dn)$:
    \begin{equation}\label{eq: intermediate approx}
        \log\P[X = 0] \leq \big(\log(1/2)+o(1)\big) \frac{dn}{2}.
    \end{equation}
The afore-mentioned results~\cite{JENSSEN2025110499, jenssen2024lower} on the estimation of $\log \P_{G_{n,m}}[X=0]$ in the intermediate regime can  be translated to upper bounds of $\log \P_{G_{n,d}}[X=0]$, with an additional additive error $O(n\log d)$ factor, in a similar manner.
This $O(n\log d)$ error comes from conditioning $G_{n,m}$ on being regular, and so the upper bounds given above are unlikely to be asymptotically tight, even  for $\log \P[X=0]$.

The set of $d$-regular bipartite graphs
gives a trivial lower bound for $\log \P[X=0]$:
\[
\log \P[X=0] \ge (\log(1/2)+o(1)) dn/2.
\]
Combining with~\eqref{eq: intermediate approx}, this gives an estimation of $\P[X=0]$, up to a multiplicative error $\exp(O(dn))$. One of our main contributions is to give an asymptotic 
estimation of $\P[X=0]$ in the dense case (see Corollary~\ref{cor: enum} below). 
We are not aware of any nontrivial lower bound for $d=\Omega(n^{1/5})$ in the sparse regime and for $d$ in the intermediate regime.

 \subsection{Our contribution}
Our primary contribution is establishing the threshold for $T_{n,d}$ being bipartite, which resolves Question~\ref{q:structure} for $T_{n,d}$. Note that $d\le n/2$ is always assumed when we discuss $T_{n,d}$ as otherwise the probability space is empty. Likewise, we assume that $n$ is even, as otherwise there are no regular bipartite graphs on $[n]$.

\begin{thm}\label{thm: 0-1}
Assume $n$ is even and $G \sim T_{n,d}$. For all $\varepsilon > 0$,
    \begin{align*}
        \mathbb{P}[G \text{ is bipartite}] \rightarrow \begin{cases}
            0 &\text{if } 2 \leq d \leq (1 - \varepsilon)\frac{\sqrt 3}{2}\sqrt{n\log n}\\
            1 &\text{if }  (1 + \varepsilon)\frac{\sqrt 3}{2}\sqrt{n \log n} \le d \le n/2.
        \end{cases}
        .
    \end{align*}
\end{thm}

As an immediate corollary, we obtain the asymptotic estimate for $\log \P[X=0]$ in $G_{n,d}$ in the dense regime, answering Question~\ref{q:enumeration} (b) in the dense regime.

\begin{cor}\label{cor: enum}
    Suppose that $ (1 + \eps)\frac{\sqrt3}{2}\sqrt{n\log n} \le d\le n/2$ for some constant $\eps > 0$. Then, the number of $d$-regular triangle-free graphs with vertex set $[n]$ is asymptotic to
    \begin{equation}\notag
        \frac{1}{2}\binom{n}{n/2}\binom{n/2}{d}^n\binom{n^2/4}{dn/2}^{-1}e^{-1/2}.
    \end{equation}
    
\end{cor}

Let $B_{n,d}$ be the uniform probability distribution over all $d$-regular bipartite graphs with vertex set $[n]$. Another direct corollary provides the following approximation of $T_{n,d}$ by $B_{n,d}$ in total variation distance.

\begin{cor} Suppose that $(1 + \eps)\frac{\sqrt3}{2}\sqrt{n\log n}\le d\le n/2$ for some constant $\eps > 0$. Then, $$d_{TV}(T_{n,d},B_{n,d})=o(1).$$

\end{cor}

\noindent {\bf Remark.}
    With more careful analysis, an explicit upper bound for the above total variation distance can be obtained.

For $0 \leq p \leq 1$, let $B_{n,p}$ denote a random graph obtained by first choosing a balanced bipartition $(A,B)$ of $[n]$ uniformly at random and then including each edge between $A$ and $B$ independently with probability $p$. Our next corollary gives a sandwiching approximation of $T_{n,d}$ by $B_{n,p}$.

\begin{cor}\label{cor: our sandwich}
     Suppose that $d = \omega((n\log n)^{3/4})$ and $d\le n/2$. Then, there exist $p_* = (2 - o(1))d/n$, $p^* = (2 + o(1))d/n$, and a coupling $(G_*, G, G^*)$ where $G_* \sim B_{n, p_*}$, $G^* \sim B_{n, p^*}$, and $G \sim T_{n,d}$ such that $\P[G_* \subseteq G \subseteq G^*] = 1 - o(1)$.
\end{cor}

\noindent {\bf Remark.} 
(a) Interestingly, it is not possible to sandwich $T_{n,d}$ between two correlated $T_{n,m_*}$ and $T_{n,m^*}$ such that $m_*,m^*=(1+o(1))dn/2$. This is because although $T_{n,m}$ is a.a.s.\ bipartite by Theorem~\ref{thm: osthus 0-1} in the dense regime, a.a.s.\ it does not have a perfectly balanced bipartition.

(b) The condition $d=\omega((n\log n)^{3/4})$ is inherited from a theorem (See Theorem~\ref{thm: bip coupling} below) which is used in the proof. It is believed that Theorem~\ref{thm: bip coupling} holds for $d\gg \log n$. However, no such statement has been proved yet.

\smallskip

The proof of Corollary~\ref{cor: our sandwich} uses a result of Klimo{\v{s}}ov{\'a}, Reiher, Ruci{\'n}ski, and {\v{S}}ileikis on sandwiching uniform random $d$-regular bipartite graphs, given below.
\begin{thm}\cite{KlimosovaReiherRucinskiSileikis2023}\label{thm: bip coupling}
For any $0 \leq p \leq 1$ and partition $(A,B)$ of $[n]$ with $|A| = |B| = n/2$, let $G_{A, B, p}$ be a random graph chosen by including each edge between $A$ and $B$ independently with probability $p$. Let $G_{A,B, d}$ be a uniform random $d$-regular bipartite graph with bipartition $(A,B)$. 
Suppose that $d = \omega((n\log n)^{3/4})$. Then, there exist $p_* = (2- o(1))d/n$,
$p^* = (2 + o(1))d/n$ and a coupling  $(G_*, G, G^*)$ such that $G_* \sim G_{A,B, p_*}$, $G^* \sim G_{A,B, p^*}$,
$G \sim G_{A,B, d}$ and $P(G_* \subseteq G \subseteq G^*) = 1- o(1)$.    
\end{thm}

\begin{proof}  Let $p_*,p^*$ be given by Theorem~\ref{thm: bip coupling}.
    Let $\rho$ be the probability that $T_{n,d}$ is nonbipartite.  By Theorem~\ref{thm: 0-1}, $\rho=o(1)$. Moreover, $T_{n,d}$ conditional on $T_{n,d}$ being bipartite is exactly $B_{n,d}$. To couple $(G_*,G,G^*)$, with probability $\rho$, sample $G$ from $T_{n,d}$ conditional on $T_{n,d}$ being non-bipartite, and sample $G_*\sim B_{n,p_*}$, and $G^*\sim B_{n,p^*}$ independently. With the remaining probability $1-\rho$, take a uniform balanced bipartition $(A,B)$, and then sample $G$, $G_*$, and $G^*$ with bipartition $(A,B)$ according to the coupling given in \cref{thm: bip coupling}. It is clear that $G$, $G_*$, and $G^*$ have the correct marginal distributions given by the coupling above. Thus, combining $\rho=o(1)$, \cref{thm: bip coupling} implies that $\P[G_* \subseteq G \subseteq G^*] = 1 - o(1)$.
\end{proof}

\subsection{Other related results}\label{sec: other}
Here, we discuss a few generalizations of results stated in \cref{sec: irregular} and \cref{sec: regular}. For any graph $H$, let
$G^{H^-}_{n,m}$ denote $G_{n,m}$ conditioned on being $H$-free. 
Osthus, Pr\"omel, and Taraz \cite{osthus2003densities} strengthened \cref{thm: osthus 0-1} to cover all cases where $H$ is an odd cycle. 

\begin{thm}\cite{osthus2003densities}\label{thm: odd cycles}
    Given an odd integer $\ell> 1$, let $H = C_\ell$ and
    \[
    t_\ell = \left(\frac{\ell}{\ell - 1}\left(\frac{n}{2}\right)^\ell\log n\right)^{1/(\ell - 1)}.
    \]
    Then, for any $\eps > 0$, 
    \begin{align*}
    {\P[G _{n,m}^{H^-}\text{ is bipartite}]}\rightarrow \begin{cases}
        1 &\text{if}\quad m = o(n)\\
        0 &\text{if}\quad n/2 \leq m \leq (1 - \eps)t_\ell\\
        1 &\text{if}\quad m \geq (1 + \eps)t_\ell
    \end{cases}
    .
    \end{align*}
\end{thm}

\cref{thm: osthus 0-1} was also generalized when $H$ is a clique by Balogh, Morris, Samotij, and Warnke \cite{BaloghMorrisSamotijWarnke2016TypicalStructure}. 

\begin{thm}\label{thm: cliques}\cite{BaloghMorrisSamotijWarnke2016TypicalStructure}
     For each $r \geq 2$, let 
    \[
    \theta_r = \frac{r-1}{2r}\cdot\left[r\cdot\left(\frac{2r+2}{r+2}\right)^{1/(r-1)}\right]^{2/(r+2)}
    \]
    and 
    \[
    m_r = \theta_rn^{2-2/(r+1)}(\log n)^{\frac{1}{\binom{r+1}{2}-1}}.
    \]
    Fix $r \geq 3$ and let $G$ be a graph obtained by choosing a $K_{r+1}$-free graph with vertex set $[n]$ and $m$ edges uniformly at random. Then, there exists a function $d_r = d_r(n) = \Theta(n)$ such that for all $\eps > 0$
    \begin{align*}
        \P[G \text{ is $r$-colourable}] \rightarrow \begin{cases}
            1 &\text{if} \quad m \leq (1 - \eps)d_r\\
            0 &\text{if} \quad (1 + \eps)d_r \leq m \leq (1 - \eps)m_r\\
            1 &\text{if}\quad m \ge (1 + \eps)m_r
        \end{cases}
        .
    \end{align*}
\end{thm}
Observe that $m_r$ agrees with the larger threshold obtained in \cref{thm: osthus 0-1} when $r = 2$. The existence of $d_r$ follows quickly from the establishment of $d_r$ as a threshold for $G \sim G_{n,m}$ being $r$-colourable by Achlioptas and Friedgut in \cite{Achlioptas1999SharpThreshold}. Together, these results completely answer \cref{q:enumeration}(a) and \cref{q:structure} for sufficiently large $m$ when $H$ is an odd cycle or a clique.

Focusing on \cref{q:structure}, {\L}uczak \cite{luczak2000triangle} proved that for any non-bipartite $H$ and sufficiently large $m$ (depending on $H$), $G_{n,m}^{H^-}$ can be made $\chi(H)$-colourable by deleting few edges a.a.s, which is a more general version of \cref{thm: luc 13}.

In the regular setting, Kim, Sudakov, and Vu \cite{KimSudakovVu2007} extended the previously discussed results of \cite{Bollobas1980,Wormald1978,Wormald1981} to 
 show that if $H$ is any balanced graph, $X_H$ converges to a Poisson random variable in $G_{n,d}$ when $d$ is a constant. This resolves \cref{q:enumeration}(b) when $d$ is a constant for a wide range of graphs $H$.
\subsection{Open questions}

For simplicity we only considered regular degree sequences on an even number of vertices. Most parts of our proofs should extend to the odd $n$ case with minor adjustments. Also, if $n$ is odd, notice that $T_{n,d}$ will contain at least $d/2$ defect edges. Thus, we ask the following:

\begin{conj}
    Let $Q$ be the event that $T_{n,d}$ can be made bipartite by deleting $d/2$ edges. Then, for odd $n$ and any $\eps > 0$, 
    \begin{align*}
        \P[Q] \rightarrow \begin{cases}
            0 \quad \text{if}\quad 2 \leq d \leq (1 - \eps)\frac{\sqrt 3}{2}\sqrt{n\log n}\\
            1 \quad \text{if} \quad (1 + \eps)\frac{\sqrt 3}{2}\sqrt{n\log n} \leq d \leq \frac{n-1}{2}
        \end{cases}
        .
    \end{align*}
\end{conj}

Likewise, many parts of our proofs extend directly to more general degree sequences. We believe that the statement of Theorem~\ref{thm: 0-1} holds for degree sequences that are sufficiently close to the regular sequences. For a degree sequence $\mathbf d$ on $[n]$, let $T_{n,\mathbf d}$ be a triangle-free graph chosen uniformly at random from all graphs with vertex set $[n]$ and degree sequence $\mathbf d$.

\begin{conj} Suppose $\mu > 0$ is a sufficiently small constant. Let $\mathbf d$ be a degree sequence on $[n]$ with average degree $d$ where $|d_u - d| \leq  d^{1/2+\mu}$ for all $u \in [n]$ that admits a bipartite graph. Then, for any $\eps > 0$,
    \begin{align*}
        \P[T_{n,{\bf d}} \text{ is bipartite}] \rightarrow \begin{cases}
            0 &\text{if}\quad 2 \leq d \leq (1 - \eps)\frac{\sqrt 3}{2}\sqrt{n\log n}\\
            1 &\text{if} \quad(1 + \eps)\frac{\sqrt 3}{2}\sqrt{n\log n} \leq d \le (1/2-\eps)n
        \end{cases}
        .
    \end{align*}

\end{conj}

It is also interesting to obtain either structural or enumeration results for $H$-free $d$-regular graphs beyond triangles. In \cite{osthus2003densities}, a threshold for the bipartiteness of random graphs avoiding longer odd cycles is established. We conjecture that this holds for random regular graphs as well. Similar to $G^{H^-}_{n,m}$, let $G^{H^-}_{n,d}$ denote $G_{n,d}$ conditional on being $H$-free.

\begin{conj}
     For all odd integers $\ell > 1$, let $t_\ell$ be defined as in \cref{thm: odd cycles}. Then, for any $\eps > 0$,
    \begin{align*}
        \P[G^{H^-}_{n,d} \text{ is bipartite}]\rightarrow \begin{cases}
            0 &\text{if}\quad 2 \leq d \leq (1 - \eps)2t_\ell/n\\
            1 &\text{if} \quad (1 + \eps)2t_\ell/n \leq d \leq n/2
        \end{cases}
        .
    \end{align*}
\end{conj}

Likewise, \cref{thm: cliques} should hold in the regular setting.

\begin{conj}
    For each $r \geq 2$, let $\theta_r$ and $m_r$ be as described in \cref{thm: cliques}.
    Fix $r \geq 3$ and let $G\sim G_{n,d}^{K_{r+1}^-}$. Then, there exists a function $d_r = d_r(n) = \Theta(1)$ such that for all $\eps > 0$
    \begin{align*}
        \P[G \text{ is $r$-colourable}] \rightarrow \begin{cases}
            1 &\text{if} \quad d \leq (1 - \eps)d_r\\
            0 &\text{if} \quad (1 + \eps)d_r \leq d \leq (1 - \eps)2m_r/n\\
            1 &\text{if}\quad (1 + \eps)2m_r/n \leq d
        \end{cases}
        .
    \end{align*}
\end{conj}



In all of these questions and conjectures, we expect the regular case to behave very similarly to the general case. It is thus natural to expect some sandwiching result to encompass all of these smaller structural results. As we have shown in \cref{cor: our sandwich}, however, this is more nuanced than just sandwiching $G \sim T_{n,d}$ between two copies of $T_{n,m}$. In the sparse regime, there is no bipartite structure on which the difference between regular bipartite and arbitrary bipartite graphs can cause problems. Thus, the following conjecture seems natural.

\begin{conj}\label{conj: sparse sandwich}
    If $d = \omega(\log n)$ and $d = o(n^{1/2})$, then there exist $m_* = (1 - o(1))dn/2$, $m^* = (1 + o(1))dn/2$, and a coupling $(G_*, G, G^*)$ where $G_* \sim T_{n, m_*}$, $G^* \sim T_{n, m^*}$, and $G \sim T_{n, d}$ such that $\P[G_* \subseteq G \subseteq G^*] = 1 - o(1)$. 
\end{conj}

It is much harder to conjecture what happens in the intermediate regime, as this is very poorly understood even in $T_{n,m}$. When $T_{n,d}$ has very few defect edges, this structural difference between regular bipartite graphs and arbitrary bipartite graphs should still have an impact. 
In \cite{osthus2003densities}, it was shown that that a graph chosen uniformly at random from all bipartite graphs with vertex set $[n]$ and $m \geq n$ edges will have a bipartition $(A,B)$ where $\big ||A| - |B|\big | \leq 2n\sqrt {\log n}/\sqrt m$. Let $G$ be a $d$-regular graph that can be made bipartite by deleting exactly $k$ edges. Then, $[n]$ has a partition $(U,V)$ where $|E(G[U])| + |E(G[V])| = k$. In \cref{prop: ab}, we show that $|U|$ and $|V|$ differ by at most $\frac{k}{d}$. Hence, the latest we can hope to sandwich $T_{n,d}$ between copies of $T_{n,m}$ is when $\Omega(\sqrt{dn\log n})$ edges deleted from $T_{n,d}$ to obtain a bipartite graph with non-negligible probability. If $T_{n,d}$ behaves similarly to $T_{n,m}$, then according to the threshold for connected defect graphs obtained in \cite{JENSSEN2025110499}, $k$ will not be this large until $d \ll \frac{1}{2}\sqrt{n \log n}$ (they do not consider defect graphs this big, and so we can only say that $d$ is much smaller than the lowest value they consider). In fact, it is likely that $k$ must be even larger than this, as in this range the maximum cut of $T_{n,m}$ may be even more unbalanced. 

\begin{que}
    When can $T_{n,d}$ be sandwiched between copies of $T_{n,m}$? Is $d = n^{1/2}$ a coarse threshold for such a sandwiching being possible? 
\end{que}

It would also be interesting to obtain conditional or joint edge probability results similar to \cref{thm: bip edge probs} in the triangle-free setting. This would allow for more direct analysis of structures in random regular triangle-free graphs.

\section{Preliminaries}\label{sec: preleminaries}
In this paper, we consider all asymptotics as $n \rightarrow \infty$. Then, property $X$ holds a.a.s if $\P[X] \rightarrow 1$. 

\begin{definition}
    Fix $f: \mathbb R_{\geq 0} \rightarrow \mathbb R$ and $g: \mathbb R_{\geq 0} \rightarrow \mathbb R$.
    \begin{itemize}
        \item $f(n) = O(g(n))$ if there exists a positive number $M$ and $n_0$ such that for all $n \geq n_0$, 
        $$|f(n)| \leq M|g(n)|;$$
        \item $f(n) = \Omega(g(n))$ if $\lim_{n \rightarrow \infty}f(n) \geq 0$, $\lim_{n \rightarrow \infty} g(n) \geq 0$, and there exists a positive number $N$ and $n_0$ such that for all $n\geq n_0$, 
        $$f(n) \geq Ng(n);$$
        \item $f(n) = \Theta(g(n))$ if $f(n) = O(g(n))$ and $f(n) = \Omega(g(n))$;
        \item $f(n) = o(g(n))$ if 
        $$\lim_{n \rightarrow \infty}\frac{f(n)}{g(n)} = 0;$$
        \item $f(n) = \omega(g(n))$ if 
        $$\lim_{n \rightarrow \infty} \frac{f(n)}{g(n)} = +\infty;$$
        \item $f(n) \sim g(n)$ if 
        $$f(n) = (1 + o(1))g(n).$$
        We also say that $f(n)$ is asymptotic to $g(n)$ in this case.
    \end{itemize}
\end{definition}

We define binomial coefficients using gamma functions. This allows for us to bound large products of binomial coefficients by simpler products of binomial coefficients that may have non-integer inputs and easily use Stirling's approximation to evaluate these non-integral binomial coefficients. See \cref{sec: Appendix} for more details.

Our main proof technique is to count the number of $d$-regular triangle-free graphs and then compare this result to the number of $d$-regular bipartite graphs. As such, we need an asymptotic enumeration of the number of $d$-regular bipartite graphs. We also need an asymptotic approximation for the number of bipartite graphs with near regular degree sequence.
Greenhill, McKay, and Wang \cite{GREENHILL2006291} proved the sparse case (when $d = o(n^{1 + \eps})$) of the following theorem and Liebenau and Wormald \cite{liebenau2023asymptotic} proved the moderate case in Remark 1.2 of their paper. This bound also holds when $d = \Theta(n)$, but is not used in this paper.
\begin{thm}\label{thm: bip enum}\cite{GREENHILL2006291,liebenau2023asymptotic}
    Let $(A,B)$ be a partition of $[n]$. Let $\mathbf d$ be a degree sequence on $[n]$ with maximum degree at most $d$. Let $M = \sum_{u \in A}d_u$, $d_A= M/|A|$, and $d_B = M/|B|$. For sufficiently small $\mu_0 > 0$, if 
    \begin{itemize}
        \item $||A| - |B|| \leq n/2$;
        \item $M/(|A||B|) \leq \mu_0$;
        \item $|d_A - d_u| \leq \sqrt{d_A}$ for all $u \in A$;
        \item  $|d_B - d_v| \leq \sqrt{d_B}$ for all $v \in B$;
    \end{itemize}  
    then the number of bipartite graphs with bipartition $(A,B)$ and degree sequence $\mathbf d'$ is approximately
    \begin{equation}\label{eq: bip count}
    \Theta\left({|A||B| \choose M}^{-1}\prod_{u \in A}{|B| \choose d_u} \prod_{v \in B}{|A| \choose d_v}\right).
    \end{equation}
    In particular, when $\mathbf d = (d, \ldots, d)$, the number of such graphs is asymptotic to 
    \begin{equation}\label{eq: reg bip}
        \binom{n^2/4}{dn/2}^{-1}\binom{n/2}{d}^ne^{-1/2}.
    \end{equation}
\end{thm}

Note that~\eqref{eq: reg bip} predates these works. In fact, a generalization where $\mathbf d|_A$ and $\mathbf d|_B$ are each regular was shown to hold for a wide range of settings by McKay and Wang \cite{McKayWang2003} and by Canfield and McKay \cite{CanfieldMcKay2005}.
Let
    $\mathcal G_{n,d}^\text{bip}$ denote the set of $d$-regular bipartite graphs on $n$ vertices.
There are $\binom{n}{n/2}$ ways to obtain an equal-sized bipartition of $[n]$.
Hence, immediately

    \begin{equation}
    |\mathcal G_{n,d}^{\text{bip}}| \sim  \frac{1}{2}\binom{n}{n/2}\binom{n/2}{d}^n\binom{n^2/4}{dn/2}^{-1}\exp[-1/2].\label{eq:reg bip count2}
    \end{equation}

We also need to bound the probability that a large bipartite graph chosen uniformly at random from all bipartite graphs with a given degree sequence avoids every edge in some large and well-behaved set of forbidden edges. To do so, we use a switching argument. For this switching argument, we need an approximation for the probability that some given edge appears in the randomly chosen graph conditioned on some fixed set of edges being included in the graph and some other fixed set of edges not being included in the graph. 
For any graph $H$, let $\mathbf{d}^H$ denote the degree sequence on $H$. Larkin, McKay, and Tian  proved the following theorem approximating edge probability in random graphs with a given degree sequence in~Lemma 5.3 of \cite{larkin2025subgraphs}.

\begin{thm}\label{thm: bip edge probs} \cite{larkin2025subgraphs}
Let $(A,B)$ be a partition of $[n]$ where $|A|, |B| = (n/2)(1 + o(1))$. Let $\mathbf d$ be a degree sequence on $[n]$ with maximum degree at most $d$. Let $G$ be a graph chosen uniformly at random from all bipartite graphs with bipartition $(A,B)$ and degree sequence $\mathbf d$.
Let $H_1$ and $H_2$ be two disjoint bipartite graphs on $[n]$ where $\mathbf{d}^{H_1} \leq \mathbf d$. Take $uv \in (A \times B)\setminus(H_1 \cup H_2)$. For sufficiently small $\mu_0 > 0$, if
\begin{itemize}
    \item $d^{H_1}_u \leq \mu_0d$;
    \item $d^{H_1}_v \leq \mu_0d$; 
    \item $d^{H_2}_u \leq \mu_0n$;
    \item $d^{H_2}_v \leq \mu_0n$;
    \item $|E(G)| = (1 - o(1))\frac{dn}{2}$;
    \item $|E(H_1)| = o(dn)$;
    \item $d < \mu_0n$;
\end{itemize}
then 
    \[
    \mathbb{P}[uv \in G \mathrel | E(H_1) \subseteq E(G), E(H_2)\cap E(G) = \emptyset] < \frac{2d}{n}(1 + O(\mu_0)).
    \]
\end{thm}

Finally, our main arguments rely on $G \sim T_{n,d}$ being fairly well structured. \cref{thm: luc 13} can be converted to the following theorem in the regular setting, which gives us our desired structure. We discuss this in more detail in \cref{sec: Luczak}. For any $0 \leq \delta \leq 1$, a graph $G$ is $\delta$-bipartite if it is $(\delta, 2)$-partite.

\begin{thm}\label{thm: Luckzak}
    Let $G \sim T_{n,d}$. For every $\delta > 0$, there exists a constant $c = c(\delta) > 0$ and $C > 0$ such that for all $c\sqrt n \leq d \leq Cn$,
    \[
    \mathbb P[G \text{ is $\delta$-bipartite}] \rightarrow 1.
    \]
\end{thm}

\subsection{Translating \cref{thm: luc 13} to \cref{thm: Luckzak}}\label{sec: Luczak}

Here, we discuss the proof of \cref{thm: luc 13} in more detail, and how we can prove the corresponding regular result in a very similar manner. We state \cref{thm: Luckzak} when the forbidden graph is a triangle, as that is all that we need, but this proof holds for all forbidden graphs with chromatic number $3$. We first need to estimate the probability that a graph $G \sim G_{n,m}$ is regular. The following enumeration of $d$-regular graphs by Liebenau and Wormald in \cite{LiebenauWormald2024} is much stronger than what we need.
\begin{thm}\label{eq: num d-reg grahs}\cite{LiebenauWormald2024}
    For all $1 \leq d \leq n-2$, the number of $d$-regular graphs with vertex set $[n]$ is asymptotic to
    \[
    \frac{\binom{n-1}{d}^{n}\binom{ \binom{n}{2}}{dn/2}}{\binom{n(n-1)}{dn}}\cdot e^{1/4}.
    \]
\end{thm}

Originally, \cref{thm: luc 13} was proved whenever the forbidden graph $H$ satisfied a conjecture of Kohayakawa, {\L}uczak, and R\"odl that was connected to the study of randomized extremal graph concepts, and particularly the study of Tur\'an's problem in random graphs \cite{kohayakawa1997k} (this is not the randomized Tur\'an statement we discussed earlier). This conjecture has since been proved for all graphs using the method of hypergraph containers by Balogh, Morris, and Samotij in \cite{BaloghMorrisSamotij2015}. 

Notation in this subsection is taken directly from \cite{luczak2000triangle}. Let $G = (V,E)$ be a graph and $\varepsilon$ be a positive constant. 
The following claim is a simple calculation and allows us to translate very strong probabilistic results in $G_{n,m}$ to results in $G_{n,d}$.

\begin{claim}\label{lem: conversion}
    Let $m = dn/2$, $G \sim G_{n,m}$, and $G' \sim G_{n,d}$. Suppose $0 < \alpha < 1$ and $P$ is a graph property where $\mathbb P[G \text{ satisfies } P] \leq \alpha^{m}$. Then, if  $n$ is sufficiently large, 
    \[
    \mathbb P[G' \text{ satisfies $A$}] \leq (\alpha + O(\log d / d))^{dn/2}.
    \]
\end{claim}

\begin{proofclaim}
    By \cref{eq: num d-reg grahs}, and using \cref{cor: choose} in the second line below, we have 
    \begin{align*}
        \mathbb P[G' \text{ satisfies $A$}] &< \frac{\alpha^{dn/2}{n^2/2 \choose dn/2}\binom{n^2-n}{dn}}{e^{1/4}{n-1 \choose d}^n\binom{n^2/2 - n/2}{dn/2}}\\
        &= \alpha^{dn/2}d^{O(n)}\\
        &= (\alpha + O(\log d/d))^{dn/2},
    \end{align*}
    as desired.
\end{proofclaim}

It is first shown that some `bad' structures in $G$ occur with probability much smaller than the probability that $G$ is $H$-free, and so are negligible. To do this, they first note that the probability that $G$ is $H$-free is at least the probability that $G$ is $(h-1)$-partite, which is 
\[
\left[\frac{h-2}{h-1} + O\left( \frac{m}{n^2} \right)\right]^m.
\]

By \cref{thm: bip enum} and \cref{eq: num d-reg grahs}, we have that the ratio of the number of $d$-regular bipartite graphs to the number of $d$-regular graphs, i.e. the probability that $G_{n,d}$ is bipartite, is 
\[
\left[\frac{h-2}{h-1} + O\left( \frac{d}{n} \right)\right]^{dn/2}.
\]

The `bad' structures correspond to $G$ not satisfying a low density version of Szemer\'edi's regularity lemma (as described in \cite{kohayakawa1997szemeredi}) or having some anomalous distribution of edges among the $\varepsilon$-regular pairs in the partition described by the regularity lemma. By \cref{lem: conversion}, these `bad' structures will still be negligible in the regular setting. More precisely, for the same choice of parameters as in the proof of Theorem 13 in \cite{luczak2000triangle}, the structures described in Corollary 9, Lemma 10, and Lemma 11 of \cite{luczak2000triangle} will still be negligible in the regular setting. 

We next consider the auxiliary graph $\mathscr{G}$ whose vertices are the parts of the regular partition of $G$ and whose edges correspond to $\varepsilon$-regular pairs that have a moderate number of edges between them. By neglecting the `bad' structures, it follows that $\mathscr{G}$ satisfies the conditions of the deterministic Corollary 15 in \cite{luczak2000triangle}. Hence, $\mathscr{G}$ will either contain many edge-disjoint copies of $K_h$ or will be $(\delta', h-1)$-partite. It is shown that the probability that $G$ is such that $\mathscr{G}$ has many edge-disjoint copies of $K_h$ is at most $3^{-M}$. It follows, by \cref{lem: conversion}, we will have that the corresponding auxiliary graph $\mathscr{G}'$ of $G'$ will have many edge-disjoint copies of $K_h$ with negligible probability. Hence, $\mathscr{G}'$ will be $(\delta', h-1)$-partite, so $G'$ can be made $(\delta', h-1)$-partite by deleting few edges. That is, $G'$ is $(\delta, h-1)$-partite for some $\delta$ slightly larger than $\delta'$.

\section{Proof of $0$-statement of \cref{thm: 0-1}}\label{sec: zero}
Observe that by Theorem 1 of \cite{McKayWormaldWysocka2004}, $T_{n,d}$ contains $C_5$ with high probability when $n = o(n^{1/5})$. Hence, it suffices to prove this for $d \geq n^{1/5}/\log n$. Let $\mathcal T_1$ be the set of all $d$-regular triangle-free graphs that can be made bipartite by deleting two edges (one from each side of a maximum cut). We prove this by showing that $|\mathcal T_1|$ is much greater than the number of $d$-regular bipartite graphs.

Let $(A,B)$ be a partition of $[n]$ with $|A| = |B| = n/2$ and let $\mathcal G_0$ denote the set of $d$-regular bipartite graphs with bipartition $(A,B)$ and $\mathcal G_1$ denote the set of $d$-regular graphs with vertex set $[n]$ where $|E(G[A])| = |E(G[B])| = 1$. Let $N_1 = |\mathcal G_1|$ and $N_0 = |\mathcal G_0|$. We obtain a lower bound for $N_1$ by first choosing the unique edges $wx \in G[A]$ and $yz \in G[B]$, then exposing the remaining neighbours of $w,x,y,z$, and then counting the ways to choose the remaining edges using \cref{thm: bip enum}. When exposing the neighbours of $w,x,y,z$, we may assume that $E(G[\{w,x,y,z\}]) = \{wx, yz\}$ since we are establishing a lower bound for $N_1$. Let $\mathbf {d'}$ denote the residual degree sequence on $(A \setminus \{w,x\})\cup (B \setminus \{y,z\})$ after exposing the neighbours of $w,x,y,z$. Let $d'_A$ be the average degree of $d'$ in $A \setminus\{w,x\}$ and $d'_B$ be the average degree of $d'$ in $B \setminus \{y,z\}$. Note that we are able to apply \cref{thm: bip enum} for graphs on the residual degrees since $d'_{A} \geq1$ and $|d'_{A} - d'_v| \leq 1 \leq \sqrt{d'_{A}}$ for all $v \in A \setminus \{w,x\}$ and likewise for $B$. 
Hence, by \cref{thm: bip enum},
\begin{equation}\label{eq: N_1}
    N_1 = \Omega\left({n^4}{\frac{n}{2} - 2 \choose d-1}^2{\frac{n}{2} - d - 1 \choose d - 1}^2{\frac{n}{2} - 2 \choose d}^{n - 4d}{\frac{n}{2} - 2 \choose d-1}^{4d -4}{\frac{n^2}{4} - 2n + 4 \choose \frac{dn}{2} - 4d + 2}^{-1}\right),
\end{equation}
On the other hand,~\eqref{eq: reg bip}  gives that 
\begin{equation}\label{eq: N_0}
    N_0 = (1 + o(1)){\frac{n}{2} \choose d}^n{\frac{n^2}{4} \choose \frac{dn}{2}}^{-1}.
\end{equation}

\begin{claim}\label{clm: 0-statement}
    \[
    \frac{N_1}{N_0} = \omega(1)
    \]
\end{claim}

Let $\mathbf d$ be a vertex set on $[n-4]$ where  $n^{1/5}/\log n \leq d \leq (1 - \eps)\frac{\sqrt 3}{2}\sqrt{n \log n}$, all but $4d-4$ vertices have degree $d$, and the remaining vertices have degree $d - 1$. The next claim allows us to show that our method of counting graphs in $\mathcal T_1$ using $N_1$ is accurate.

\begin{claim}\label{clm: connectivity}
    Let $\mathcal S$ be the set of bipartite graphs with vertex set $[n-4]$ and degree sequence $\mathbf d$ that have a balanced bipartition. Then, almost every graph in $\mathcal S$ is $3$-edge-connected.
\end{claim}

In fact, it is expected that the edge connectivity of almost all graphs in $\mathcal S$ is $d - 1$.

 To complete the proof of the $0$-statement, notice that \cref{clm: connectivity} implies that almost every graph in $\mathcal T_1$ has a unique maximum cut with $dn/2 - 2$ edges. To see this, let $G \in \mathcal T_1$ have maximum cuts $S_1$ and $S_2$ induced by $A_1,A_2 \subseteq \binom{[n]}{n/2}$ respectively. Let $e,f$ be the defect edges with respect to $S_2$ and consider the bipartite graph with edge set $S' =S_1 \setminus \{e,f\}$. If $S_1$ and $S_2$ are distinct, then $S'$ must have two valid bipartitions, and so is not connected. But $S_1 \subseteq \mathcal S$ (where $\mathcal S$ is as described in \cref{clm: connectivity}) and almost all graphs in $\mathcal S$ are $3$-edge-connected for any choice of $(A,B)$, so $S'$ will be connected for almost all graphs $G \subseteq T_1$.

 Thus,
 \begin{align*}
     \frac{|\mathcal T_1|}{|\mathcal G_{n,d}^\text{bip}|} = \Omega\left(\frac{\sum_{(A,B)}N_1}{\sum_{(A,B)}N_0}\right) = \omega(1),
 \end{align*}
 where the final equality follows from \cref{clm: 0-statement}. \cref{clm: 0-statement} and \cref{clm: connectivity} are proved below in this section. 
\qed

\medskip
\noindent\textit{Proof of \cref{clm: 0-statement}.}
By first applying \cref{prop: choose top} and then applying \cref{prop: choose bot} to all but the final binomial coefficient in \eqref{eq: N_1}, we have 

{\allowdisplaybreaks \begin{align*}
    &\binom{n/2 - 2}{d - 1}^2 \geq \binom{n/2}{d}^2\left(2d/n\right)^{2}\exp[o(d^2/n)];\\
    &\binom{n/2 - d - 1}{d - 1}^2 \geq \binom{n/2}{d}^2(2d/n)^2\exp[-4d^2/n + o(d^2/n)];\\
    &\binom{n/2 - 2}{d}^{n - 4d} \geq \binom{n/2}{d}^{n-4d}\exp[-4d - 4d^2/n + 16d^2/n  + o(d^2/n)];\\
    &\binom{n/2 - 2}{d - 1}^{4d-4} \geq \binom{n/2}{d}^{4d-4}(2d/n)^{4d-4}\exp[-16d^2/n + 8d^2/n + o(d^2/n)].
\end{align*}}
Substituting these into~\eqref{eq: N_1} gives

\begin{equation}\notag
    N_1 \geq n^4 \left(\frac{2d}{n}\right)^{4d} \binom{\frac{n}{2}}{d}^n \binom{\frac{n^2}{4} - 2n + 4}{\frac{dn}{2} - 4d + 2}^{-1}\exp\left[-4d + o\left(d^2/{n}\right) + O(1)\right].
\end{equation}
Combining this with \eqref{eq: N_0}, we have 
\begin{equation*}
    \frac{N_1}{N_0} \geq n^4\left(\frac{2d}{n}\right)^{4d}\frac{\binom{n^2/4}{dn/2}}{\binom{n^2/4 - 2n + 4}{dn/2 - 4d + 2}}\exp[-4d + o(d^2/n) + O(1)].
\end{equation*}
We then apply \cref{cor: choose} to both binomial coefficients in $\frac{\binom{n^2/4}{dn/2}}{\binom{n^2/4 - 2n + 4}{dn/2 - 4d + 2}}$ to obtain
\begin{align}
    \frac{\binom{n^2/4}{dn/2}}{\binom{n^2/4 - 2n + 4}{dn/2 - 4d + 2}} &= \frac{(n/2)^{dn}(dn/2)^{dn/2 - 4d + 5/2}(1 - 2d/n)^{n^2/4 - dn/2 -2n + 4d}}{(dn/2)^{dn/2 + 1/2}(1 - 2d/n)^{n^2/4 - dn/2}(n/2)^{dn-8d + 4}}\exp[O(1)]\notag\\
    &=(n/2)^{4d-2}d^{-4d + 2}\exp[4d - 4d^2/n + o(d^2/n) + O(1)].\notag
\end{align}
That is, 

\begin{align}
    \frac{N_1}{N_0} \geq d^2n^2\exp\left[ -\frac{4d^2}{n}(1 + o(1)) + O(1)\right]. \label{eq: N_1/N_0}
\end{align}
It is easy to check that the second derivative of the logarithm of the right hand side of~\eqref{eq: N_1/N_0} with respect to $d$ is negative, and so the right hand side of~\eqref{eq: N_1/N_0} is minimized either when $d = 2$ or $d = (1 - \eps)\frac{\sqrt 3}{2}\sqrt{n\log n}$. If $d = 2$, it is easy to see that the right hand side of~\eqref{eq: N_1/N_0} is $\omega(1)$. If $d = (1 - \varepsilon)\frac{\sqrt 3}{2}\sqrt{n\log n}$, then we have
\[
\frac{N_1}{N_0} \geq n^3\log n \exp[-3\log n + O(1)] = \omega(1).
\]
\qed

\noindent\textit{Proof of \cref{clm: connectivity}.}
    We prove this via a simple counting argument. Let $\mathcal S_3$ be the set of graphs in $\mathcal S$ that are not $3$-edge-connected. Then, it suffices to show that $\frac{|\mathcal S_3|}{|\mathcal S|} = o(1)$. To do so, we bound $|\mathcal S_3|$ by summing over all possible cuts of size at most $3$ and using \cref{thm: bip enum} to bound the number of graphs on either side of the cut. We use \cref{thm: bip enum} to approximate the number of graphs in $\mathcal S$ as well. Notice that any cut of size $3$ in a graph in $\mathcal S$ must have at least $d$ vertices on either side of the cut. Let $M_\ell$ denote the maximum number of bipartite graphs with vertex set $[\ell]$ and degree sequence $\mathbf d'$ over all degree sequences $\mathbf d'$ that can be obtained from a $d$-regular sequence by first lowering the degree of up to $2d-2$ distinct vertices by $1$ and then lowering the degree of any $3$ vertices by $1$. Then, by \cref{thm: bip enum}, it is clear that if $\ell = \omega(1)$, $M_\ell \leq 3\binom{\ell}{\ell/2}\binom{\ell/2 + 2}{d}^\ell/\binom{\ell^2/4 - 4}{d\ell/2}$ (by first observing that each side of the cut of a graph in $\mathcal S_3$ will be bipartite and have bipartition $(U,V)$ where $||U| - |V|| \leq 4$ and also noting that \eqref{eq: bip count} increases when degrees are increased) and $\mathcal S = \Theta\left(\binom{n-2}{(n-2)/2}\binom{(n-4)/2}{d}^{n - 4d + 4}\binom{(n-4)/2}{d-1}^{4d-4}/\binom{(n-4)^2/4}{dn/2 - 4d + 2}\right)$. 
    \begin{align}
        \frac{|\mathcal S_3|}{|\mathcal S|} &\leq \frac{\sum_{n' = d}^{(n-2)/2}\binom{n-2}{n'}M_{n'}M_{n-n'-2}n'^3(n-n'-2)^3}{|\mathcal S|}\notag \\
        &\leq \sum_{n' = d}^{(n-2)/2}O\left(\binom{n-2}{n'}n^6\frac{\binom{n'}{n'/2}\binom{n'/2 + 2}{d}^{n'}\binom{n - n' - 2}{(n - n'-2)/2}\binom{(n - n' - 2)/2 + 2}{d}^{n - n' - 2}\binom{(n-4)^2/4}{dn/2 - 4d + 2}}{\binom{n-2}{(n-2)/2}\binom{(n-2)/2}{d}^{n-2}\binom{n'^2/4 - 4}{dn'/2}\binom{(n - n' - 2)^2/4 - 4}{d(n - n'-2)/2}}\right) \label{eq: connectivity bad}.
    \end{align}

    We omit the details of simplifying~\eqref{eq: connectivity bad}, but note that we use \cref{cor: choose}, \cref{prop: choose approx ub top}, and \cref{prop: choose approx ub bot}. For a complete description of similar calculations, see \cref{sec: Janson simplifications} or \cref{sec: small torso max}. This gives that the right hand side of~\eqref{eq: connectivity bad} is at most 

    \begin{align}
        &\sum_{n' = d}^{(n-2)/2}\binom{n}{n'}n^6\frac{(n'/2)^{dn'/2}(n/2 - n'/2)^{d(n - n')/2}}{(n/2)^{dn/2}}\exp\left[O(d^2)\right]\notag\\
        &\leq n^{d + 7}\exp[-\Theta(d^2\log n)] = o(1) \label{eq: maximization}
    \end{align}
    where~\eqref{eq: maximization} holds by maximizing the interior of the above sum with respect to $n'$. Again, see \cref{sec: Janson simplifications} for similar maximization arguments.
\qed

\section{Proof of the 1-statement of \cref{thm: 0-1}}\label{sec: one}

The remainder of this paper focuses on proving the 1-statement of \cref{thm: 0-1}. We split the proof into two ranges of $d$. The first range consists of $d$ with values at least $n/\log n$. In this case, our theorem below follows from some results in \cite{PROMEL1996311}. 

\begin{thm}\label{thm: linear d}
    Let $G \sim T_{n,d}$. Then, when $\frac{n}{\log n} \leq d \leq n/2$, 
    \[
    \P[G \text{ is bipartite}] \rightarrow 1.
    \]
\end{thm}
\subsection{Proof of~\cref{thm: linear d}}

\begin{proof}
    Let $\mathcal G_{n,m}^{\text{bip}}$ denote the number of bipartite graph with vertex set $[n]$ and $m$ edges. We prove this by applying the results of \cite{PROMEL1996311}. Namely, in Lemma 2.3 of \cite{PROMEL1996311}, it was shown that every triangle-free graph that is not bipartite is contained in one of four sets, called $\mathcal A_{n,m}$, $\mathcal B_{n,m}$, $\mathcal C_{n,m}$, and $\mathcal D_{n,m}$. It immediately follows that each $d$-regular triangle-free graph that is not bipartite is also contained in one of these sets. The authors of \cite{PROMEL1996311} then bounded the size of each of these sets. We will show that when $m$ is sufficiently large (which corresponds to $d\geq n/\log n$), their results imply our desired result. In fact, $\mathcal A_{n,m}$ is the set of graphs on vertex set $[n]$ with $m$ edges and maximum degree at most $\sqrt n$. When $m = dn/2$, this is much smaller than $d$, so no regular graph will be contained in $\mathcal A_{n,m}$.
    
    By Theorem 2.2 in \cite{PROMEL1996311}, we have 
    \begin{equation}\label{eq: num bip graphs}
        |\mathcal G_{n,m}^\text{bip}| = 2^{n - o(n)}\binom{n^2/4}{m}.
    \end{equation}
    Thus, by \eqref{eq:reg bip count2}, when $m = dn/2$ we have 
    \begin{equation}
        \frac{|\mathcal G_{n,d}^\text{bip}|}{|\mathcal G_{n,m}^\text{bip}|} \geq \exp[-Cn\log n]
    \end{equation}
    for some constant $C > 0$. Thus, it suffices to show that each of $\mathcal B_{n,m},\mathcal C_{n,m},\mathcal D_{n,m}$ contain at most 
    \begin{equation}\label{max set}
        \exp[-Cn\log n + O(n)]\binom{n^2/4}{dn/2}
    \end{equation}
    graphs. In the proof of Theorem 2.13 in \cite{PROMEL1996311}, it is shown that 
    \begin{align}
        &|\mathcal B_{n,m}| \leq \binom{n^2/4}{dn/2}\exp\left[-dn^{1/4} + O(n)\right],\label{B_n,m}\\
        &|\mathcal C_{n,m}| \leq \binom{n^2/4}{dn/2} \exp\left[-n^{\frac{3n}{4d}} + O(n)\right],\label{C_n,m}\\
        &|\mathcal D_{n,m}| \leq \binom{n^2/4}{dn/2}\exp\left[ -\frac{d^2n}{5} + O(n) \right].\label{D_n,m}
    \end{align}
    Then,~\eqref{B_n,m}and~\eqref{D_n,m} are smaller than~\eqref{max set} because $d \geq n/\log n$ and~\eqref{C_n,m} is smaller than~\eqref{max set} because $d \leq n/2$. This completes the proof.
    
\end{proof}

Next we restrict to $d = o(n)$. We follow the same strategy in \cref{thm: osthus 0-1}, but the restriction to the space of $d$-regular graphs requires novel treatments, as many probabilistic tools used in \cref{thm: osthus 0-1} do not apply directly to $G_{n,d}$.  

We first give a brief overview of the proof~\cite{osthus2003densities} of the 1-statement~\eqref{eq: second 1 statement} of \cref{thm: osthus 0-1}, which concerns $G_{n,m}$.
 Then, we introduce the new counting methods  when treating $G_{n,d}$, in a high level, in Section~\ref{sec: overview}. Afterwards we give an overview of the proof of the 1-statement of Theorem~\ref{thm: 0-1} in Sections~\ref{sub: Torso} -- \ref{sec: completing the proof}.    

\subsection{The old and the new new}\label{sec: overview}

Observe that by \cref{thm: luc 13} we may assume that $k \le \delta m$ for some small constant $\delta > 0$.

For any $0 \leq k \leq m$, we say a graph with $m$ edges is $k$-bipartite if $k$ is the minimum number of edges that must be deleted to obtain a bipartite graph. Let $\mathcal T_{k,n,m}$ denote the set of triangle-free $k$-bipartite graphs on $[n]$ with $m$ edges. Let $G \sim G_{n,m}$. Then, it suffices to show that
\[
\sum_{k = 0}^{m}\P[G \in \mathcal T_{k,n,m}] = (1 + o(1))\P[G \text{ is bipartite}].
\]
Since every graph is $k$-bipartite for some $0 \leq k \leq m$, observe that 
\[
\P[G \in \mathcal T_{k, n,m}] = \P[G \in \mathcal T_{k,n,m} \mathrel | G \text{ is $k$-bipartite}]\cdot\P[G \text{ is $k$-bipartite}].
\]
When $k$ is small, it was shown that $G$ being $k$-bipartite implies that $G$ has a close-to-balanced partition $(A,B)$ (so that $G[A]\cup G[B]$ has $k$ edges). If $G[A]$ and $G[B]$ are both triangle-free, then $G$ is triangle-free if and only if the cross-edges (between $A$ and $B$) in $G$ avoid creating triangles with edges in $G[A]$ and $G[B]$. By further conditioning on $(A,B)$, $G[A]$, and $G[B]$ and letting $H$ denote a random bipartite graph on $(A,B)$ with $m-k$ edges chosen uniformly at random, the problem can be reduced to bounding the probability $H$ avoids creating triangles with $G[A]\cup G[B]$. We will commonly refer to $G[A]$ and $G[B]$ as defect graphs. Let $X$ be the number of triangles in $H\cup G[A]\cup G[B]$. 
The authors of \cite{osthus2003densities} estimate the probability of $X=0$ by Janson's inequality~\cite{janson1988exponential} when the variance of $X$ is sufficiently small.
The variance of $X$ is determined by the structure of $G[A]$ and $G[B]$ (which determines the dependency graph; we refer the readers to~\cref{sec: janson args} for detailed discussions). When $G[A]$ and $G[B]$ are closer to being regular, then $X$ has a smaller variance. The authors of \cite{osthus2003densities} introduced graph structures called `torsos' (see \cref{sub: Torso} for their definitions) to characterise the regularity of the defect graphs. More specifically, the size of the torso of a graph meausres the regularity of that graph. Then they count the defect graphs conditioning on their torsos and use Janson's inequality to bound $\P(X=0)$ under the same conditioning. When the torsos are sufficiently large, the bound from  Janson's inequality turned out to be sufficient.

When the torso is small, Janson's inequality is not sufficient. In this case, $G[A]$ (or $G[B]$) has many vertices with very high degree. Let $x$ be such a high degree vertex, $U(x)$ be the cross-neighbourhood of $x$, and $V(x)$ be the cross-neighbourhood of $U(x)$. It was shown that with very high probability, there are many high defect degree vertices $x$ where $|V(x)|$ is very large --- $x$ is then expanding if $|V(x)|$ is very large. But then $G$ will contain a triangle unless $N_{G[A]}(x) \cap V(x) = \emptyset$ for every expanding vertex $x$, which the authors of \cite{osthus2003densities} showed to be very unlikely. Their argument to prove that there are many expanding vertices is quite involved, and relies on first establishing nice intersecting properties of $U(x)$ for high degree vertices $x\in A$.

If we naively translate these result to random $d$-regular graphs by conditioning on the degree sequence of $G_{n,m}$ being $d$-regular then it creates an error $\exp(O(n\log n))$ which is too large when the number of edges in $G[A]\cup G[B]$ is $O(n\log n)$. However, both Janson's inequality and their treatment for small torsos --- expanding $x$ and intersecting properties of $U(x)$ --- rely heavily on the model $G_{n,m}$, and the analogous tools do not exist in the random $d$-regular graphs.

{\bf Our contribution}: instead of exposing all the cross-edges of $G$ in a single round, we expose them in multiple rounds. We carefully restrict the set of edges that may be exposed in each round, so that Janson's inequality as well as the expanding properties (in the small torso case) can be applied ``locally'' in some round, whereas in some other round, the degree constraint is taken care of. We describe below a template of our counting argument for the large torso case. Note that the full proof needs to consider various cases according to the size of the torso and the size of $G[A]$ and $G[B]$ so the template will be adjusted slightly in each scenario (see \cref{sec: Janson} for a more complete description of this argument).

Suppose $E_X$ is a subset of edges in $G[X]$ for $X\in \{A,B\}$ that are pre-chosen (depending on the cases). Suppose $S_A$ and $S_B$ are subsets of vertices in $A$ and $B$; typically $S_A$ and $S_B$ are the vertices incident to the edges in $E_A$ and $E_B$. To bound the number of $d$-regular triangle-free
graphs with defect graph $G[A]$ and $G[B]$ given, we construct a superset of graphs ${\mathcal H}$ 
as follows.
\begin{enumerate}
    \item (local $(S_A,S_B)$-sprinkling) Add (a pre-set number) $m_1$ edges joining $S_A$ and $S_B$ such that no edge in $E_A$ or $E_B$ is contained in a triangle.
    \item (local $(S_A,\bar S_B)$-sprinkling) Add (a pre-set number) $m_A$ edges joining $S_A$ and $B\setminus S_B$ such that no edge in $E_A$ is contained in a triangle.
    \item (saturating $S_A$) Add edges so that all the vertices in $S_A$ have degree equal to $d$.
    \item Repeat steps 2 and 3 swapping the role of A and B.\
    \item (degree completion) Add edges so that every vertex has degree equal to $d$.
\end{enumerate}

There are some nuances that we skipped in the template above. For instance, in each of steps 1--4, choices need to be made so that no vertex is incident to more than $d$ edges. Also, we consider any values of $m_1$, $m_A$, and $m_B$ that can result in a $d$-regular graph. As mentioned before, small adjustment is required to the template in different cases in the proof, as well as the choices for parameters like $m_1,m_A,m_B$. 
We bound $|{\mathcal H}|$ by bounding the number of choices in each step; sometimes the number of choices in a step depends on the choices of previous steps.
The important observation we want to point out here is that we can count the choices in steps 1 and 2 by considering random bipartite graphs on $(S_A,S_B)$ (or $(S_A,\bar S_B)$) with $m_1$ (or $m_A$) edges, and we can apply Janson's inequality in these steps. For step 5, we apply the enumeration result for bipartite graphs with fixed degree sequences found in \cref{thm: bip enum}. See \cref{sec: Janson Counting 1} and \cref{sec: Janson counting 2} for complete descriptions of these arguments.

The small torso case uses a different template. Without loss of generality, assume that $G[A]$ has a small torso. Let $E_A$ and $S_A$ be pre-set subsets of edges and vertices as before. Instead of conditioning on $G[A]$, we expose only the degree sequence ${\bf d}^*$ of $G[A]$ on $A$. However, $G[B]$ is fully exposed as before. To count such $d$-regular triangle-free graphs we construct them following the template below.
\begin{enumerate}
    \item (choosing cross-edges for $S_A$) Pick $d-d^*_u$ neighbours for every $u\in S_A$ from $B$.
    \item (degree completion) Add cross-edges so that all vertices in $B$ has degree equal to $d$, for all $u\in A$, its degree is equal to $d-d^*_u$, and no triangle is created.
    \item (choosing $G[A]$) Choose a defect graph $G[A]$ conditioning on its degree sequence ${\bf d}^*$ such that $G$ is triangle-free.
\end{enumerate}

A detailed analysis of this process is given in \cref{sec: small torso proof 1}. As in their proof, we show that there are many vertices in $S_A$ that expand.  Due to the degree constraint in step 2, we prove this using the switching method (see \cref{sec: subswitching} for more details). This argument applies when the number of edges in $G[A]$ is sufficiently large. When $G[A]$ has fewer edges (but still has a small torso), we use a different but simpler counting argument (see \cref{sec: small torso simple}).



\subsection{Torso structure}\label{sub: Torso}
Here, we formally define torsos and describe how we use them to partition the set of graphs that need to be considered to prove the 1 statement of \cref{thm: 0-1}. Let $\mathcal T_{\delta_0,n,d}$ denote the set of $d$-regular triangle-free $(\delta_0,2)$-partite graphs with vertex set $[n]$.
By \cref{thm: Luckzak}, it suffices to show that $|\mathcal T_{\delta_0, n,d}| = |\mathcal G_{n,d}^\text{bip}|(1 + o(1))$  for some small constant $\delta_0$ when $(1 + \varepsilon)\frac{\sqrt 3}{2}\sqrt{n \log n} \leq d \ll n$. 


Let $(A,B)$ be a partition of $[n]$ and $\mathcal G_{A, B, k_A, k_B, d}$ denote the collection of $d$-regular graphs with vertex set $[n]$ where $|E(G[A])| = k_A$ and $|E(G[B])| = k_B$. Our arguments rely on further partitioning $\mathcal G_{A,B,k_A,k_B,d}$ based on the torsos of the defect graphs, as described in \cite{osthus2003densities}. Take $0 < \varepsilon < 10^{-6}$ to be a small constant. Also, assume that $k_A + k_B \leq \delta_0 dn/2$ where $\log\log 1/\delta_0 = \varepsilon ^{-5}$. For all non-negative integers $r$, let 

\begin{equation}
k_{A,r, \varepsilon} := (1- \varepsilon)^rk_A,\label{def:kAr}
\end{equation}
let
\begin{equation}\label{eq: D}
    D_{A, \varepsilon} := \frac{\varepsilon^4d}{2\log\left(\frac{dn}{2k_A}\right)},
\end{equation}
and let $\overbar{r}_{A, \varepsilon}$ be such that
\begin{equation}\label{eq: r1}
    (1-\varepsilon)^{\overbar{r}_{A,\varepsilon}} = \frac{4\log\log\left( \frac{dn}{2k_A} \right)}{\varepsilon^2\log\left( \frac{dn}{2k_A} \right)}.
\end{equation}
It follows that
\begin{equation}\label{eq: r asymptotics}
    \overbar{r}_{A, \varepsilon} = \Theta\left(\log\log\frac{dn}{2k_A}\right).
\end{equation}

Note that the error incurred in~\eqref{eq: r1} by treating $\bar r_A$ as an integer is $(1-\varepsilon)^\beta$ for some $0 \leq \beta < 1$. This will be negligible for our analysis, so we assume that $\bar r_A$ is an integer for simplicity. We extend the above definitions to $k_{B,r, \varepsilon}$, $D_{B, \varepsilon}$,  and $\overbar{r}_{B, \varepsilon}$ in the natural way. Since $\varepsilon$ is fixed throughout the argument, we suppress it from the subscripts of the above notation to ease readability. Thus, we use $k_{X,r}$, $D_{X}$,  and $\overbar{r}_{X}$ for $X\in \{A,B\}$.

By~\eqref{def:kAr} and~\eqref{eq: r1}, it is immediate that for $X\in \{A,B\}$,
\begin{equation}
\label{eq:k'bound}    k_{X,r}\log(dn/2k_X)=\Omega(k_X\log\log (dn/2k_X))= \omega(k_X),\quad \text{for all $1\le r\le \bar r_X$}.
\end{equation}
We are now able to define torsos and use them to partition $\mathcal T_{\delta_0, n, d}$.

\begin{definition}
Let $H$ be a graph and $D>0$. A spanning subgraph $T$ of $H$ is a $D$-torso of $H$ if it is an edge-maximum spanning subgraph of $H$ with maximum degree at most $D$.     
\end{definition}


Let $\mathcal T_{A,B,k_{A},{r_A}, k_{B},{r_B}, d}$ be the set of triangle-free $d$-regular graphs $G$ where for each $X\in \{A,B\}$, $G[X]$ has a $D_X$-torso $T$ with the following properties:
 \begin{itemize}
     \item $k_{X, r_X} < |T| \leq k_{X, r_X-1}$ if $r < \overbar r_{X}$;
     \item  $|T| \leq k_{X, \overbar{r}_{X}-1}$ if $r = \overbar r_{X}$.
     \item $G[X]$ had maximum degree $d/2$.
 \end{itemize} 
 
Since every graph has a torso, the defect graphs for any max cut will have maximum degree $d/2$, and the first two conditions above give us a way to classify every graph in $\mathcal T_{\delta_0,n,d}$ based on the size of the torsos of its defect graphs, we have that 
\[
\mathcal T_{\delta_0, n, d} = \bigcup_{A,B}\bigcup_{k_A+ k_B = 0}^{\delta_0 n}\bigcup_{r_A, r_B = 1}^{\bar r_A, \bar r_B}\mathcal T_{A, B, k_A, r_A, k_B, r_B}.
\]
Moreover, note that
$\bigcup_{A,B}\mathcal T_{A,B,0, r_A, 0,r_B, d} = \mathcal G_{n,d}^\text{bip}$, and so
\begin{equation}\label{eq: main counting}
    \frac{|\mathcal T_{\delta_0, n, d}|}{|\mathcal G_{n,d}^\text{bip}|} \leq 1 + \sum_{A,B}\sum_{k_A + k_B = 1}^{\delta_0dn/2}  \sum_{r_A, r_B = 1}^{\overbar{r}_{A}, \overbar{r}_{B}}\frac{|\mathcal T_{A, B, k_{A}, {r_A},k_{B}, {r_B}, d}|}{|\mathcal G_{n,d}^{\text{bip}}|}.
\end{equation}

We now fix a partition $A,B$ of $[n]$, and bound $|\mathcal T_{A, B, k_A, r_A, k_B, r_B, d}|$ given $k_A$, $k_B$, $r_A$, and $r_B$. As $A$ and $B$ are fixed, we further suppress them from the subscript. Similarly we omit $d$ whenever no confusion can arise. Thus, we use $\mathcal T_{k_A, r_A, k_B, r_B}$ for $\mathcal T_{A, B, k_A, r_A, k_B, r_B, d}$ in the rest of the paper. Throughout the whole proof we make the following assumptions as discussed before.
\begin{align}
\text{$\mu_0, \varepsilon > 0$ are fixed }&\text{and sufficiently small,    $\log\log 1/\delta_0 = \varepsilon^{-5}$,}\notag\\ 
&\text{and $(1 + \varepsilon)\frac{\sqrt 3}{2}\sqrt{n\log n} \leq d = o(n)$.}\label{Assumptions}
\end{align}

\subsection{Large defect graphs}
The lemmas presented in the subsection arise from directly applying the results of \cite{osthus2003densities} when the defect graphs are large enough that the bounds established in \cite{osthus2003densities} are small compared to the probability that a randomly chosen graph is regular.
When $k_A + k_B$ is sufficiently large, the two lemmas below give sufficient bounds on 
$|\mathcal T_{k_A, r_A, k_B, r_B}|$.

\begin{lem}\label{lem: large k far threshold}
    Assume \eqref{Assumptions}. If $n\log d \leq k_A + k_B \leq \delta_0dn/2$, we have
    \[
    \frac{|\mathcal T_{k_A, r_A, k_B, r_B}|}{|\mathcal G_{n,d}^{\text{bip}}|} \leq \frac{1}{\binom{n}{n/2}}\exp\left[-\varepsilon^{3}\left(k_{A, r_A - 1}\log \frac{dn}{2k_A} + k_{B, r_B - 1}\log\frac{dn}{2k_B}\right)\right].
    \]
    
    \end{lem}

With slightly more sophisticated methods, we are able to get the same bound as above for even smaller defect graphs as long as $d$ is close to the critical order $\sqrt{n \log n}$.

    \begin{lem}\label{lem: large k near threshold}
Assume~\eqref{Assumptions} and suppose that $k_{A,r_A-1}\ge k_{B,r_B-1}$, $d \leq \eps^{-4}\sqrt{n\log n}$, $\eps^4 n \leq k_{A,r_A-1} \leq \delta_0dn/2$, and $k_A, k_B \leq n\log d$. Then,
  
    \[
    \frac{|\mathcal T_{k_A, r_A, k_B, r_B}|}{|\mathcal G_{n,d}^{\text{bip}}|} \leq \frac{1}{\binom{n}{n/2}}\exp\left[-{\varepsilon^{3}}\left(k_{A, r_A - 1}\log \frac{dn}{2k_A} + k_{B, r_B - 1}\log\frac{dn}{2k_B}\right)\right].
    \]
    
\end{lem}

\cref{lem: large k far threshold} and \cref{lem: large k near threshold} are proved in \cref{sec: Large k}.

\subsection{Large Torsos}
Next, we consider the case where $G[A]$ and $G[B]$ do not have too many edges in total, and at least one of them has a large torso, i.e.\ $r_A<\bar r_A$ or $r_B<\bar r_B$. In the case that only one of them, say $G[A]$, has a large torso, we need a technical condition that
$k'_A\log\frac{dn}{2k_A} \geq \frac{2}{\varepsilon^2}k'_B\log\frac{dn}{2k_B}$, which allows us to bound $|\mathcal T_{k_A, r_A,k_B, r_B}|$ by using the large torso in $G[A]$ without worrying about edges in $G[B]$ creating triangles.

\begin{lem}\label{lem: Janson Counting 1}
    Assume \eqref{Assumptions}. Suppose $d \leq \eps^{-4}\sqrt{n \log n}$, $k_{A, r_A -1},  k_{B, r_B - 1} \leq \eps^4 n$, $1 \leq k_A + k_B \leq n \log d$, $r_A < \overbar{r}_{A}$, and either $r_B < \overbar{r}_{B}$ or $k_{A, r_A -1}\log\frac{dn}{2k_A} \geq \frac{2}{\varepsilon^2}k_{B, r_B -1}\log\frac{dn}{2k_B}$. Then,
    \[
    \frac{|\mathcal T_{k_A, r_A,k_B, r_B}|}{|\mathcal G^{\text{bip}}_{n,d}|} \leq \frac{1}{\binom{n}{n/2}} \exp\left[-\varepsilon^2\left(k_{A, r_A -1}\log\frac{dn}{2k_A} + k_{B, r_B -1}\log\frac{dn}{2k_B}\right)\right].
    \]
\end{lem}

Since the above lemma only considers the case that $d$ is near the threshold, we will only need to apply it when the defect graphs are too small to apply \cref{lem: large k near threshold}. The second lemma below also deals with the case that one of the defect graphs, say $G[A]$, has a large torso, and applies to the range of $d$ where $d\gg \sqrt{n\log n}$. Restricting $d$ away from the threshold allows us to weaken the technical condition $k'_A\log\frac{dn}{2k_A} \geq \frac{2}{\varepsilon^2}k'_B\log\frac{dn}{2k_B}$ in ~\cref{lem: Janson Counting 1}.

\begin{lem}\label{lem: Janson Counting 2}
    Assume \eqref{Assumptions}. Suppose that $k_{A, r_A - 1} \log\frac{dn}{2k_A} \geq k_{B, r_B - 1}\log\frac{dn}{2k_B}$, $1 \leq k_A + k_B \leq n \log d$, $\varepsilon^{-4}\sqrt{n\log n} \leq d$, and $r_A < \overbar{r}_{A}$. Then,
    \[
    \frac{|\mathcal T_{k_A, r_A ,k_B, r_B}|}{|\mathcal G^{\text{bip}}_{n,d}|} \leq \frac{1}{\binom{n}{n/2}}\exp\left[-{\varepsilon^{-1}}\left(k_{A, r_A - 1} \log\frac{dn}{2k_A} + k_{B, r_B - 1}\log\frac{dn}{2k_B}\right)\right].
    \]
\end{lem}

\cref{lem: Janson Counting 1} and \cref{lem: Janson Counting 2} are proved in \cref{sec: Janson}.

\subsection{Small Torsos}
Finally, we consider the case where one of $G[A]$ and $G[B]$ has a small torso, i.e.\ $r_A=\bar r_A$ or $r_B=\bar r_B$. Assume $r_B = \bar r_B$. We will apply these lemmas when $k_A + k_B$ is not too large (as otherwise we apply \cref{lem: large k far threshold}) and when $k_{A, r_A - 1}\log\frac{dn}{2k_A} \leq \frac{2}{\varepsilon^2}k_{B, r_B - 1}\log\frac{dn}{2k_B}$ (as otherwise we may apply \cref{lem: Janson Counting 1}). 

\begin{lem}\label{lem: Small torso counting small k}
    Assume \eqref{Assumptions}. Suppose $1 \leq k_A + k_B \leq n\log d$, $k_{A, r_A - 1}\log\frac{dn}{2k_A} \leq \frac{2}{\varepsilon^2}\times k_{B, r_B - 1}\log\frac{dn}{2k_B}$, $k_B/16 < |A|\log^{-4000/\eps^4}\frac{dn}{2k_B}$, and $r_B = \overbar{r}_{B}$. Then,

    \[
    \frac{|\mathcal T_{k_A, r_A, k_B, r_B}|}{|\mathcal G^{\text{bip}}_{n,d}|} \leq \frac{1}{\binom{n}{n/2}} \exp\left[ -2\left( k_{A, r_A - 1}\log\frac{dn}{2k_A} + k_{B, r_B -1}\log\frac{dn}{2k_B}\right) \right].
    \]
\end{lem}

\newtheorem*{lem: Small Torso Counting}{Lemma \ref{lem: Small Torso Counting}}
\begin{lem}\label{lem: Small Torso Counting}
    Assume \eqref{Assumptions}. Suppose $1 \leq k_A + k_B \leq n\log d$, $k_{A, r_A - 1}\log\frac{dn}{2k_A} \leq \frac{2}{\varepsilon^2} k_{B, r_B - 1}\log\frac{dn}{2k_B}$, $k_B/16 \geq |A|\log^{-4000/\eps^4}\frac{dn}{2k_B}$, and $r_B = \overbar{r}_{B}$. Then,

    \[
    \frac{|\mathcal T_{k_A, r_A, k_B, r_B}|}{|\mathcal G^{\text{bip}}_{n,d}|} \leq \frac{1}{\binom{n}{n/2}} \exp\left[ -2\left( k_{A, r_A - 1}\log\frac{dn}{2k_A} + k_{B, r_B -1}\log\frac{dn}{2k_B}\right) \right].
    \]
\end{lem}
Note that these lemmas obtain the same bound. We split them up because their proofs differ quite significantly.
\cref{lem: Small torso counting small k} and \cref{lem: Small Torso Counting} are proved in \cref{sec: Small Torso}.

\subsection{Completing the proof}\label{sec: completing the proof}
One final observation is needed to complete the proof. 
\begin{prop}\label{prop: ab}
    $|\mathcal T_{k_A, r_A, k_B, r_B}| > 0$ only if $|A| = \frac{n}{2} + \frac{k_A-k_B}{d}$ and $|B| = \frac{n}{2} - \frac{k_A - k_B}{d}$.
\end{prop}
\begin{proof}
    Let $a=|A|$ and $b=|B|$.
    Suppose $|\mathcal T_{k_A, r_A, k_B, r_B}| > 0$. Then, 
    \[
        da - 2k_A = db - 2k_B,\quad \text{and}\         a + b = n.
    \]
    That is, 
    \[
    a = \frac{n}{2} + \frac{k_A - k_B}{d},\quad \text{and}\     b = \frac{n}{2} - \frac{k_A-k_B}{d}.
    \]
\end{proof}

By this proposition, it suffices to consider only partitions $(A,B)$ such that 
\begin{equation}
|A| = \frac{n}{2} + \frac{k_A-k_B}{d},\quad \text{and}\ \  |B| = \frac{n}{2} - \frac{k_A - k_B}{d},\label{eq:AB}
\end{equation}
as otherwise $|\mathcal T_{k_A, r_A, k_B, r_B}| = 0$.  Moreover,~\eqref{eq:AB} fixes $|A|$ and $|B|$ in terms of $k_A$ and $k_B$, which is convenient for many arguments. 

\medskip

\noindent \textit{Proof of $1$-statement of \cref{thm: 0-1}.} By \cref{thm: linear d}, we may assume $d = o(n)$. Hence, assume~\eqref{Assumptions}. By \cref{thm: Luckzak}, it suffices to consider partitions $(A,B)$ satisfying~\eqref{eq:AB} where $k_A + k_B \leq \delta_0dn/2$. We now use \cref{lem: large k far threshold}, \cref{lem: large k near threshold}, \cref{lem: Janson Counting 1}, \cref{lem: Janson Counting 2}, \cref{lem: Small torso counting small k}, and \cref{lem: Small Torso Counting} to bound $|\mathcal T_{k_A, r_A, k_B, r_B}|/|\mathcal G_{n,d}^{\text{bip}}|$ in~\eqref{eq: main counting} when $k_A + k_B > 0$. 
To do so, it is useful to first define the following conditions to clarify our cases:
\begin{enumerate}
    \item[(A)] $k_A + k_B \geq n\log d$;
    \item[(B)] $k_{A, r_A - 1} \geq \eps^4 n$ or $k_{B, r_B -1} \geq \eps^4 n$;
    \item[(C)] $d \leq \eps^{-4}\sqrt{n \log n}$;
    \item[(D)] $r_A < \bar r_A$;
    \item[(E)] $k_{A, r_A - 1}\log\frac{dn}{2k_A} \geq 2\eps^{-2}k_{B, r_B - 1}\log\frac{dn}{2k_B}$ or $r_B < \bar r_B$;
    \item[(F)] $k_B/16 < |A|\log^{-4000/\eps^4}\frac{dn}{2k_B}$.
\end{enumerate}

Then, we apply \cref{lem: large k far threshold} when (A) holds, \cref{lem: large k near threshold} when (B) and (C) hold, \cref{lem: Janson Counting 2} when (C) fails and (D) and (E) hold (by potentially swapping $A$ and $B$ in (D) and (E)), \cref{lem: Small torso counting small k} when (D) and (E) do not both hold (even when swapping $A$ and $B$) and (F) holds, \cref{lem: Small Torso Counting} when (D) and (E) do not both hold and (F) does not hold, and \cref{lem: Janson Counting 1} when (B) fails, (C) holds, and (D) and (E) both hold.

Note that \cref{lem: large k far threshold} and \cref{lem: large k near threshold} give the weakest bounds of all of the lemmas, so we have
\begin{align*}
    \frac{|\mathcal T_{\delta_0, n, d}|}{|\mathcal G_{n,d}^\text{bip}|}
    &\leq1 + \sum_{k_A + k_B = 1}^{\delta_0dn/2}\sum_{r_A, r_B \geq 1}^{\overbar{r}_{A}, \overbar{r}_{B, }}\frac{\binom{n}{|A|}}{\binom{n}{n/2}}\exp\left[-{\varepsilon^3}\left((1 - \varepsilon)^{r_A}k_A\log\frac{dn}{2k_A} + (1 - \varepsilon)^{r_B}k_B\log\frac{dn}{2k_B}\right)\right]\\
    &\leq1 +  \sum_{k_A + k_B = 1}^{\delta_0dn/2}\exp\left[-2\varepsilon\left(k_A\log\log\frac{dn}{2k_A} + k_B\log\log\frac{dn}{2k_B}\right) + O\left( \log\log\log n\right)\right]\\
    &= 1 + o(1)
\end{align*}
where the second inequality holds by (\ref{eq: r1}) and (\ref{eq: r asymptotics}). \qed

\section{Large defects: proofs of \cref{lem: large k far threshold} and \cref{lem: large k near threshold} } \label{sec: Large k}

Before we begin these two proofs, we introduce general counting results for the number of ways to choose defect edges and cross-edges that will be used throughout the remainder of this paper.

In our calculations, it is important to consider choosing graphs with degree constraints, but it is not important to consider the precise degree sequence constraints that arise. The lemma below allows us to consider only the worst case degree sequence, which is when the degree sequence is regular. Even when there is no regular degree sequence with the correct degree sum, the lemma below allows us to bound the number of degree sequences by considering binomial coefficients with non-integer entries.

\begin{lem}\label{cor: max sequence}
    Assume~\eqref{Assumptions}. Let $U,V$ be disjoint vertex sets where $||U| - |V|| \leq (|U| + |V|)/2$. Let $\mathbf d$ be a degree sequence on $U \cup V$ with maximum degree at most $d$ and where $\sum_{x \in U \cup V}d_x \geq 2(|U| + |V|)$. Let $M = \sum_{x \in U}d_x$. Then, there are at most 
    \begin{equation}\notag
        O\left[{|U||V| \choose M}^{-1}{|V| \choose M/|U|}^{|U|}{|U| \choose M/|V|}^{|V|}\right]
    \end{equation}
    bipartite graphs with bipartition $(U,V)$ and degree sequence $\mathbf d$.
    

\end{lem}

Next, we give a definition that allows us to easily reference and classify the defect graphs we are considering in $\mathcal T_{k_A, r_A, k_B, r_B}$.

\begin{definition} Let $1 \le r_A \leq \overbar{r}_{A}$.
    A graph $G$ is a $(k_A, r_A, A, D_A)$-graph if
    \begin{itemize}
        \item $V(G) = A$,
        \item $|E(G)| = k_A$.
        \item If $r_A < \overbar r_A$, $G$ has a $D_A$-torso $T$ where $k_{A} < |E(T)| \leq k_{A, r_{A - 1}}$,
        \item If $r_A = \overbar r_A$, $G$ has a $D_A$-torso $T$ where $|E(T)| \leq k_{A, r_A - 1}$.
    \end{itemize}
    Let $(k_B, r_B, B, D_B)$ be defined symmetrically.
\end{definition}
Notice that given any $(k_B, r_B)$, if $G\in \mathcal T_{k_A, r_A, k_B, r_B}$ then $G[A]$ is a $(k_A, r_A, A, D_A)$-graph. Therefore, we may upper bound $|\mathcal T_{k_A, r_A, k_B, r_B}|$ by bounding the number of $(k_A, r_A, A, D_A)$-graphs and the number of $(k_B, r_B, B, D_B)$-graphs. The following proposition is used many times to approximate the number of defect graphs.


\begin{prop}\label{cor: counting defects}
    Suppose $|A| \leq \frac{n}{2}(1 + \delta_0)$. For $1 \leq r_A \leq \overbar{r}_{A}$, there are at most
    \[
        \exp\left[k\log\frac{n}{d} + (1 + \varepsilon^2)k_{A, r_A-1}\log\frac{dn}{2k_A}\right]
    \]
    $(k_A, r_A, A, D_{A})$-graphs.    
\end{prop}

\cref{cor: max sequence} and \cref{cor: counting defects} are proved in \cref{sec: Defects}. \medskip


\subsection{Proof of \cref{lem: large k far threshold}}
Let $a=|A| $, $b=|B| $, $k'_A = k_{A, r_A - 1}$, and $k'_B=k_{B, r_B - 1}$. 
    By~\eqref{eq:AB} and the assumption that $k_A + k_B \leq \delta_0dn/2$ from the lemma, it is sufficient to consider that $|a-b| \leq \delta_0n < \varepsilon^2dn/2$ for sufficiently small $\varepsilon$, as otherwise $|\mathcal T_{k_A, r_A, k_B, r_B}|=0$ and the lemma holds trivially. 
    
Let $\mathcal G'_{k_A, r_A, k_B, r_B}$ denote the set of graphs $G$ on $A \cup B$ with exactly $dn/2$ edges where $G[X]$ is a $(k_X, r_X, X, D_{X})$-graph for both $X \in \{A, B\}$. Let $\mathcal T'_{k_A, r_A, k_B, r_B}$ denote the set of triangle-free graphs in $\mathcal G'_{k_a, r_A, k_B, r_B}$. Notice that $\mathcal T'_{k_A, r_A, k_B, r_B}$ is a superset of $\mathcal T_{k_A, r_A, k_B, r_B}$ since each $d$-regular graph on $n$ vertices has $dn/2$ edges and both sets have the same restrictions for defect graphs. The size of $\mathcal T'_{k_A, r_A, k_B, r_B}$ has been bounded in~\cite{osthus2003densities}, by using Lemma 12 of \cite{osthus2003densities} to bound $|\mathcal G'_{k_A, r_A, k_B, r_B}|$ and then taking the weakest bound among Lemma 17, Lemma 18, and Lemma 19 of \cite{osthus2003densities} to bound the probability that a graph chosen uniformly at random from $\mathcal G'_{k_A, r_A, k_B, r_B}$ is in $\mathcal T'_{k_A, r_A, k_B, r_B}$. We restate the bound in the claim below.

\begin{claim}\label{clm: no degree constraints}\cite{osthus2003densities} 
     \[
     |\mathcal T'_{ k_A, r_A,k_B, r_B}| \leq \binom{n^2/4}{dn/2}\exp\left[-\varepsilon^2\left(k'_A\log \frac{dn}{2k_A} + k'_B\log\frac{dn}{2k_B}\right)\right].
     \]
\end{claim}
Hence,
    \begin{equation}\label{eq: trivial bound}
        |\mathcal T_{k_A, r_A, k_B, r_B}| \leq \binom{n^2/4}{dn/2}\exp\left[-\varepsilon^{2}\left(k'_A\log\frac{dn}{2k_A} + k'_B \log\frac{dn}{2k_B}\right)\right].
   \end{equation}
    Then, using \eqref{eq: trivial bound} to bound $\mathcal T_{k_A, r_A, k_B, r_B}$ and  \eqref{eq:reg bip count2} to enumerate $\mathcal G_{n,d}^\text{bip}$, we have
    \[
    \frac{|\mathcal T_{k_A, r_A, k_B, r_B}|}{|\mathcal G_{n,d}^{\text{bip}}|} \leq \frac{2}{\binom{n}{n/2}}\frac{\binom{n^2/4}{dn/2}^2}{\binom{n/2}{d}^{n}}\exp\left[-\varepsilon^{2}\left(k'_A\log\frac{dn}{2k_A} + k'_B \log\frac{dn}{2k_B}\right)\right].
    \]
    We then approximate $\frac{\binom{n^2/4}{dn/2}^2}{\binom{n/2}{d}^{n}}$ by applying \cref{cor: choose} to each binomial coefficient. This gives 
    \begin{align}
        \frac{\binom{n^2/4}{dn/2}^2}{\binom{n/2}{d}^{n}} &= \frac{(n/2)^{2dn}}{(dn/2)^{dn + 1}(1 - 2d/n)^{n^2/2 - dn}} \cdot \frac{d^{dn + n/2}(1 - 2d/n)^{n^2/2 - dn}}{(n/2)^{dn}}\exp[O(n\log d)]\notag\\
        &=\exp[O(n\log d)].\notag
    \end{align}
    
    It follows that
    \begin{align*}
        \frac{|\mathcal T_{k_A, r_A, k_B, r_B}|}{|\mathcal G_{n,d}^{\text{bip}}|}&\leq \frac{1}{\binom{n}{n/2}}\exp\left[-\varepsilon^{2}\left(k'_A\log\frac{dn}{2k_A} + k'_B \log\frac{dn}{2k_B}\right) + O(n\log d)\right]\\
        &\leq \frac{1}{\binom{n}{n/2}}\exp\left[-\varepsilon^{3}\left(k'_A\log\frac{dn}{2k_A} + k'_B \log\frac{dn}{2k_B}\right)\right],
    \end{align*}
    where the final inequality holds since $n\log d = O(k_A + k_B) = o(k'_A\log\frac{dn}{2k_A} + k'_B\log\frac{dn}{2k_B})$ by~\eqref{eq:k'bound}. This completes the proof of the lemma. \qed

\smallskip

\subsection{Proof of \cref{lem: large k near threshold}} \label{sec:large k near threshold}
We may assume that $(A,B)$ satisfies~\eqref{eq:AB}.
    Let $a = |A|$, $b = |B|$, $k'_A = k_{A, r_A - 1}$, and $k'_B = k_{B, r_B- 1}$. By~\eqref{eq:AB}, $|a-b| \leq \delta_0n < \varepsilon^2dn/2$ for sufficiently small $\varepsilon$. Let $\lambda  = \lambda (n)=\beta/\log n$ where $\beta>0$ is an arbitrary fixed constant such that $\lambda  dn/2$ is an integer.
    \begin{claim}\label{clm: k'Ak'B} Suppose that $k_A,k_B=o(dn)$.
     \[
     k'_A\log\frac{dn}{2k_A} \geq k'_B\log\frac{dn}{2k_B}(1 + o(1)).
     \]
    \end{claim}
    \begin{proofclaim}
       By \eqref{eq:k'bound}, $k_X = o\left(k'_X{\log\frac{dn}{2k_X}}\right)$, and $k_X\ge k'_X$, for $X \in \{A,B\}$. Thus, $k'_X\log\frac{dn}{2k_X} = k'_X\log\frac{dn}{2k'_X}(1 + o(1))$  for $X \in \{A,B\}$, so
    \begin{align}
        k'_B\log\frac{dn}{2k_B} = k'_B\log\frac{dn}{2k'_B}(1 + o(1)) \leq k'_A\log\frac{dn}{2k'_A}(1 + o(1)) = k'_A\log\frac{dn}{2k_A}(1 + o(1)),\nonumber
    \end{align}
    where the inequality above holds because $k'_A \geq k'_B$ and $y\log\frac{C}{y}$ is monotone increasing when $y< C/e$. 
    \end{proofclaim}
    
Let $\mathcal T'$ denote the set of triangle-free graphs $G$ on $A \cup B$ with exactly $(1 - \lambda )dn/2$ edges where $G[X]$ is a $(k_X, r_X, X, D_{X})$-graph for both $X \in \{A, B\}$.
    We bound $|\mathcal T_{ k_A, r_A, k_B, r_B}|$ by first choosing a graph in $\mathcal T'$, and then choosing the remaining $\lambda  dn/2$ edges to obtain a regular graph. By \cref{clm: no degree constraints} (notice the difference in the number of edges in  $\mathcal{T}'$ compared to the original application of \cref{clm: no degree constraints}; this does not affect the validity of the claim),
    \begin{equation}
    |\mathcal T'| \leq \binom{n^2/4}{(1-\lambda )dn/2}\exp\left[-\varepsilon^2\left(k'_A\log\frac{dn}{2k_A} + k'_B\log\frac{dn}{2k_B}\right)\right]. \label{eq:no degree constraints 2}
    \end{equation}

For convenience, let $m' = \lambda  dn/2$, $a^* = m'/a$, and $b^* = m'/b$.
    Observe that 
    \begin{itemize}
        \item $a = \frac{n}{2}\left(1 + \frac{2k_A - 2k_B}{dn}\right)$,
        \item $b = \frac{n}{2}\left(1 - \frac{2k_A - 2k_B}{dn}\right)$,
        \item $a^* = \lambda  d\left(1 - \frac{2k_A - 2k_B}{dn} + O\left(\frac{(k_A + k_B)^2}{d^2n^2}\right)\right)$,
        \item $b^* = \lambda  d\left(1 + \frac{2k_A - 2k_B}{dn} + O\left(\frac{(k_A + k_B)^2}{d^2n^2}\right)\right)$.
    \end{itemize}

By \cref{cor: max sequence}, there are at most, up to a constant factor, 
$
\binom{b}{a^*}^a\binom{a}{b^*}^b/\binom{ab}{m'}
$
ways to choose the remaining $m'$ edges to complete any graph in ${\mathcal T}'$ into a graph in $\mathcal T_{k_A, r_A, k_B, r_B}$. On the other hand, every graph in $\mathcal T_{k_A, r_A, k_B, r_B}$ has at least $\binom{dn/2 - k_A - k_B}{m'}$ subgraphs that are members of ${\mathcal T}'$, and so can be constructed in at least $\binom{dn/2 - k_A - k_B}{m'}$ different ways by the process we have described. It follows that
\[
|\mathcal T_{k_A, r_A, k_B, r_B}|\le \frac{\binom{n^2/4}{(1 - \lambda )dn/2}\binom{b}{a^*}^a\binom{a}{b^*}^b}{\binom{ab}{m'}\binom{dn/2 - k_A - k_B}{m'}}\exp\left[-{\varepsilon^2}\left(k'_A\log \frac{dn}{2k_A} + k'_B\log\frac{dn}{2k_B}\right)+O(1)\right].
\]
    Then, by~\eqref{eq:reg bip count2},
    \begin{align}
        \frac{|\mathcal T_{k_A, r_A, k_B, r_B}|\binom{n}{n/2}}{|\mathcal G_{n,d}^{\text{bip}}|} &\leq \frac{\binom{n^2/4}{dn/2}\binom{n^2/4}{(1 - \lambda )dn/2}\binom{b}{a^*}^a\binom{a}{b^*}^b}{\binom{n/2}{d}^n\binom{ab}{m'}\binom{dn/2 - k_A - k_B}{m'}}\notag\\
        &\times\exp\left[-{\varepsilon^2}\left(k'_A\log \frac{dn}{2k_A} + k'_B\log\frac{dn}{2k_B}\right)+O(1)\right]. \label{eq: large k initial ineq}
    \end{align}

By \cref{prop: choose approx ub top} and  \cref{prop: choose approx ub bot}  we obtain 
\begin{align}
    \binom{b}{a^*}^a\binom{a}{b^*}^b &= \binom{n/2}{\lambda  d}^n\left(2\lambda  d/n\right)^{2k_A - 2k_B}(2 \lambda  d/n)^{-2k_A + 2k_B}\notag\\
    &\times\exp\left[O\left(k_A + k_B + \frac{(k_A + k_B)^2}{dn}\log\frac{n}{d}\right)\right]\notag\\
    &= \binom{n/2}{\lambda  d}^n\exp[O(k_A + k_B)].\label{eq: large k approx 1}
\end{align}
Likewise, by applying \cref{prop: choose approx ub top} (with $x = n^2/4$ to $\binom{ab}{m'}$ and $x = dn/2$ to \newline $\binom{dn/2 - k_A - k_B}{m'}$), we obtain
\begin{equation}\label{eq: large k approx 2}
    \binom{ab}{m'}\binom{dn/2 - k_A - k_B}{m'} = \binom{n^2/4}{m'}\binom{dn/2}{m'}\exp[O(k_A + k_B)].
\end{equation}
Substituting \eqref{eq: large k approx 1} and \eqref{eq: large k approx 2} into \eqref{eq: large k initial ineq}, we have 
\begin{equation}\label{eq: large k simplified binom coeff}
    \frac{\binom{n^2/4}{dn/2}\binom{n^2/4}{(1 - \lambda )dn/2}\binom{b}{a^*}^a\binom{a}{b^*}^b}{\binom{n/2}{d}^n\binom{ab}{m'}\binom{dn/2 - k_A - k_B}{m'}} = \frac{\binom{n^2/4}{dn/2}\binom{n^2/4}{(1 - \lambda )dn/2}\binom{n/2}{\lambda  d}^n}{\binom{n/2}{d}^n\binom{n^2/4}{\lambda  dn/2}\binom{dn/2}{\lambda  dn/2}}\exp[O(k_A + k_B)]. 
\end{equation}
We now evaluate each binomial coefficient by applying \cref{cor: choose}. We again do not give the complete details of this computation as it is not too difficult. We use that $n = O(k_A + k_B)$ to absorb many terms into $\exp[O(k_A + k_B)]$ below. This gives that the right hand side of \eqref{eq: large k simplified binom coeff} is 
\begin{align}
   & \lambda ^{-n/2}\frac{\left(1 - \frac{2(1-\lambda )d}{n}\right)^{n^2/4 - (1 - \lambda )dn/2}\left(1 - \frac{2\lambda d}{n}\right)^{n^2/4 - \lambda dn/2}}{\left(1 - \frac{2d}{n}\right)^{n^2/4 - dn/2 }}\exp[O(k_A + k_B)] \notag \\
   &= \lambda ^{-n/2}\exp[O(\lambda  d^2) + O(k_A + k_B)].\notag
\end{align}
That is, 
\begin{align*}
    \frac{|\mathcal T_{k_A, r_A, k_B, r_B}|\binom{n}{n/2}}{|\mathcal G_{n,d}^{\text{bip}}|} \leq \exp\left[-\varepsilon^2\left(k'_A\log\frac{dn}{2k_A} + k'_B\log\frac{dn}{2k_B}\right)(1 + o(1))\right.\\ \left.+ (n/2)\log\frac{1}{\lambda } + O(\lambda  d^2 + k_A + k_B)\right].
\end{align*}

 To complete the proof, it suffices to show that  $\lambda  d^2 +\frac{n}{2}\log \frac{1}{\lambda } = o(k'_A\log \frac{dn}{k_A})$. Since $d \leq \eps^{-4}\sqrt{n\log n}$,
    \begin{align*}
        \lambda  d^2 = O(n) = O(k'_A) = o\left(k'_A\log\frac{dn}{2k_A}\right).
    \end{align*}
    Also, we have 
    \[
    \frac{n}{2}\log\frac{1}{\lambda } = \frac{n}{2}(\log\log n - \log\beta) \leq \frac{k'_A}{2\eps^4}(\log\log n  - \log \beta) = o\left(k'_A\log\frac{dn}{2k}\right).
    \]
\qed

\subsection{Proof of \cref{cor: max sequence} and \cref{cor: counting defects}}\label{sec: Defects}

\noindent\textit{Proof of \cref{cor: max sequence}.}
We first show that the number of graphs with degree sequence $\mathbf d$ is upper bounded by the number of graphs with degree sequence $\mathbf d'$, where $\mathbf d'$ is ``as regular as possible'' over $U$ and $V$. The following claim allows us to do this iteratively.

\begin{claim}\label{clm: iterative step}
    If there exist $x,y \in U$ or $x,y \in V$ where $d_x \geq d_y + 2$, then for $\mathbf d'$ defined by $d'_x = d_x - 1$, $d'_y = d_y + 1$, and $d'_w = d_w$ for all $w \in U\cup V \setminus \{x,y\}$, there are more bipartite graphs with bipartition $(U,V)$ and degree sequence $\mathbf d'$ than with degree sequence $\mathbf d$. 
\end{claim}
\begin{proofclaim}
    Without loss of generality, suppose $x,y \in U$.
    We count the number of bipartite graphs with bipartition $(U,V)$ and degree sequence $\mathbf d$ by first exposing the neighbourhoods of $x$ and $y$ and then choosing all remaining edges. Note that the number of choices for the remaining edges depend only on the degree sequence in $V$ obtained by exposing the neighbourhoods of $x$ and $y$. Thus, it suffices to show that for every degree sequence $\mathbf r$ on $V$ that could be obtained by exposing the neighbourhoods of $x$ and $y$, there are at least as many ways to obtain that degree sequence when exposing $d'_x$ neighbours of $x$ and $d'_y$ neighbours of $y$ than when exposing $d_x$ neighbours of $x$ and $d_y$ neighbours of $y$.

     Fix a degree sequence $\mathbf r$ on $V$ that can be obtained as above. Observe that $0 \leq r_v \leq 2$ for all $v \in B$ and $\sum_{v \in V} r_v = d_x + d_y$. Let $N$ denote the number of ways to expose $d_x$ neighbours of $x$ and $d_y$ neighbours of $y$ that induce degree sequence $r$ on $V$ and $N'$ denote the number of ways to expose $d'_x$ neighbours of $x$ and $d'_y$ neighbours of $d_y$ that induce degree sequence $r$ on $V$. It remains to show that $N \leq N'$. 

    Let $\alpha$ be the number of vertices with degree $2$ in $\mathbf r$. If $N = 0$, $N \leq N'$ holds trivially, so assume $N > 0$. Let $\beta$ be the number of vertices in $V$ that have degree $1$ in $\mathbf r$. Note that $\beta = d_x + d_y - 2\alpha$ and $d_x + d_y \leq 2d_x - 2$, so 
    \begin{equation}\label{eq: half}
        d_x - 1 - \alpha \geq \beta/2
    \end{equation}
    Also ,
    \begin{equation}\label{eq: N and N'}
        N = {\beta \choose d_x - \alpha} \quad \text{and} \quad N' = {\beta \choose d_x - 1 - \alpha}.
    \end{equation}
     Thus, by~\eqref{eq: half} and~\eqref{eq: N and N'}, $N' \geq N$, which completes the proof.
\end{proofclaim}

Let $u = |U|$ and $v = |V|$. Let $s = M/u$ and $t = M/v$.
By repeatedly applying \cref{clm: iterative step}, we have that the maximum is obtained when $\mathbf d|_U$ and $\mathbf d|_V$ are as regular as possible. Let $s' = \lceil s \rceil $ and $t' = \lceil t \rceil$. When $\mathbf d|_U$ and $\mathbf d|_V$ are as regular as possible, we will certainly be able to apply \cref{thm: bip enum}. Thus, there are at most 
\begin{equation}\label{eq: max sequence 1}
    \binom{uv}{M}^{-1}\binom{v}{s'-1}^{u - q}\binom{v}{s'}^q\binom{u}{t'-1}^{v - r}\binom{u}{t'}^r\exp[O(1)]
\end{equation}
bipartite graphs on degree sequence $\mathbf s, \mathbf t$ for some $0 \leq q \leq u$ and $0 \leq r \leq v$ where $s = s' - 1 + q/u$ and $t = t' - 1 + r/v$. The desired result follows immediately from \cref{prop: binom average}.


\qed

\medskip

\noindent\textit{Proof of \cref{cor: counting defects}.}
Let $a = |A|$, $k = k_A$, and $k' = k_{A, r_A - 1}$. Recall that $k'\log\frac{dn}{2k} \geq \frac{4}{\varepsilon^2}k\log\log\frac{dn}{2k}$ by (\ref{eq: r1}) and take $s = \frac{2k'}{D_{U}}$ if $r_A > 1$ and $s = 0$ if $r_A = 1$. Then, by Proposition~11 of \cite{osthus2003densities}, there are at most
\begin{equation}\label{eq: counting defects}
    a^s\binom{as}{k - k'}\binom{\binom{a}{2}}{k'}
\end{equation}
$(k_A, r_A, A, D_A)$-graphs (when $r_A = 1$ we define $s = 0$, so this holds trivially).

Observe that $a^s = \exp\left[\frac{k'\log\frac{dn}{2k}\log\frac
     {n}{2}}{d}O(1)\right]$ and $d = \omega\left(\log\frac{dn}{2k}\log\frac
     {n}{2}\right)$. Hence, $a^s = \exp\left[o\left(k'\log\frac{dn}{k}\right)\right]$. Moreover, observe that ${as \choose k - k'}{{a \choose 2} \choose k'}$ is bounded from above by
    \begin{align}
         & \frac{(as)^{k - k'}a^{2k'}}{(k - k')^{k - k'}(2k')^{k'}}\exp(O(k))\notag\\
        &= \exp\left[ k\log\frac{as}{k-k'} + k'\log\frac{a(k - k')}{2k's} + O(k)\right]\notag\\
        &=\exp\left[k\log\frac{nk'\log\frac{dn}{2k}}{d(k-k')} + k'\log\frac{dn(k - k')}{k'^2\log\frac{dn}{2k}} + O(k)\right]\notag\\
        &= \exp\left[k\log \frac{n}{d} + k\log\log\frac{dn}{2k} + k'\log\frac{dn}{k'} + k'\log\frac{(k - k')}{k'}+ O(k + k'\log\log n)\right]\notag\\
        &\leq \exp\left[k\log \frac{n}{d} + \frac{\varepsilon^2}{4}k'\log\frac{dn}{2k} + k'\log\frac{dn}{k'} + O(k + k'\log\log n)\right]\label{eq: ineq a}\\
        &= \exp\left[k\log \frac{n}{d} + (1 + \varepsilon^2/4)k'\log\frac{dn}{2k} + O(k + k'\log\log n)\right]\notag\\
        &\leq \exp\left[k\log\frac{n}{d} + (1 + \varepsilon^2)k'\log\frac{dn}{2k}\right],\notag
    \end{align}
    where~\eqref{eq: ineq a} holds by~\eqref{eq: r1}. Hence, by \eqref{eq: counting defects}, the proposition holds.
\qed

\section{Large torsos: proofs of \cref{lem: Janson Counting 1} and \cref{lem: Janson Counting 2}}\label{sec: Janson}


\subsection{Proof of \cref{lem: Janson Counting 1}}\label{sec: Janson Counting 1}
 As before,     let $a = |A|$, $b = |B|$, $k'_A = k_{A, r_A - 1}$, and $k'_B = k_{B, r_B - 1}$ and we may assume that $(A,B)$ satisfies \eqref{eq:AB}. Observe that if $k'_A = 0$, then $k'_A\log\frac{dn}{2k_A} = 1$ (by convention), and so $k'_B\log\frac{dn}{2k_B} = 1$, implying $k'_B = 0$. But then $k_A = k_B = 0$, so $k_A + k_B < 1$. Hence, $k'_A > 0$.
Let $\eps^7/\log n < \lambda  = \lambda (n) < \eps^6/\log n$. 
We will bound $|\mathcal T_{k_A, r_A, k_B, r_B}|$ by first specifying a triangle-free $(k_A, r_A, A, D_{A})$-graph $G_A$ and a triangle-free $(k_B, r_B, B, D_{B})$-graph $G_B$, and then estimating the number of ways to complete $(G_A,G_B)$ to a $d$-regular triangle-free graph with $G_A$ and $G_B$ being the defect graphs.
By~\cref{cor: counting defects},
 the number of choices for $(G_A,G_B)$ is at most
    \begin{equation}
\exp\left[(k_A+k_B)\log\frac{n}{d} + (1 + \varepsilon^2)\left(k'_A\log\frac{dn}{2k_A}+k'_B\log\frac{dn}{2k_B}\right)\right]\label{eq:defectbounds}
    \end{equation}
Fix $(G_A,G_B)$, an arbitrary $D_A$-torso $T_A$ of $G_A$, and an arbitrary $D_B$-torso $T_B$ of $G_B$. Let $S_A \subseteq A$ be chosen arbitrarily such that 

\begin{equation}
|S_A| = 2k'_A\ \text{and every vertex in $A' = A \setminus S_A$ has degree $0$ in $T_A$}.\label{eq:sizeSA}
\end{equation}
Here, the error created by assuming $k'_A$ is an integer is negligible for us. Note that $S_A$ exists because $T_A$ has at most $k'_A$ edges.
Let $S_B = \emptyset$ if $r_B = \overbar{r}_{B}$, and otherwise let $S_B \subseteq B$ be arbitrarily chosen such that $|S_B| = k'_B$ (again, assuming $k'_B$ to be an integer creates a negligible error) and every vertex in $B' = B \setminus S_B$ has degree $0$ in $T_B$. Let $s_A=|S_A| $ and $s_B=|S_B| $. Notice that $(S_A,S_B)$ depends on $(G_A,G_B)$. However, $s_A,s_B$ do not.

Let $\mathcal H_{G_A, G_B, T_A, T_B, S_A, S_B}$ be the set of $d$-regular graphs $H$ on vertex set $[n]$ where $H[A] = G_A$, $H[B] = G_B$, and no edge in $T_A \cup T_B$ is contained in a triangle in $H$. Notice that $\mathcal H_{G_A, G_B, T_A, T_B, S_A, S_B}$ is independent of  $(S_A, S_B)$. However, we will bound $|\mathcal H_{G_A, G_B, T_A, T_B, S_A, S_B}|$ by counting the number of ways to construct a member $H$ of $\mathcal H_{G_A, G_B, T_A, T_B, S_A, S_B}$ by adding edges into $H$ in multiple stages, where the sets of edges added in each stage depends on $(S_A,S_B)$. Hence, it is convenient to keep $S_A,S_B$ in the subscript in the notation $\mathcal H_{G_A, G_B, T_A, T_B, S_A, S_B}$.      Then,
\begin{equation}
|\mathcal T_{k_A, r_A, k_B, r_B}|\le \sum_{G_A,G_B} \min_{T_A,T_B,S_A,S_B}|{\mathcal H}_{G_A,G_B,T_A,T_B,S_A,S_B}|.\label{eq:insideGAGB}
\end{equation}

Next, we fix $(G_A,G_B,T_A,T_B,S_A,S_B)$ and we upper bound $|\mathcal H|$, where \newline${\mathcal H}={\mathcal H}_{G_A,G_B,T_A,T_B,S_A,S_B}$,
by counting graphs that can be constructed following the procedure below. 

    \begin{itemize}
        \item[1.](local $(S_A, S_B)$ sprinkling) Choose $m_1$ such that $0 \leq m_1 \leq \min\{s_As_B, ds_A,ds_B\}$.
 Include $m_1$ edges from $S_A \times S_B$ into $H$ such that no vertex in $H$ has degree more than $d$ and no edge in $T_A \cup T_B$ is contained in a triangle in $H$.
        \item[2.](local $(S_X, Y')$ sprinkling for distinct $X,Y \in \{A,B\}$) Let $m_A = ds_A  - \sum_{u \in S_A}d_{G_A}(u) - m_1$ and $m_B= ds_B -  \sum_{v \in S_B}d_{G_B}(v)- m_1$.  Include $(1 - \lambda )m_A$ edges from $S_A \times B'$ and $(1 - \lambda )m_B$ edges from $A' \times S_B$ into $H$ such that no triangle in $H$ contains an edge of $T_A \cup T_B$.
        \item[3.](saturating $S_A$ and $S_B$) For each $u \in S_A \cup S_B$, let $d'_u$ denote the degree of $u$ in the currently constructed $H$. For each $u \in S_A$, choose a set $\mathcal{U}$ of $d - d'_u$ vertices in $B'$ and add edges $\{ux: x\in{\mathcal U}\}$ to $H$.  Likewise, for each $v \in S_B$, choose a set $\mathcal{V}$ of $d - d'_v$ vertices in $A'$ and add edges $\{vy: y\in{\mathcal V}\}$ to $H$. All choices are made such that the resulting $H$ is a simple graph (i.e.\ no edge was chosen to be added when it was already in $H$) and no vertex in $H$ has degree more than $d$. Notice that there are precisely $\lambda  m_A$ edges incident to $S_A$ added in this step and $\lambda  m_B$ edges incident to $S_B$ added in this step.
        \item[4.](degree completion) Complete $H$ to be a (simple) $d$-regular graph on $[n]$ by including $dn/2 - k_A - k_B - m_1 - m_A - m_B$ edges in $A'\times B'$.
    \end{itemize}

    A visual representation of this procedure is given in Figure~\ref{fig: large torso}.

    \begin{figure}[h]
        \centering{
        \resizebox{0.8\textwidth}{!}{\fontsize{30pt}{30pt}\selectfont
\begingroup%
  \makeatletter%
  \providecommand\color[2][]{%
    \errmessage{(Inkscape) Color is used for the text in Inkscape, but the package 'color.sty' is not loaded}%
    \renewcommand\color[2][]{}%
  }%
  \providecommand\transparent[1]{%
    \errmessage{(Inkscape) Transparency is used (non-zero) for the text in Inkscape, but the package 'transparent.sty' is not loaded}%
    \renewcommand\transparent[1]{}%
  }%
  \providecommand\rotatebox[2]{#2}%
  \newcommand*\fsize{\dimexpr\f@size pt\relax}%
  \newcommand*\lineheight[1]{\fontsize{\fsize}{#1\fsize}\selectfont}%
  \ifx\svgwidth\undefined%
    \setlength{\unitlength}{626.70302023bp}%
    \ifx\svgscale\undefined%
      \relax%
    \else%
      \setlength{\unitlength}{\unitlength * \real{\svgscale}}%
    \fi%
  \else%
    \setlength{\unitlength}{\svgwidth}%
  \fi%
  \global\let\svgwidth\undefined%
  \global\let\svgscale\undefined%
  \makeatother%
  \begin{picture}(1,0.91074336)%
    \lineheight{1}%
    \setlength\tabcolsep{0pt}%
    \put(0,0){\includegraphics[width=\unitlength,page=1]{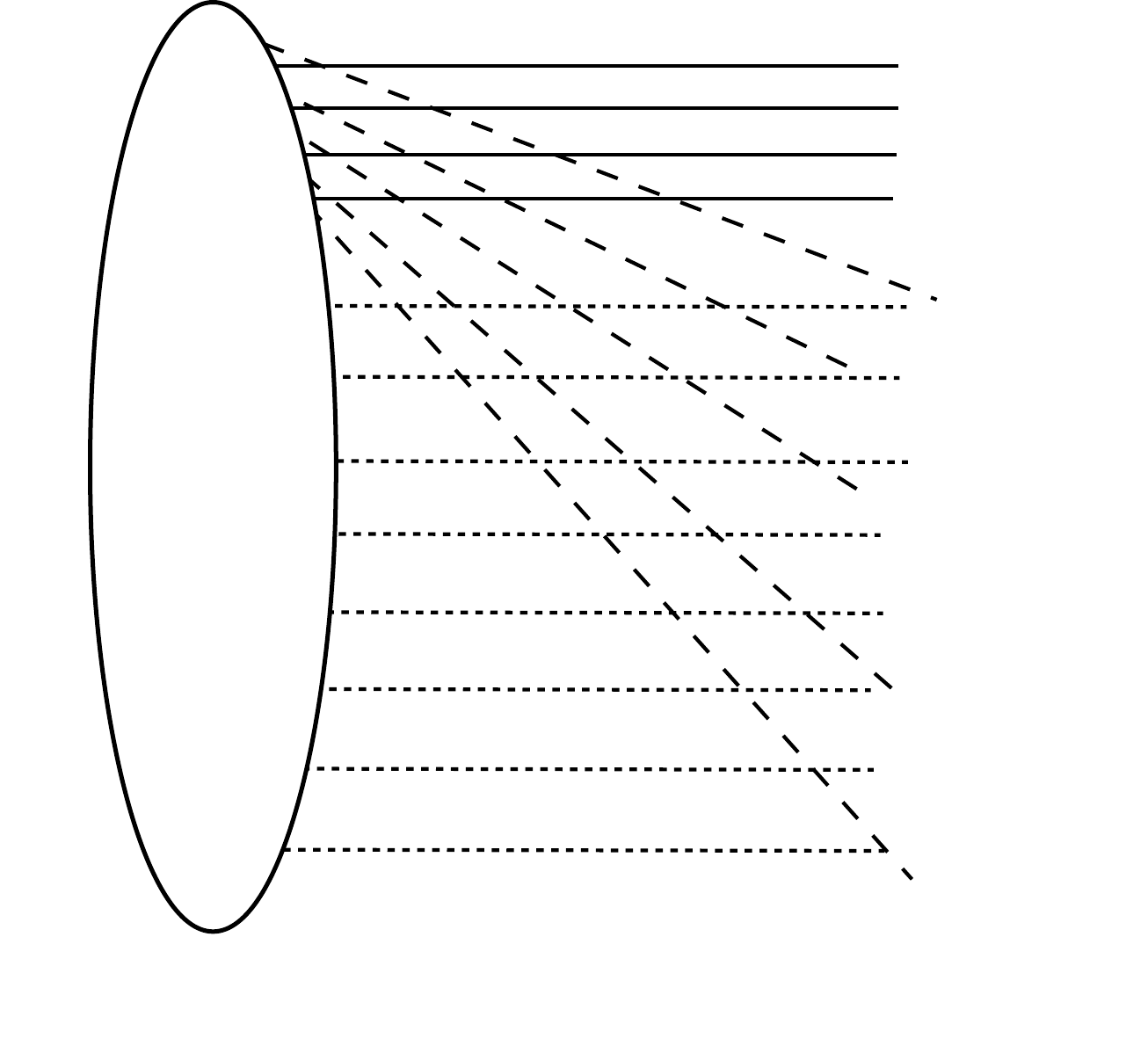}}%
    \put(0.16459814,0.04388433){\color[rgb]{0,0,0}\makebox(0,0)[lt]{\lineheight{1.25}\smash{\begin{tabular}[t]{l}$A$\end{tabular}}}}%
    \put(0.79628497,0.04355083){\color[rgb]{0,0,0}\makebox(0,0)[lt]{\lineheight{1.25}\smash{\begin{tabular}[t]{l}$B$\end{tabular}}}}%
    \put(0,0){\includegraphics[width=\unitlength,page=2]{Large_torso_thesis.pdf}}%
    \put(0.03488326,0.78991048){\color[rgb]{0,0,0}\makebox(0,0)[lt]{\lineheight{1.25}\smash{\begin{tabular}[t]{l}$S_A$\end{tabular}}}}%
    \put(0.90611351,0.79232245){\color[rgb]{0,0,0}\makebox(0,0)[lt]{\lineheight{1.25}\smash{\begin{tabular}[t]{l}$S_B$\end{tabular}}}}%
    \put(0,0){\includegraphics[width=\unitlength,page=3]{Large_torso_thesis.pdf}}%
  \end{picture}%
\endgroup%
}}
        
        \caption[Procedure for constructing graphs in \cref{lem: Janson Counting 1}]{Representation of above procedure. The solid cross-edges are added in step 1 of the procedure and only consider triangle-free constraints. The large dashed cross-edges are added in steps 2 and 3. The edges added in step 2 only consider triangle-free constraints and the edges added in step 3 only consider degree constraints. The small dashed cross edges are added in step 4 and only consider degree constraints.}
        \label{fig: large torso}
\end{figure}

Observe that every graph in ${\mathcal H}$ can be constructed by the procedure above (in multiple ways). Thus, we can estimate $|{\mathcal H}|$ by upper bounding the number of choices in each of the steps in the procedure, and by lower bounding the number of ways a graph in ${\mathcal H}$ can be constructed by the procedure.

Thus, we first upper bound the number of choices in steps 1 and 2. Notice that the restrictions 
\begin{equation}
m_A = ds_A  - \sum_{u \in S_A}d_{G_A}(u) - m_1,\ m_B= ds_B -  \sum_{v \in S_B}d_{G_B}(v)- m_1
\label{eq:oldconditions}
\end{equation}
depend on $(G_A,G_B)$ and we aim to obtain bounds that are independent of them so that the summation in~\eqref{eq:insideGAGB} is simpler. Thus, we relax these conditions to 
\begin{align}
    0 \leq m'_1 &\leq \min\{s_As_B, ds_A,ds_B\},\  ds_A - 2k_A-m'_1 \leq m'_A \leq ds_A - m'_1,\notag\\
    &\max\{ds_B - 2k_B - m_1', 0\} \leq m'_B \leq ds_B - m_1'.\label{eq:relaxedconditions}
\end{align}
Observe that any choices of $(m_1,m_A,m_B)=(m_1',m_A',m_B')$ that satisfy~\eqref{eq:oldconditions} must satisfy~\eqref{eq:relaxedconditions}. Next we estimate the number of choices in each step given parameters $(m_1',m_A',m_B')$ for $(m_1,m_A,m_B)$ that satisfies the relaxed conditions~\eqref{eq:relaxedconditions}.

\begin{claim}\label{clm: janson 1 step 3}
Given $(m_1,m_A,m_B)=(m_1',m_A',m_B')$,
the number of choices in steps 1 and 2 is at most
\[
   \binom{s_As_B}{m'_1}{s_A|B'| \choose (1-\lambda )m'_A} {s_B|A'| \choose (1-\lambda )m'_B}\exp\left[-(1 + 8\varepsilon^2) \left(k'_A \log\frac{dn}{2k_A}+k'_B \log\frac{dn}{2k_B} \right)\right].
    \]
\end{claim}

\begin{claim}\label{clm: Janson 1 step 4}
Given $(m_1,m_A,m_B)=(m_1',m_A',m_B')$,
the number of choices in step 3 is at most
\[
    \binom{|B'|}{\lambda m'_A/s_A}^{s_A}\binom{|A'|}{\lambda m'_B/s_B}^{s_B},
    \]
where we take the convention that $\binom{|A'|}{\lambda m'_B/s_B}^{s_B} = 1$ if $s_B = 0$.
\end{claim}
\begin{proofclaim}
    Fix choices for steps 1 and 2 and recall that $d'_u$ is the degree of $u \in [n]$ in the currently constructed graph. Observe that there are precisely
    \[
    \prod_{u \in S_A}\binom{|B'|}{d - d'_u}\prod_{v \in S_B}\binom{|A'|}{d - d'_v}
    \]
    choices in step 3. Moreover, $\sum_{u \in S_A}(d - d'_u) = \lambda m'_A$ and $\sum_{v \in S_B}(d - d'_u) = \lambda m'_B$, so \cref{prop: binom average} gives the desired result.
\end{proofclaim}

Given $(m_1,m_A,m_B)=(m_1',m_A',m_B')$,
let $m' = dn/2 - k_A - k_B - m'_A - m'_B - m'_1$. By Corollary~\ref{cor: max sequence}, the number of choices in step 4 is at most
\begin{equation}
O\left(\binom{|B'|}{m'/|A'|}^{|A'|}\binom{|A'|}{m'/|B'|}^{|B'|}\binom{|A'||B'|}{m'}^{-1}\right). \label{eq:step5}
\end{equation}

On the other hand, similarly to the argument in subsection~\ref{sec:large k near threshold} (right above~\eqref{eq: large k initial ineq}),  given $(m_1,m_A,m_B)=(m_1',m_A',m_B')$, every graph in $\mathcal H$ that can be constructed in the procedure with parameter $(m_1,m_A,m_B)=(m_1',m_A',m_B')$  can be obtained in at least $\binom{m'_A}{(1 - \lambda )m'_A}\binom{m'_B}{(1 - \lambda )m'_B}$ different ways. Therefore, combining \cref{clm: janson 1 step 3},  \cref{clm: Janson 1 step 4}, and~\eqref{eq:step5},
\begin{align*}
|{\mathcal H}| \le &\sum_{m_1'}\max_{(m'_A, m'_B)}\left\{\frac{\binom{s_As_B}{m'_1}\binom{s_A|B'|}{(1 - \lambda )m'_A}\binom{s_B|A'|}{(1 - \lambda )m'_B}\binom{|B'|}{\lambda m'_A/s_A}^{s_A}\binom{|A'|}{\lambda m'_B/s_B}^{s_B}\binom{|B'|}{m'/|A'|}^{|A'|}\binom{|A'|}{m'/|B'|}^{|B'|}}                {\binom{|A'||B'|}{m'}\binom{m_A'}{(1 - \lambda )m_A'}\binom{m'_B}{(1 - \lambda )m'_B}}\right\}\notag\\
&\times\exp\left[-(1 + 8\eps^2)\left(k'_A\log\frac{dn}{2k_A} + k'_B\log\frac{dn}{2k_B}\right) + O(1) \right],
\end{align*} 
where the summation of $m_1'$ is over all nonnegative integers that are at most \newline$\min\{s_As_B, ds_A, ds_B\}$, whereas the maximization above is over all choices of $(m'_A,m'_B)$ satisfying~\eqref{eq:relaxedconditions}. Notice that there are $O(s_As_B)$ choices for $m_1'$ and $s_As_B=\exp(O(k_A + k_B))$. Thus, by~\eqref{eq:insideGAGB} and observing that $s_A,s_B$ are independent of $S_A,S_B$, we immediately obtain that $|\mathcal T_{k_A, r_A, k_B, r_B}|$ is bounded from above by 
\begin{align*}
  \sum_{G_A,G_B}&\max\left\{\frac{\binom{s_As_B}{m'_1}\binom{s_A|B'|}{(1 - \lambda )m'_A}\binom{s_B|A'|}{(1 - \lambda )m'_B}\binom{|B'|}{\lambda m'_A/s_A}^{s_A}\binom{|A'|}{\lambda m'_B/s_B}^{s_B}\binom{|B'|}{m'/|A'|}^{|A'|}\binom{|A'|}{m'/|B'|}^{|B'|}}                {\binom{|A'||B'|}{m'}\binom{m_A'}{(1 - \lambda )m_A'}\binom{m'_B}{(1 - \lambda )m'_B}}\right\}\notag\\
&\times\exp\left[-(1 + 8\eps^2)\left(k'_A\log\frac{dn}{2k_A} + k'_B\log\frac{dn}{2k_B}\right) + O(k_A + k_B) \right],
\end{align*}
where the maximization above is over all choices of $(m'_1,m'_A,m'_B)$ satisfying~\eqref{eq:relaxedconditions}. Thus, by~\eqref{eq:reg bip count2} and~\eqref{eq:defectbounds}, we have that $|\mathcal T_{k_A, r_A, k_B, r_B}|/|\mathcal G_{n,d}^\text{bip}|$ is bounded from above by 
\begin{align}
     \max\left\{\frac{\binom{s_As_B}{m'_1}\binom{s_A|B'|}{(1 - \lambda )m'_A}\binom{s_B|A'|}{(1 - \lambda )m'_B}\binom{|B'|}{\lambda m'_A/s_A}^{s_A}\binom{|A'|}{\lambda m'_B/s_B}^{s_B}\binom{|B'|}{m'/|A'|}^{|A'|}\binom{|A'|}{m'/|B'|}^{|B'|}\binom{n^2/4}{dn/2}}                {\binom{n}{n/2}\binom{|A'||B'|}{m'}\binom{n/2}{d}^{n}\binom{m_A'}{(1 - \lambda )m_A'}\binom{m'_B}{(1 - \lambda )m'_B}}\right\}\notag\\\times\exp\left[(k_A + k_B)\log\frac{n}{d} -7\varepsilon^2\left(k'_A\log\frac{dn}{2k_A} + k'_B\log\frac{dn}{2k_B}\right) + O(k_A + k_B)\right].\label{eq: big equation janson 1}
\end{align}
Simplifying~\eqref{eq: big equation janson 1} we obtain the following bound, which completes the proof of the lemma.
\begin{claim}\label{clm: final counting janson 1}
    \[
        \frac{|\mathcal T_{k_A, r_A, k_B, r_B}|}{|\mathcal G_{n,d}^\text{bip}|} \leq\frac{1}{\binom{n}{n/2}}\exp\left[-\varepsilon^{2}\left(k'_A \log\frac{dn}{2k_A} + k'_B \log\frac{dn}{2k_B}\right)\right].
    \]
\end{claim}
\cref{clm: janson 1 step 3} is proved in~\cref{sec: janson args} and \cref{clm: final counting janson 1} is proved below in this subsection.

 \qed

\medskip
\noindent\textit{Proof of \cref{clm: final counting janson 1}.}
We prove this by evaluating the ratio in~\eqref{eq: big equation janson 1} separately when $s_B = 0$ and $s_B > 0$. 

\begin{claim}\label{clm: s_B = 0}
    Under the assumptions of either \cref{lem: Janson Counting 1} (and taking $\eps^7/\log n < \lambda < 
    \eps^6/\log n$) or \cref{lem: Janson Counting 2} (and taking $\lambda = 0$), 
    \begin{align*}
        &\max\left\{\frac{\binom{s_Ab}{(1 - \lambda )m'_A}\binom{b}{\lambda m'_A/s_A}^{s_A}\binom{b}{m'/|A'|}^{|A'|}\binom{|A'|}{m'/b}^{b}\binom{n^2/4}{dn/2}}                {\binom{|A'||B'|}{m'}\binom{n/2}{d}^{n}\binom{m_A'}{(1 - \lambda )m_A'}}\right\}\\
        &\leq\begin{cases}
            \exp\left[-(k_A + k_B)\log\frac{n}{d} + O\left(\eps^3k'_A\log\frac{dn}{2k_A} + \eps^3k'_B\log\frac{dn}{2k_B}\right)\right] &\text{when}\quad \lambda > 0\\
            \exp\left[-(k_A + k_B)\log\frac{n}{d} + 6\left(k'_A\log\frac{dn}{2k_A} + k'_B\log\frac{dn}{2k_B}\right)\right] & \text{when} \quad \lambda = 0
        \end{cases} .
    \end{align*}
\end{claim}

\begin{claim}\label{clm: s_B > 0}
When $s_B > 0$,
    \begin{align*}
        \max&\left\{\frac{\binom{s_As_B}{m'_1}\binom{s_A|B'|}{(1 - \lambda )m'_A}\binom{s_B|A'|}{(1 - \lambda )m'_B}\binom{|B'|}{\lambda m'_A/s_A}^{s_A}\binom{|A'|}{\lambda m'_B/s_B}^{s_B}\binom{|B'|}{m'/|A'|}^{|A'|}\binom{|A'|}{m'/|B'|}^{|B'|}\binom{n^2/4}{dn/2}}                {\binom{|A'||B'|}{m'}\binom{n/2}{d}^{n}\binom{m_A'}{(1 - \lambda )m_A'}\binom{m'_B}{(1 - \lambda )m'_B}}\right\}\\ &\leq \exp\left[-(k_A + k_B)\log\frac{n}{d} +O\left(\eps^3k'_A\log\frac{dn}{2k_A} + \eps^3k'_B\log\frac{dn}{2k_B}\right)\right].
    \end{align*}
\end{claim}

If $s_B > 0$, the desired result follows by combining or \cref{clm: s_B > 0} with the exponent in~\eqref{eq: big equation janson 1}.
If $s_B = 0$, the maximization term in~\eqref{eq: big equation janson 1} is equivalent to the maximization term in \cref{clm: s_B = 0}, and so combining \cref{clm: s_B = 0} and the exponent in~\eqref{eq: big equation janson 1} gives the desired result.

\cref{clm: s_B = 0} and \cref{clm: s_B > 0} are proved in \cref{sec: Janson simplifications}.
\qed

\subsection{Proof of \cref{lem: Janson Counting 2}}\label{sec: Janson counting 2}

As before,     let $a = |A|$, $b = |B|$, $k'_A = k_{A, r_A - 1}$, and $k'_B = k_{B, r_B - 1}$ and we may assume that $(A,B)$ satisfies \eqref{eq:AB}. We again have $k'_A > 0$ by the same argument as in the beginning of the proof of \cref{lem: Janson Counting 1}.
We bound $|\mathcal T_{k_A, r_A, k_B, r_B}|$ in the same approach as in the proof of \cref{lem: Janson Counting 1}, by first specifying a triangle-free $(k_A, r_A, A, D_{A})$-graph $G_A$ and a triangle-free $(k_B, r_B, B, D_{B})$-graph $G_B$, and then estimating the number of ways to complete $(G_A,G_B)$ to a $d$-regular triangle-free graph with $G_A$ and $G_B$ being the defect graphs.
By~\cref{cor: counting defects},
 the number of choices for $(G_A,G_B)$ is bounded in~\eqref{eq:defectbounds}.
Fix $(G_A,G_B)$ and an arbitrary $D_A$-torso $T_A$ of $G_A$. Let $S_A \subseteq A$ be chosen arbitrarily such that $|S_A| = \min\{\eps^{1/2}2k'_A, \eps^{1/2}n\}$ and $|E(T_A[S_A])| \geq \eps|E(T_A)|$. 
    \begin{claim}\label{clm: dense set}
        There exists a set $S_A \subseteq A$ that satisfies the constraints above.
    \end{claim}
    \begin{proofclaim}
    Let $A^*$ denote the set of vertices in $A$ incident to an edge in $T_A$. Let $a^* = |A^*|$ and observe that $a^* \leq \min\{2k'_A, n\}$. Consider choosing a set $S'_A \subseteq A^*$ with $\eps^{1/2} a^*$ vertices uniformly at random. Then, (assuming that $\eps^{1/2} a^*$ is an integer) each edge is included in $T_A[S'_A]$ with probability
    \begin{align*}
        \binom{a^*-2}{\eps^{1/2} a^* - 2}/\binom{a^*}{\eps^{1/2} a^*} &\geq (\eps^{1/2} a^*)(\eps^{1/2} a^*-1)/(a^* - \eps^{1/2} a^* + 2)(a^* - \eps^{1/2} a^* + 1) \\
        &= \eps \frac{1 - 1/\eps^{1/2} a^*}{(1 - \eps^{1/2} + 2/a^*)(1 - \eps^{1/2} + 1/a^*) } \geq \eps .
    \end{align*}

    Thus, the expected number of edges in $T_A[S'_A]$ is at least $\eps|E(T_A)|$, so there is a valid choice of $S_A$.
    \end{proofclaim}

Let $s_A=|S_A|$. Let $A' = A \setminus S_A$. Notice that $S_A$ depends on $G_A$.  However, $s_A$ does not. 
Notice also that $S_A$ does not necessarily contain all the vertices in $T_A$ with positive degrees, which is different from the case in the proof of \cref{lem: Janson Counting 1}.

Let $\mathcal H_{G_A, G_B, T_A, S_A}$ be the set of $d$-regular graphs $H$ where $H[A] = G_A$, $H[B] = G_B$, and no edge in $T_A[S_A]$  is contained in a triangle in $H$. Then,
\begin{equation}
|\mathcal T_{k_A, r_A, k_B, r_B}|\le \sum_{G_A,G_B} \min_{T_A, S_A}|{\mathcal H}_{G_A,G_B,T_A,S_A}|.\label{eq:insideGAGB2}
\end{equation}

Next, we fix $(G_A,G_B,T_A,S_A)$ and we upper bound $|\mathcal H|$, where ${\mathcal H}={\mathcal H}_{G_A,G_B,T_A,S_A}$,
by counting graphs that can be constructed following the procedure below. 

    \begin{itemize}
        \item[1.](local $(S_A, B)$ sprinkling) Let $m_A = ds_A  - \sum_{u \in S_A}d_{G_A}(u)$. Include $m_A$ edges from $S_A \times B$ into $H$ such that no triangle in $H$ contains an edge in $T_A[S_A]$ and no vertex in $H$ has degree more than $d$.
        \item[2.](degree completion) Complete $H$ to be a (simple) $d$-regular graph on $[n]$ by including $dn/2 - k_A - k_B - m_A$ edges in $A'\times B$.
    \end{itemize}

    Once $S_A$ has been chosen, this procedure is identical to the procedure given in \cref{sec: Janson Counting 1} when $S_B$ is empty.
    
    Observe that every graph in ${\mathcal H}$ can be constructed by the procedure above (in multiple ways). Unlike the proof of \cref{lem: Janson Counting 1}, though, there is not any large over-counting. In fact, we can obtain a sufficient bound on $|{\mathcal H}|$ by upper bounding the number of choices in each of the steps in the procedure, and by lower bounding the number of ways a graph in ${\mathcal H}$ can be constructed by the procedure by $1$.

As before, it is convenient to relax the condition $m_A = ds_A - \sum_{u \in S_A} d_{G_A}(u)$ to 
\begin{equation}\label{eq: relaxed condition}
    ds_A - 2k_A \leq m'_A \leq ds_A,
\end{equation}
so that the bounds on $m_A'$ no longer depend on $G_A$.
 Next, we estimate the number of choices in each step given parameter $m'_A$ for $m_A$ that satisfy~\eqref{eq: relaxed condition}.

\begin{claim}\label{clm: janson 2 step 2}
Given $m_A = m'_A$, the number of choices in step 1 is at most
\[
     {s_Ab \choose m_A'}\exp\left[-\frac{1}{2\eps^2} k'_A \log\frac{dn}{2k_A}\right].
    \]
\end{claim}


Given $m_A = m'_A$, let $m' = dn/2 - k_A - k_B - m'_A$. By Corollary~\ref{cor: max sequence}, the number of choices in step 2 is at most
\begin{equation}
O\left(\binom{b}{m'/|A'|}^{|A'|}\binom{|A'|}{m'/b}^{b}\binom{|A'|b}{m'}^{-1}\right). \label{eq: jan 2 step5}
\end{equation}

Therefore, combining~\Cref{clm: janson 2 step 2}  and~\eqref{eq: jan 2 step5},
\begin{align*}
|{\mathcal H}| \le &\exp\left[-\eps^{-2}k'_A\log\frac{dn}{2k_A} + O(1) \right]\max\left\{\frac{\binom{s_Ab}{m'_A}\binom{b}{m'/|A'|}^{|A'|}\binom{|A'|}{m'/b}^{b}}                {\binom{|A'|b}{m'}}\right\},
\end{align*} 
where the maximization is over all integers $m'_A$ that satisfy~\eqref{eq: relaxed condition}. Thus, by~\eqref{eq:insideGAGB2}, we immediately obtain that 

\begin{align*}
|\mathcal T_{k_A, r_A, k_B, r_B}| \le &\sum_{G_A, G_B}\max\left\{\frac{\binom{s_Ab}{m'_A}\binom{b}{m'/|A'|}^{|A'|}\binom{|A'|}{m'/b}^{b}}                {\binom{|A'|b}{m'}}\right\} \exp\left[-\frac{1}{2\eps^2}k'_A\log\frac{dn}{2k_A}  + O(1) \right],
\end{align*} 
where the maximization is again over all choices of $m'_A$ that satisfy~\eqref{eq: relaxed condition}. Recall that~\eqref{eq:defectbounds} gives an upper bound for the number of choices of $G_A$ and $G_B$.
Observe that 
\[
-\frac{1}{2\eps^{2}}k'_A\log\frac{dn}{2k_A} + (1 + \eps^2)\left(k'_A\log\frac{dn}{2k_A} + k'_B\log\frac{dn}{2k_B}\right) \leq -\frac{1}{3\eps^2}\left(k'_A\log\frac{dn}{2k_A}+ k'_B\log\frac{dn}{2k_B}\right),
\]
since $k'_A\log\frac{dn}{2k_A} \geq k'_B\log\frac{dn}{2k_B}$. Thus, we have that 
\begin{align}
|\mathcal T_{k_A, r_A, k_B, r_B}| \le &\max\left\{\frac{\binom{s_Ab}{m'_A}\binom{b}{m'/|A'|}^{|A'|}\binom{|A'|}{m'/b}^{b}}                {\binom{|A'|b}{m'}}\right\}\notag\\& \times\exp\left[(k_A + k_B)\log\frac
n d -\frac{1}{3\eps^2}\left(k'_A\log\frac{dn}{2k_A} + k'_B\log\frac{dn}{2k_B}\right) + O(1) \right].\label{eq:janson 2 ugly}
\end{align} 
Simplifying~\eqref{eq:janson 2 ugly} we obtain the following bound, which completes the proof.

\begin{claim}\label{clm: final counting janson 2}
    \[
        \frac{|\mathcal T_{k_A, r_A, k_B, r_B}|}{|\mathcal G_{n,d}^\text{bip}|} \leq\frac{1}{\binom{n}{n/2}}\exp\left[-\eps^{-1}\left(k'_A \log\frac{dn}{2k_A} + k'_B\log\frac{dn}{2k_B}\right)\right]
    \]
\end{claim}
\cref{clm: janson 2 step 2} is proved in~\cref{sec: janson args} and \cref{clm: final counting janson 2} is proved below.
\qed
\medskip

    \medskip
    \noindent\textit{Proof of \cref{clm: final counting janson 2}.}
    Observe that by~\eqref{eq:reg bip count2}, 
    \begin{align}
        \frac{|\mathcal T_{k_A, r_A, k_B, r_B}|}{|\mathcal G_{n,d}^\text{bip}|} \leq &\max\left\{\frac{\binom{s_Ab}{m'_A}\binom{b}{m'/|A'|}^{|A'|}\binom{|A'|}{m'/b}^{b}\binom{n^2/4}{dn/2}}                {\binom{n}{n/2}\binom{n/2}{d}^n\binom{|A'|b}{m'}}\right\}\notag\\& \times\exp\left[(k_A + k_B)\log\frac
n d -\frac{1}{3\eps^2}\left(k'_A\log\frac{dn}{2k_A} + k'_B\log\frac{dn}{2k_B}\right) + O(1) \right].
    \end{align}
    Moreover, taking $\lambda = 0$, the ratio of binomial coefficients in this maximization is equivalent to the ratio of binomial coefficients in \cref{clm: s_B = 0}. The desired result then follows from \cref{clm: s_B = 0}.
    \qed

    \subsection{Proofs of \cref{clm: janson 1 step 3} and \cref{clm: janson 2 step 2}}\label{sec: janson args}

We reduce these problems to bounding the probability that randomly chosen cross-edges do not form triangles with the fixed defect edges. Since the torsos we are considering are large here, these defect graphs are fairly regular. As such, the correlation inequalities in the theorem below, established by Janson, {\L}uczak, and Ruci\'nski in \cite{janson1988exponential}, give meaningful bounds on the probability that no triangle is created when adding cross-edges. The calculations in this subsection are very similar to the calculations in~subsection 6 of \cite{osthus2003densities}, but we consider only cross-edges incident to one of $S_A$ and $S_B$ while all cross-edges are considered in \cite{osthus2003densities}.

\begin{thm}\label{thm: Janson}\cite{janson1988exponential}
    Let $\{I_j\}_{j \in \mathcal J}$ be independent 0-1 variables where $\mathcal J$ is an arbitrary index set. For every subset $\alpha$ of $\mathcal J$, let $I_\alpha = \prod_{j \in \alpha}I_j$. Let $\mathcal S$ be a collection of subsets of $\mathcal J$ and set $X = \sum_{\alpha \in \mathcal S}I_\alpha$. Let $\mu = \mathbb E[X]$ and $\Delta = \sum_{\alpha \sim \beta} \mathbb E[I_\alpha I_\beta]$, where the sum is over all ordered pairs $(\alpha, \beta)$ of elements of $\mathcal S$ where $\alpha \cap \beta \neq \emptyset$ and $\alpha \neq \beta$. Then,
\begin{equation}\label{eq: Janson 1}
    \mathbb P[X = 0] \leq \exp\left[ -\mu + \frac{\Delta}{2} \right]
\end{equation}
and if $\Delta > \mu$, 
\begin{equation}\label{eq: Janson 2}
    \mathbb P[X = 0] \leq \exp\left[ -\frac{\mu^2}{2\Delta} \right].
\end{equation}
\end{thm}
\medskip
\noindent\textit{Proof of \cref{clm: janson 1 step 3}.}
We bound the number of choices in these steps by including edges uniformly at random and bounding the probability that the resulting graph does not contain triangles that include edges in $T_A \cup T_B$. 

Let $H_A$ be a random bipartite graph with bipartition $(S_A, B')$ where $(1 - \lambda )m'_A$ edges are uniformly chosen.  Analogously, let  $H_B$ be a random bipartite graph with bipartition $(S_B, A')$ where $(1 - \lambda )m'_B$ edges are uniformly chosen. Let  $H_1$ be a random bipartite graph with bipartition $(S_A, S_B)$ where $m'_1$ edges are uniformly chosen. Moreover $H_A$, $H_B$ and $H_1$ are mutually independent. 
Let ${\mathcal E}$ denote the event that $H_A\cup H_B\cup H_1\cup T_A\cup T_B$ is triangle-free. It follows immediately that the number of choices in steps 1 and 2 is at most (since we ignore the condition that $H_A\cup H_B\cup H_1\cup G_A\cup G_B$ has maximum degree at most $d$)
\[
\binom{s_As_B}{m_1'}\binom{s_A|B'|}{(1 - \lambda )m'_A}\binom{s_b|A'|}{(1-\lambda )m'_B}\P({\mathcal E}).
\]
It only remains to estimate $\P({\mathcal E})$. Let $X_1$, $X_A$, and $X_B$ denote the number of triangles in $H_1 \cup T_A \cup T_B$, $H_A \cup T_A$, and $H_B \cup T_B$ respectively. By the independence of $H_1$, $H_A$ and $H_B$, we have that $X_1$, $X_A$ and $X_B$ are independent. Hence,
\[
\P({\mathcal E})=\P(X_1=0)\P(X_A=0)\P(X_B=0).
\]
We prove the following bounds on these probabilities depending on the range of $m'_1$. Recall that $s_A > 0$.

\begin{claim}\label{clm: big m1}
If $s_A, s_B > 0$ and $m'_1 \geq \varepsilon^{3/2}ds_\sigma$ for some $\sigma \in \{A,B\}$, then
    \[
    \P[X_1= 0] \leq \exp\left[-\eps^{-1/2}(1 - 3\eps)k'_\sigma\log\frac{dn}{2k_\sigma} + O(1)\right].
    \]
\end{claim}
\begin{claim}\label{clm: small m1}
If $m'_\sigma \geq (1 - 2\eps^{3/2})ds_\sigma$ and $s_\sigma > 0$ for some $\sigma \in \{A,B\}$, then
    \[
    \P[X_\sigma = 0] \leq \exp\left[-(1 + 9\eps^2)k'_\sigma\log\frac{dn}{2k_\sigma}+ O(1)\right].
    \]
\end{claim}

Suppose $s_A, s_B > 0$. Recall that this implies that $s_A = k'_A$ and $s_B = k'_B$ by~\eqref{eq:sizeSA}. We use these claims to show that 
\begin{equation}\label{eq: desired prob}
    \P[X_1 = 0]\P[X_A = 0]\P[X_B = 0] \leq \exp\left[-(1 + 8\eps^2)\left(k'_A\log\frac{dn}{2k_A} + k'_B\log\frac{dn}{2k_B}\right)\right]
\end{equation}
for all possible values of $m'_1$, $m'_A$, and $m'_B$. 

Suppose $\sigma,\tau \in \{A,B\}$ such that $\sigma\neq \tau$ and 
 $k'_\sigma \geq k'_\tau$. We consider the following three cases in terms of $m'_1$.

 {\em Case 1}: $m_1'>\eps^{3/2} ds_{\sigma}$. In this case,  \cref{clm: big m1} gives that 
\begin{align*}
    \P[X_1 = 0] &\leq \exp\left[-\eps^{-1/2}(1 - 3\eps)k'_\sigma\log\frac{dn}{2k_\sigma}+ O(1)\right]\\
    &\leq \exp\left[-\frac{1 + o(1)}{3\eps^{1/2}}\left( k'_\sigma\log\frac{dn}{2k_\sigma} + k'_\tau\log\frac{dn}{2k_\tau} \right)+ O(1)\right]\\
    &\leq \exp\left[-(1 + 8\eps^2 )\left(k'_A\log\frac{dn}{2k_A} + k'_B\log\frac{dn}{2k_B}\right)\right],
\end{align*}
where the second inequality above holds since
$k'_\sigma\log\frac{dn}{2k_\sigma} \geq k'_\tau\log\frac{dn}{2k_\tau}(1 + o(1))$ by \newline\cref{clm: k'Ak'B}.

{\em Case 2}: $\eps^{3/2} ds_{\tau}< m_1'\le \eps^{3/2} ds_{\sigma}$. In this case, we use \cref{clm: big m1} to bound $\P[X_1 = 0]$ and \cref{clm: small m1} to bound $\P[X_\sigma = 0]$. We first justify that the assumption of \cref{clm: small m1} is satisfied for $\sigma$. Recall that $k'_\sigma = \omega( k_\sigma/\log\frac{dn}{2k_\sigma}) $ by \eqref{eq:k'bound} and $s_{\sigma}=2k'_\sigma$. Then,  $\frac{2k_\sigma}{ds_\sigma} = \frac{k_\sigma}{dk'_\sigma} = o(1)$, and so by~\eqref{eq:relaxedconditions}, 
\begin{equation}
m'_\sigma \geq ds_\sigma\left(1 - \frac{m_1'}{ds_\sigma} - o(1)\right).\label{eq:mprimebound}
\end{equation}
Since $m_1' \leq \eps^{3/2}ds_\sigma$ in this case, we have $m_\sigma' \geq ds_\sigma(1 - 2\eps^{3/2})$, which verifies the assumptions in \cref{clm: small m1}. Hence,

\begin{align*}
    \P[X_1 = 0]\P[X_\sigma = 0] &\leq \exp\left[-(1 + 9\eps^2)k'_\sigma\log\frac{dn}{2k_\sigma} - \eps^{-1/2}(1 - 3\eps)k'_\tau\log\frac{dn}{2k_\tau}+ O(1)\right]\\
    &\leq \exp\left[-(1 + 8\eps^2)\left(k'_A\log\frac{dn}{2k_A} + k'_B\log\frac{dn}{2k_B}\right)\right].
\end{align*}

{\em Case 3}: $m_1'\le \eps^{3/2} ds_{\tau}$. As in case 2, the assumption $m_1'\le \eps^{3/2} ds_{\tau}$ implies that the hypotheses of \cref{clm: small m1} hold for both $\sigma$ and $\tau$, i.e.\ $A$ and $B$. Thus, we may use \cref{clm: small m1} to bound both $\P[X_A = 0]$ and $\P[X_B = 0]$. This gives 
\begin{align*}
    \P[X_A= 0]\P[X_B = 0] &\leq  \exp\left[-(1 + 9\eps^2)\left(k'_A\log\frac{dn}{2k_A} + k'_B\log\frac{dn}{2k_B}\right)+ O(1)\right]\\
    &= \exp\left[-(1 + 8\eps^2)\left(k'_A\log\frac{dn}{2k_A} + k'_B\log\frac{dn}{2k_B}\right)\right],
\end{align*}
as desired. By~\eqref{eq: desired prob} this completes the proof for \cref{clm: janson 1 step 3} when $s_B>0$.

 Assume $s_B = 0$. Then, by~\eqref{eq:relaxedconditions}, $m'_1=0$ and by the same argument leading to~\eqref{eq:mprimebound}, $m'_A = ds_A(1 +o(1))$. By \cref{clm: small m1} for $\sigma=A$,  
\begin{equation}
\P[X_A = 0] \leq \exp\left[-(1 + 9\eps^2)k'_A\log\frac{dn}{2k_A} + O(1)\right].\label{eq:XA0}
\end{equation}
We consider two cases. 
If $ k'_B = 0$, then the probability bound above immediately implies that
\[
\P[X_A = 0] \leq \exp\left[-(1 + 8\eps^2)\left(k'_A\log\frac{dn}{2k_A} + k'_B\log\frac{dn}{2k_B}\right)\right],
\]
which yields~\eqref{eq: desired prob} as desired. In the other case where $ k'_B>0$, we have $k'_A\log\frac{dn}{2k_A} \geq \frac{2}{\eps^2}k'_B\log\frac{dn}{2k_B}$ by the assumption of \cref{lem: Janson Counting 1}. By~\eqref{eq:XA0}, 
\[
\P[X_A = 0]\leq \exp\left[-(1 + 8\eps^2)\left(k'_A\log\frac{dn}{2k_A} + k'_B\log\frac{dn}{2k_B}\right)\right].
\]
\cref{clm: big m1} and \cref{clm: small m1} are proved later in this subsection, which completes this proof.
\qed

Recall that the probability spaces in \cref{clm: big m1} and \cref{clm: small m1} are $H_1, H_A$ and $H_B$. We introduce $H_1',H_A'$ and $H_B'$, the ``binomial version'' of them where edges appear independently, which is convenient when applying Janson's inequalities, and with properly chosen parameters they approximate $H_1,H_A$ and $H_B$ well.  To do this,  set 
\begin{align*}
p_1 &= \frac{m_1'}{s_As_B},\quad \text{if $s_A, s_B > 0$}\\
p_\sigma &= \frac{(1 - \lambda)m'_\sigma}{s_\sigma|\tau'|}\quad \text{if $s_\sigma > 0$}.
\end{align*}
Let $H_1'$ be a random graph obtained by including each edge in $S_A \times S_B$ independently with probability $p_1$, and $H_\sigma'$ be a random graph obtained by including each edge in $S_\sigma \times \tau'$ independently with probability $p_\sigma$.

The following claim allows us to translate probability bounds in the random graphs $H_1', H_A', H_B'$ to $H_1,H_A$ and $H_B$.


\begin{claim}\label{clm: m to p}
    For $C \in \{1,A,B\}$,
    \[
    \P_{H_C'}[X_C = 0] =\Omega(\P_{H_C}[X_C = 0]).
    \]
\end{claim}
\begin{proof}
    Let $(U,V)$ be the bipartition of $H_C$. Since the number of edges in $H'_C$ is a binomial random variable with expectation $m'_C$, we have 
    \[
    \P_{H'_C}[|E(H'_C)|\leq m'_C] = \Omega(1).
    \] 
    Thus, 
    \begin{align*}
        \P_{H'_C}[X_C = 0] &\geq \P_{H'_C}[X_C = 0 \mathrel ||E(H'_C)| \leq m'_C] \cdot\P_{H'_C}[|E(H'_C)| \leq m'_C]\\
        &= \Omega(\P_{H'_C}[X_C = 0 \mathrel ||E(H'_C)| \leq m'_C]).
    \end{align*}
We conclude this proof by showing that for any $r \geq 1$ where $\P[|E(H'_C)| = r] \neq 0$, 
    \[
        \P_{H'_C}[X_C = 0 \mathrel | |E(H'_C)| = r] \leq \P_{H'_C}[X_C = 0 \mathrel||E(H'_C)| = r-1],
    \]
    since it would follow that 
    \begin{align*}
        \P_{H'_C}[X_C = 0\mathrel | |E(H'_C)| \leq m'_C] &= \Omega(\P_{H'_C}[X_C = 0\mathrel | |E(H'_C)| = m'_C])\\
        &= \Omega(\P_{H_C}[X_C = 0]).
    \end{align*}
    For any such $r$ (or $r = 0$), let $\mathcal G_r$ denote the total number of bipartite graphs on bipartition $(U,V)$ with $r$ edges and let $\mathcal T_r \subseteq \mathcal G_r$ be the set of graphs $G \in \mathcal G_r$ where $G \cup T_A \cup T_B$ is triangle-free. Let $t_r = |\mathcal T_r|$ and $g_r = |\mathcal G_r|$. It suffices to show that $t_r/g_r \leq t_{r-1}/g_{r-1}$. That is,
    \[
    \frac{t_{r}}{t_{r-1}} \leq \frac{g_{r}}{g_{r-1}} = \frac{\binom{|U||V|}{r}}{\binom{|U||V|}{r-1}} = \frac{|U||V| - r + 1}{r}.
    \]
    We prove this using a simple double counting argument. Consider a relation $\sim $ on $\mathcal T_r \times \mathcal T_{r-1}$ where $G \sim G'$ if $G' = G -e$ for some $e \in E(G)$. Let $\mathscr G$ be the bipartite graph with bipartition $(T_{r} , \mathcal T_{r-1})$ and edges defined by $\sim$. Let $d_r$ denote the minimum degree of graphs in $\mathcal T_{r}$ and $d_{r-1}$ denote the maximum degree of graphs in $\mathcal T_{r-1}$ in $\mathscr G$. Then, $d_rt_r \leq d_{r-1}t_{r-1}$, so it suffices to show that 
    \[
    \frac{d_{r-1}}{d_{r}} \leq \frac{|U||V| - r+1}{r}.
    \]
    Note that deleting any edge of a graph in $\mathcal T_{r}$ will result in a graph in $\mathcal T'_{r-1}$, so we have 
    \[
    d_r \geq r.
    \]
    Also, for any graph $T \in \mathcal T'_{H, V, m}$, at most $|U||V| - r + 1$ edges can be added to $T$ to obtain a bipartite graph with bipartition $(U, V)$. That is,  $d_{r-1} \leq |U||V| - r+1$. This completes the proof.
\end{proof}

\noindent\textit{Proof of \cref{clm: big m1}. } 
By \cref{clm: m to p}, it suffices to show that
    \[
    \P_{H_\sigma'}[X_\sigma = 0] \leq \exp\left[-\eps^{-1}(1 - 3\eps)k'_\sigma\log\frac{dn}{2k_\sigma}\right].
    \]

We apply \cref{thm: Janson} to bound $\P_{H_1'}[X_1 = 0]$. Here, $\mathcal J$ is $S_A \times S_B$ and and $\mathcal S$ is the set of pairs of edges $(xz,yz)$ where $x$ and $y$ are adjacent in $T_A$ or $T_B$. Let $\mu_1 = \mathbb E_{H'_1}[X_1]$ and $\Delta_1$ be as described in \cref{thm: Janson}. 

Let $\tau\in\{A,B\}\setminus \{\sigma\}$. Observe that 
    \begin{align}
        & \mu_1 \geq (1 - \eps)(k'_As_Bp^2_1 + k'_Bs_Ap_1^2) = \frac{2(1 - \eps)m_1'^2}{s_As_B},\label{eq: mu1}\\
        &\Delta_1 \leq 2k'_AD_As_Bp_1^3 + 2k'_BD_Bs_Ap_1^3 + 8k'_Ak'_Bp_1^3 \leq \left(\frac{\eps^4dm_1'^3}{s_A^2s_B^2\log\frac{dn}{2k_A}} + \frac{\eps^4dm_1'^3}{s_A^2s_B^2\log\frac{dn}{2k_B}}\right)(1 + o(1))\label{eq: Delta1}.
    \end{align}
    Then,
    \[
    \mu_1/2 \geq \frac{(1 - \eps)\eps^{3}d^2s_\sigma}{s_\tau}\geq (1 - \eps)(1 + \eps)\eps^{3}\frac{3}{4}\frac{n}{s_\tau}s_\sigma\log n\geq \frac{(1 - \eps^2)\eps^{3}}{\eps^4}s_\sigma\log n \geq \eps^{-{1/2}}k'_\sigma\log\frac{dn}{2k_\sigma}
    \]
    where the first inequality holds by~\eqref{eq: mu1}, the second holds because $d \geq (1 + \eps)\sqrt 3/2\sqrt{n\log n}$, the third holds because $s_\tau \leq 2\eps^4 n$, and the fourth holds because $n^{1/\eps} \gg dn$. Hence, \eqref{eq: Janson 1} gives the desired result if $\mu_1/2 \geq \Delta_1$. 
    Likewise, notice that 
    \begin{equation}\label{eq: mu1/delta1}
            \frac{\mu_1^2}{2\Delta_1} \geq \frac{2(1 - \eps)^2s_\sigma\log\frac{dn}{2k_A}\log\frac{dn}{2k_B}}{\eps(1 + \eps^2)\left(\log\frac{dn}{2k_A} + \log\frac{dn}{2k_B}\right)}.
    \end{equation}
    Note that $\log\frac{dn}{2k_\tau} \geq \eps\log\frac{dn}{2k_\sigma}$, as otherwise $k_\tau > (dn)^{1 - \eps}k_\sigma$, contradicting that $k_A + k_B \leq n\log d$. Then, $\eps\left(\log\frac{dn}{2k_A} + \log\frac{dn}{2k_B}\right) \leq(1  + \eps)\log\frac{dn}{2k_\tau}$. That is, the right hand side of \eqref{eq: mu1/delta1} is bounded from below by 
    \[
    \eps^{-1/2}(1 - 3\eps)k'_\sigma\log\frac{dn}{2k_\sigma}.
    \]
    Thus, \eqref{eq: Janson 2} gives the desired bound, completing the proof.
\qed

\noindent\textit{Proof of \cref{clm: small m1}.} By \cref{clm: m to p}, it suffices to show that 
\[
\P_{H_\sigma'}[X_\sigma = 0] \leq \exp\left[-(1 + 9\eps^2)k'_\sigma\log\frac{dn}{2k_\sigma}\right].
\]
    
    Let $\tau \in \{A, B\} \setminus \sigma$. This proof follows a similar structure to the proof of \cref{clm: big m1}, where we again use \cref{thm: Janson} to bound $\P_{H'_\sigma}[X_\sigma = 0]$. Here, $\mathcal J$ is $S_\sigma \times \tau'$ and and $\mathcal S$ is the set of pairs of edges $(xz,yz)$ where $x$ and $y$ are adjacent in $T_\sigma$. Let $\mu = \mathbb E_{H'_\sigma}[X_1]$ and $\Delta$ be as described in \cref{thm: Janson}. 
    Recall that $k'_\sigma = s_\sigma$ and $|\tau'| = (n/2)(1 + O(\eps^2))$. We emphasize here that $O(\eps^2)$ denotes a quantity whose absolute value is at most $C\eps^2$ for some absolute constant $C>0$ independent of $n$ or $\eps$. Observe that 
    \begin{align}
        \mu &\geq (1 - \varepsilon)|\tau'|k'_\sigma p_\sigma^2 = \frac{(1 - \varepsilon)(1 - \lambda )^2m_\sigma'^2|\tau'|k'_\sigma}{(s_\sigma|\tau'|)^2} \geq \frac{2(1 - \varepsilon)(1 + O(\eps^{3/2}))d^2}{n}k'_\sigma\label{eq: Janson 1 mu}
    \end{align}
     since $\lambda  \leq \varepsilon^2$. Likewise, 
    \begin{align}
        \Delta &= \sum_{xy, yz \in E(T_\sigma)}\sum_{w \in \tau'} p_\sigma^3 \leq 2k'_\sigma D_{\sigma}|\tau'|p_\sigma^3 = \frac{2(1 - \lambda )^3k'_\sigma D_{\sigma}|\tau'|m_\sigma'^3}{(s_\sigma|\tau'|)^3}  \leq 8(1 + O(\eps^2))\frac{k'_\sigma D_{\sigma}d^3}{n^2}\label{eq: Janson 1 Del}
    \end{align}
    If $\mu < \Delta$, (\ref{eq: Janson 2}), (\ref{eq: Janson 1 mu}), and (\ref{eq: Janson 1 Del}) give
    \begin{align*}
        \mathbb P[X_\sigma = 0] &\leq \exp\left[ -\frac{\mu^2}{2\Delta} \right] \leq \exp\left[ -\frac{(1 - \varepsilon)^2(1 + O(\varepsilon^{3/2}))dk'_\sigma}{4(1 + O(\varepsilon^2))D_{\sigma}} (1 + o(1))\right]\\ 
        &= \exp\left[ -\frac{(1 - \varepsilon)^2(1 + O(\varepsilon^{3/2}))}{\varepsilon^{4}(1 + O(\varepsilon^2))}k'_\sigma\log \frac{dn}{2k_\sigma} \right].\\
        &\leq \exp\left[-2k'_\sigma\log\frac{dn}{2k_\sigma}\right]. 
    \end{align*}
    Hence, it suffices to consider $\mu \geq \Delta$. Suppose $k_\sigma\geq dn^{1 - \frac{3}{4}}/2$. That is, $n^{\frac{3}{4}} \geq \frac{dn}{2k_\sigma}$. Then, by (\ref{eq: Janson 1 mu}),
    \begin{align*}
        \mu/2 &\geq \frac{(1 - \varepsilon)(1 + O(\varepsilon^{3/2}))d^2k'_\sigma}{n}\\
        &\geq (1 + \varepsilon)^2(1 - \varepsilon)(1 + O(\varepsilon^{3/2}))k'_\sigma\log n^{3/4}\\
        &\geq (1 + 10\varepsilon^2)k'_\sigma\log\frac{dn}{2k_\sigma}
    \end{align*}
    for sufficiently small $\varepsilon > 0$. Thus, (\ref{eq: Janson 1}) gives the desired result, so assume $k_\sigma < dn^{1 - \frac{3}{4}}/2$. Here, it will be the case that $\Delta$ is much smaller than $\mu$, so we may still use (\ref{eq: Janson 1}) to obtain the desired result. Specifically, when $d\leq 3(1 + \varepsilon)\sqrt {n\log n}$, (\ref{eq: Janson 1 mu}) and (\ref{eq: Janson 1 Del}) give 
    \begin{align}
        \frac{\Delta}{\mu} &\leq \frac{4(1 + O(\varepsilon^{3/2}))D_{\sigma}d}{(1 - \varepsilon)n}\notag\\
        &= \frac{\varepsilon^4(1 + O(\varepsilon^{3/2}))d^2}{(1 - \varepsilon)n\log\frac{dn}{2k_\sigma}}\notag\\
        &\leq \frac{9\varepsilon^4(1 + O(\varepsilon^{3/2}))}{(1 - \varepsilon)}\frac{\log n}{\log\frac{dn}{2k_\sigma}}\notag\\
        &\leq 18\varepsilon^3. \label{eq: small Delta}
    \end{align}
    Then, applying (\ref{eq: small Delta}) in the second line below, (\ref{eq: Janson 1 mu}) gives
    \begin{align*}
        -\mu + \Delta/2 &= -\mu(1 - \Delta/(2\mu))\\
        &\leq -\frac{2(1 - \varepsilon)(1 + O(\varepsilon^{3/2}))d^2}{n}k'_\sigma\\
        &\leq -(1 - \varepsilon)(1 + O(\varepsilon^{3/2}))(1 + \varepsilon)^2k'_\sigma\log n^{3/2} \\
        &\leq -(1 + \varepsilon/2)(1 + O(\varepsilon^{3/2}))k'_\sigma\log n^{3(1 + \varepsilon/2)/2}  \\
        &\leq -(1 + 10\varepsilon^2)k'_\sigma\log\frac{dn}{2k_\sigma}
    \end{align*}
    where the final inequality holds considering $d \leq 3(1 + \varepsilon)\sqrt{n\log n}$, so $dn = O(n^{3/2}\log n) \leq n^{3(1 + \varepsilon/2)/2}$. Thus, (\ref{eq: Janson 1}) again gives the desired result. It remains to consider $d \geq (1 + \varepsilon)3\sqrt{n \log n}$. In this case, (\ref{eq: Janson 1 mu}) gives 
    \begin{align*}
        \mu/2 &\geq 3(1 - \varepsilon)(1 + O(\varepsilon^{3/2}))k'_\sigma\log n^{9/4}(1 + o(1))\\
        &\geq 2k'_\sigma\log\frac{dn}{2k_\sigma}.
    \end{align*}
    Thus, (\ref{eq: Janson 1}) gives the desired result, which completes the proof.
\qed

\medskip 
\noindent\textit{Proof of \cref{clm: janson 2 step 2}.}
We again consider choosing a random graph $H$ and bounding the probability that that $H$ avoids triangles in $T_A[S_A]$. Let $H$ be a random bipartite graph with bipartition $(S_A, B)$ where $m'_A$ edges chosen uniformly. Let $X$ be the number of triangles in $H \cup T_A$. It immediately follows that the number of choices in step 1 is at most 
\[
\binom{s_Ab}{m'_A}\P[X = 0].
\]
It then suffices to show that 
\begin{equation}\label{eq: desired prob 2}
    \P[X = 0] \leq \exp\left[-\frac{1}{2\eps^2}k'_A\log\frac{dn}{2k_A}\right].
\end{equation}
We again show that this holds for a random graph with edges chosen independently. Let $p = \frac{m'_A}{s_Ab}$ and $H'$ be a random graph obtained by including each edge in $S_A \times B$ independently with probability $p$. Notice that the proof of \cref{clm: m to p} when $\sigma = A$ is also valid under our current conditions, so we have
    \[
    \P_{H'}[X = 0] = \Omega(\P_{H}[X = 0]).
    \]
Thus, it only remains to show that 
\[
\mathbb P_{H'}[X = 0] \leq \exp\left[ -\eps^{-2}k'_A\log\frac{dn}{2k_A} \right].
\]
 Here, $\mathcal J$ is $S_A \times B$ and $\mathcal S$ is the set of pairs $(xz,yz)$ of edges in $\mathcal J$ where $x$ and $y$ are adjacent in $T_A$. Let $\mu = \mathbb E[X]$ and $\Delta$ be as described in \cref{thm: Janson}. Observe that 
    \begin{equation}\label{eq: Janson 2 mu}
        \mu = \mathbb E[X] \geq (1 - \eps)\eps k'_Abp^2 = \frac{(1 - \eps)\eps m_A'^2k'_A}{bs_A^2} = \frac{2\eps(1 + O(\eps))d^2k'_A}{n}
    \end{equation}
    where the final equality holds by observing that $b = n/2(1 + o(1))$ and $m_A = ds_A(1 + O(k_A/ds_A)) = ds_A(1 + O(\log n/(d\log\log n))) = ds_A(1 + o(1))$.
    Also, 
    \begin{equation}\label{eq: Janson 2 Del}
        \Delta = \sum_{xy, yz \in E(T_A[S_A])}\sum_{w \in B} p^3 \leq 2\eps k'_AD_{A}bp^3 = \frac{2\eps D_{A}m_A'^3k'_A}{b^2s_A^3} = \frac{8\eps (1 + O(\eps))D_Ad^3k'_A}{n^2}.
    \end{equation}

    If $\mu \geq \Delta$, (\ref{eq: Janson 1}) gives 
    \begin{align*}
        \mathbb P[X = 0] &\leq \exp[-\mu/2] = \exp\left[ -\frac{\eps(1 + O(\eps))d^2k'_A}{n} \right] \\
        &\leq \exp\left[ -\eps^{-3}(1 + O(\eps))k'_A\log n \right] \\
        &\leq  \exp\left[ -\eps^{-2}k'_A\log\frac{dn}{2k_A} \right],
    \end{align*}
as desired.
    If $\mu < \Delta$, \cref{eq: Janson 2} gives
    \begin{align*}
        \mathbb P[X = 0] &\leq \exp\left[ -\frac{\mu^2}{2\Delta} \right] \leq \exp\left[ -\frac{(1 + O(\eps))}{8\eps^3}k'_A\log n \right]\\ 
        &\leq \exp\left[ -\eps^{-2}k'_A\log\frac{dn}{2k_A} \right], 
    \end{align*}
    which completes the proof.
    \qed

\subsection{Proofs of \cref{clm: s_B = 0} and \cref{clm: s_B > 0}}\label{sec: Janson simplifications}

\noindent\textit{Proof of \cref{clm: s_B = 0}.}
We first expand each binomial coefficient in the ratio using \cref{cor: choose}. When $\lambda > 0$, Observe that
 \begin{align}
     &\binom{s_Ab}{(1 - \lambda )m'_A} = \frac{s_A^{(1 - \lambda )m'_A}b^{(1 - \lambda )m'_A}\exp[O(k_A + k_B)]}{(1 - \lambda )^{(1 - \lambda )m'_A}m'^{(1 - \lambda )m'_A + 1/2}_A(1 - (1 - \lambda )m'_A/s_Ab)^{s_Ab - (1 - \lambda )m'_A}};\notag\\
    &\binom{b}{\lambda m'_A/s_A}^{s_A} = \frac{b^{\lambda m'_A}}{\lambda ^{\lambda m'_A + s_A/2}(m'_A/s_A)^{\lambda m'_A + s_A/2}(1 - \lambda m'_A/bs_A)^{bs_A - \lambda m'_A}}\exp[O(k_A + k_B)];\label{eq: lambda cases 1}\\
     &\binom{b}{m'/|A'|}^{|A'|} = \frac{b^{m'}(1 - m'/|A'|b)^{-|A'|b + m' - |A'|/2}\exp[O(k_A + k_B)]}{(2\pi)^{|A'|/2} (m'/|A'|)^{m' + |A'|/2}\big(1 + |A'|/(12m') + O(|A'|^2/m'^2)\big)^{|A'|}};\notag\\
     &\binom{|A'|b}{m'} = \frac{|A'|^{m'}b^{m'}}{m'^{m' + 1/2}(1 - m'/|A'|b)^{|A'|b - m'}}\exp[O(k_A + k_B)];\notag\\
     &\binom{m'_A}{(1 - \lambda )m'_A} = \frac{1}{(1 - \lambda )^{(1 - \lambda )m'_A}m'^{1/2}_A\lambda ^{\lambda m'_A}}\exp[O(k_A + k_B)];\label{eq: lambda cases 2}\\
     &\binom{n^2/4}{dn/2}   = \frac{(n/2)^{dn}}{(dn/2)^{dn/2 + 1/2}(1 - 2d/n)^{n^2/4 - dn/2}}\exp[O(1)]\notag\\
     &\binom{n/2}{d}^n = \frac{(n/2)^{dn}}{(2\pi)^{n/2}d^{dn + n/2}(1 + 1/12d + O(1/d^2))^{n}(1 - 2d/n)^{n^2/2 - dn + n/2}}\exp[O(1)].\notag
 \end{align}
 When $\lambda = 0$, we use the same expansions apart from \cref{eq: lambda cases 1} and \cref{eq: lambda cases 2}, which are both $1$.
  Notice that 
 \begin{align}
    (1 + |A'|/(12m') + O(|A'|/m'^2))^{|A'|} &=\left(1 + \frac{1}{12d}\left(1 +O\left(\frac{s_A + s_B}{n}\right)\right)\right)^{|A'|}\notag\\
    &= (1 + 1/12d)^{n/2}\exp[O(k_A + k_B)],\label{eq: stirling error fine}
\end{align}
and likewise when replacing $|A'|$ by $b$. We now combine each expanded binomial coefficient and write the result in terms of $n/2$, $d$, and $s_A$. That is, for any $m'_A$ that satisfies~\eqref{eq: relaxed condition},
\begin{align}
    &\frac{\binom{s_Ab}{(1 - \lambda )m'_A}\binom{b}{\lambda m'_A/s_A}^{s_A}\binom{b}{m'/|A'|}^{|A'|}\binom{|A'|}{m'/b}^{b}\binom{n^2/4}{dn/2}}                {\binom{|A'||B'|}{m'}\binom{n/2}{d}^{n}\binom{m_A'}{(1 - \lambda )m_A'}} = 
    (d/n)^{k_A + k_B} \cdot \Lambda \cdot \zeta \cdot \exp\left[O(k_A + k_B)\right],\label{eq: simplified equation}
\end{align}
where 
\begin{align*}
   \Lambda = \begin{cases}
        d^{s_A/2} &\text{if}\quad \lambda = 0\\
        \lambda^{-s_A/2}&\text{otherwise}
    \end{cases}
\end{align*}
and 
\begin{align}
    \zeta =& (1 - (1 - \lambda )m'_A/s_Ab)^{s_Ab - (1 - \lambda )m'_A}(1 - \lambda m'_A/s_Ab)^{s_Ab-\lambda m'_A}\notag\\
    &\times(1 - m'/|A'|b)^{-|A'|b + m' - |A'|/2 - b/2}(1 - 2d/n)^{n^2/4 - dn/2 + n/2}.\label{eq: janson 0 bad zeta}
\end{align}
If $\lambda >0$, recall that $\lambda  = \Omega(1/\log n)$, and so we have 
\begin{equation}\label{eq: bound lambda term}
    \lambda ^{-s_A/2} = \exp[O(k'_A\log\log n)] = \exp\left[o\left(k'_A\log\frac{dn}{2k_A}\right)\right].
\end{equation}
If $\lambda = 0$, we have 
\begin{equation}\label{eq: Gamma when gamma 0}
    \Lambda \leq \exp\left[k'_A\log d\right] \leq \exp\left[3k'_A\log\frac{dn}{2k_A}\right].
\end{equation}
It remains to approximate $\zeta$. By \cref{prop: expand log} and expanding $m'_A$, $m'$, and $|A'|$, we have
\begin{align}
    \log \zeta =& \sum_{i = 2}^w  \left[\frac{d^i}{i(i-1)(n/2)^{i-2}}\left(1 - (1 - 2s_A/n) -((1 - \lambda )^i + \lambda ^i) (2s_A/n)^i\right)\right] \notag\\
    &+\sum_{i = 1}^w\left[\frac{d^i}{i(n/2)^{i-1}}\left(-1 + (1 - 2s_A/n)^i\right)\right]+ O(k_A + k_B) = O\left(\frac{\lambda d^2s_A}{n}\right).\notag
\end{align}
Thus, since $\lambda  = O(\eps^6/\log n)$ and $k_A \leq n\log d$, 
\begin{equation}\label{eq: final zeta}
    \zeta = \exp\left[O\left(\eps^6k'_A\log n\right)\right] = \exp\left[O\left(\eps^5k'_A\log\frac{dn}{2k_A}\right)\right].
\end{equation}

Substituting~\eqref{eq: bound lambda term} and~\eqref{eq: final zeta} into~\eqref{eq: simplified equation} then gives
\begin{align*}
\max&\left\{\frac{\binom{s_Ab}{(1 - \lambda )m'_A}\binom{b}{\lambda m'_A/s_A}^{s_A}\binom{b}{m'/|A'|}^{|A'|}\binom{|A'|}{m'/b}^{b}\binom{n^2/4}{dn/2}}                {\binom{|A'||B'|}{m'}\binom{n/2}{d}^{n}\binom{m_A'}{(1 - \lambda )m_A'}}\right\}\\& \leq \Lambda \cdot\exp\left[-(k_A + k_B)\log\frac{n}{d}  + O\left(\eps^5k'_A\log\frac{dn}{2k_A}\right)\right].
\end{align*}
The claim then follows quickly by substituting~\eqref{eq: bound lambda term} (in the setting of \cref{lem: Janson Counting 1}) or~\eqref{eq: Gamma when gamma 0} (in the setting of \cref{lem: Janson Counting 2}) for $\Lambda$ in the above upper bound.
\qed

\medskip

\noindent\textit{Proof of \cref{clm: s_B > 0}.}
We prove this by first expanding each binomial coefficient and then maximizing with respect to $m_1'$. As in the proof of \cref{clm: s_B = 0}, it will suffice to consider $m'_A = ds_A - m_1' + O(k_A)$ and $m'_B = ds_B - m_1' + O(k_B)$, and so we do not need to maximize with respect to $m'_A$ or $m'_B$. Applying \cref{cor: choose} to each binomial coefficient gives the following.
 {\allowdisplaybreaks\begin{align}
      &\binom{s_As_B}{m_1'} = \frac{(s_As_B)^{m_1'}}{m_1'^{m_1'}(1-m_1'/s_As_B)^{s_As_B- m_1'}}\exp[O(k_A + k_B)];\notag\\
     &\binom{s_A|B'|}{(1 - \lambda )m'_A} = \frac{s_A^{(1 - \lambda )m'_A}|B'|^{(1 - \lambda )m'_A}\exp[O(k_A + k_B)]}{(1 - \lambda )^{(1 - \lambda )m'_A}m'^{(1 - \lambda )m'_A + 1/2}_A(1 - (1 - \lambda )m'_A/s_A|B'|)^{s_A|B'| - (1 - \lambda )m'_A}};\label{eq: janson 1 choose 2}\\
    &\binom{|B'|}{\lambda m'_A/s_A}^{s_A} = \frac{|B'|^{\lambda m'_A}}{\lambda ^{\lambda m'_A + s_A/2}(m'_A/s_A)^{\lambda m'_A + s_A/2}(1 - \lambda m'_A/|B'|s_A)^{|B'|s_A - \lambda m'_A}}\exp[O(k_A + k_B)];\label{eq: janson 1 choose 3}\\
     &\binom{|B'|}{m'/|A'|}^{|A'|} = \frac{|B'|^{m'}(1 - m'/|A'||B'|)^{-|A'||B'| + m' - |A'|/2}\exp[O(k_A + k_B)]}{(2\pi)^{|A'|/2} (m'/|A'|)^{m' + |A'|/2}\big(1 + |A'|/(12m') + O(|A'|^2/m'^2)\big)^{|A'|}};\notag\\
     &\binom{|A'||B'|}{m'} = \frac{|A'|^{m'}|B'|^{m'}}{m'^{m' + 1/2}(1 - m'/|A'||B'|)^{|A'||B'| - m'}}\exp[O(k_A + k_B)];\notag\\
     &\binom{m'_A}{(1 - \lambda )m'_A} = \frac{1}{(1 - \lambda )^{(1 - \lambda )m'_A}m'^{1/2}_A\lambda ^{\lambda m'_A}}\exp[O(k_A + k_B)];\label{eq: janson 1 choose 6}\\
     &\binom{n^2/4}{dn/2}   = \frac{(n/2)^{dn}}{(dn/2)^{dn/2 + 1/2}(1 - 2d/n)^{n^2/4 - dn/2}}\exp[O(1)]\notag\\
     &\binom{n/2}{d}^n = \frac{(n/2)^{dn}}{(2\pi)^{n/2}d^{dn + n/2}(1 + 1/12d + O(1/d^2))^{n}(1 - 2d/n)^{n^2/2 - dn + n/2}}\exp[O(1)].\notag
 \end{align}}

Note that bounds for $\binom{s_B|A'|}{(1 - \lambda )m'_A}$, $\binom{|A'|}{\lambda m'_B/s_B}^{s_B}$, and $\binom{m'_B}{(1 - \lambda )m'_B}$ can be obtained by swapping $A$ and $B$ in~\eqref{eq: janson 1 choose 2},~\eqref{eq: janson 1 choose 3}, and~\eqref{eq: janson 1 choose 6} respectively. Also, note that~\eqref{eq: stirling error fine} holds in this setting as well for either $A'$ or $B'$.
 Next, we combine the expansions of the binomial coefficients and write everything in terms of $n/2$, $d$, $s_A$, $s_B$, and $m_1'$ (this is possible because $m_A' = ds_A - m_1' + O(k_A)$ and $m_B' = ds_B - m_1' + O(k_B)$). For any $(m_1', m_A', m_B')$ that satisfy~\eqref{eq:relaxedconditions}, we have
 \begin{align}\label{eq: simple main terms}
     &\frac{\binom{s_As_B}{m'_1}\binom{s_A|B'|}{(1 - \lambda )m'_A}\binom{s_B|A'|}{(1 - \lambda )m'_B}\binom{|B'|}{\lambda m'_A/s_A}^{s_A}\binom{|A'|}{\lambda m'_B/s_B}^{s_B}\binom{|B'|}{m'/|A'|}^{|A'|}\binom{|A'|}{m'/|B'|}^{|B'|}\binom{n^2/4}{dn/2}}                {\binom{|A'||B'|}{m'}\binom{n/2}{d}^{n}\binom{m_A'}{(1 - \lambda )m_A'}\binom{m'_B}{(1 - \lambda )m'_B}} \notag \\
     & = \left(d/n\right)^{k_A + k_B}\lambda ^{-s_A/2 - s_B/2}\cdot\xi\cdot\zeta\cdot\exp[O(k_A + k_B)], 
 \end{align}
 where 
 \begin{align}
     \xi = \left(\frac{2ds_As_B}{nm'_1}\right)^{m'_1}\left(1 - \frac{2s_A}{n}\right)^{dn/2 - ds_A}\left(1 - \frac{2s_B}{n}\right)^{dn/2 -ds_B}\left(1 - \frac{m_1'}{ds_B}\right)^{-ds_B + m_1'}\notag\\\times \left(1 - \frac{m_1'}{ds_A} \right)^{-ds_A + m_1'}\left(1 - \frac{2s_A}{n} - \frac{2s_B}{n} + \frac{2m_1'}{dn}\right)^{-dn/2 + ds_A + ds_B -m_1'}\label{eq: Janson 1 xi}
 \end{align}
 and 
 \begin{align}
     \zeta &= \left(1 - \frac{m_1'}{s_As_B}\right)^{-s_As_B + m_1'}
      \left(1 - \frac{(1 - \lambda )m'_A}{s_A|B'|}\right)^{-s_A|B'| + (1 - \lambda )m'_A}\notag \\ &\times\left(1 - \frac{(1 - \lambda )m'_B}{s_B|A'|}\right)^{-s_B|A'| + (1 - \lambda )m'_B} \left(1 - \frac{\lambda  m'_A}{s_A|B'|}\right)^{-s_A|B'| + \lambda m'_A}
      \left(1 - \frac{\lambda m'_B}{s_B|A'|}\right)^{-s_B|A'| + \lambda m'_B}\notag \\ &\times\left(1 - \frac{m'}{|A'||B'|}\right)^{-|A'||B'| + m' - |A'|/2 - |B'|/2}\left(1 - \frac{2d}{n}\right)^{n^2/4 - dn/2 + n/2}.\label{eq: Janson 1 zeta}
 \end{align}
 Similarly to~\eqref{eq: bound lambda term}, observe that 
 $$\lambda ^{-s_A/2 - s_B/2} = \exp[O((k'_A + k'_B)\log\log n)]= \exp\left[o\left(k'_A\log \frac{dn}{2k_A} + k'_B\log\frac{dn}{2k_B}\right)\right].$$
 That is, the right hand side of~\eqref{eq: simple main terms} can be simplified to 
 \begin{equation}\label{eq: simplified main term 2}
     (d/n)^{k_A + k_B}\cdot \xi \cdot \zeta \cdot \exp\left[o\left(k'_A\log\frac{dn}{2k_A} + k'_B\log\frac{dn}{2k_B}\right)\right].
 \end{equation}
 We next show that~\eqref{eq: Janson 1 xi} is maximized when $m_1' = \frac{2ds_As_B}{n}$.
 To do so, we evaluate the first and second derivative of $\log \xi$ with respect to $m_1'$. Observe that 
\begin{align}
    \frac{d\log \xi}{dm_1'} &=\log\frac{2ds_As_B}{nm_1'} + \log\left(1 - \frac{m_1'}{ds_B}\right)  + \log\left(1 - \frac{m_1'}{ds_A}\right)  - \log\left(1 - \frac{2s_A}{n} - \frac{2s_B}{n}  +\frac{2m_1'}{dn}\right).\label{eq: xi der}
\end{align}
When $m_1' = \frac{2ds_As_B}{n}$, the right hand side of~\eqref{eq: xi der} is 
\[
    \log\frac{(1 - 2s_A/n)(1 - 2s_B/n)}{(1 - 2s_A/n - 2s_B/n + 4s_As_B/n^2)} = 0.
\]
Moreover, notice that the derivative of~\eqref{eq: xi der} with respect to $m_1'$ is negative, so $m_1' = \frac{2ds_As_B}{n}$ maximizes $\log \xi$, and hence also maximizes $\xi$. By substituting this into~\eqref{eq: Janson 1 xi}, we have $\xi \leq 1$.
Thus,~\eqref{eq: simplified main term 2} can again be simplified to
\begin{equation}
    (d/n)^{k_A + k_B}\cdot\zeta\cdot\exp\left[o\left(k'_A\log\frac{dn}{2k_A} + k'_B\log\frac{dn}{2k_B}\right)\right].
\end{equation}

Next, we estimate $\log\zeta$ in a similar way. Observe that the first derivative of $\log \zeta$ with respect to $m_1'$ is
\begin{align}
    & \log\left(1 - \frac{m_1'}{s_As_B}\right) -(1 - \lambda ) \log\left(1 - \frac{(1 - \lambda )m'_A}{s_A|B'|}\right) - (1 - \lambda )\log\left(1 - \frac{(1 - \lambda )m'_B}{s_B|A'|}\right) \notag\\
    & -\lambda \log\left(1 - \frac{\lambda m'_A}{s_A|B'|}\right)  - \lambda \log\left(1 - \frac{\lambda m'_B}{s_B|A'|}\right)  + \log\left(1 - \frac{m'}{|A'||B'|}\right) \notag\\
    &+ \frac{1}{2|B'|(1 - m'/|A'||B'|)} + \frac{1}{2|A'|(1 - m'/|A'||B'|)}.\label{eq: zeta der}
\end{align}

Notice that the derivative of~\eqref{eq: zeta der} is negative, so it suffices to find where~\eqref{eq: zeta der} is zero.
If $m_1'= \frac{2ds_As_B}{n}$, the right hand side of~\eqref{eq: zeta der} is 
\begin{align*}
    \log(1 - 2d/n) - (1 - \lambda )\log(1 - (1-\lambda )2d/n) - (1 - \lambda )\log(1 - (1 - \lambda )2d/n )\\
    -\lambda \log(1 - \lambda 2d/n) - \lambda \log(1 - \lambda 2d/n) + \log\left(1 - 2d/n\right) + O(1/n)
    \leq 0,
\end{align*}
since $\frac{(1 - 2d/n)^2}{(1 - (1 - \lambda )2d/n)^{2 - 2\lambda }(1 - \lambda 2d/n)^{2\lambda }} \leq \exp[-8\lambda d/n + O(\lambda d^2/n^2)]$. In fact, since each of the first six terms in~\eqref{eq: zeta der} increases as $m_1'$ decreases and the first term increases by $16\lambda d/n + O(\lambda d^2/n^2)$ when $m_1'$ is decreased to $\frac{2ds_As_B}{n}(1 - 16\lambda )$, it follows that~\eqref{eq: zeta der} is positive when $m'_1 = \frac{2ds_As_B}{n}(1 - 16\lambda )$. Thus, $\zeta$ is maximized when $m_1' = \frac{2ds_As_B}{n}(1 - \beta)$ for some $0 < \beta \leq 16\lambda $. 
It remains to estimate $\log\zeta$ when $m_1' = \frac{2ds_As_B}{n}(1 - \beta)$. To do so, we use \cref{prop: expand log} to expand the each term in~\eqref{eq: Janson 1 zeta}. The second sum below is the expansion of the $|A'|/2 - |B'|/2$ and $n/2$ exponents of the final two terms in~\eqref{eq: Janson 1 zeta}. This gives that $\log \zeta$ is equivalent to 
\begin{align}
    &\sum_{i = 2}^w\left[\frac{d^i}{i(i-1)(n/2)^{i-2}}\left(1 - (1 - \beta)^i\frac{4s_As_B}{n^2} - \frac{(1 - 2s_A/n - 2s_B / n + 4s_As_B(1 - \beta)/n^2)^i}{(1 - 2s_A/n - 2s_B/n + 4s_As_B/n^2)^{i-1}}  \right. \right.\notag\\&\left.\left. -(\lambda ^i + (1 - \lambda )^i)\frac{2s_A(1-2(1 - \beta)s_B/n)^i}{n(1 - 2s_B/n)^{i-1}} - (\lambda ^i + (1 - \lambda )^i)\frac{2s_B(1-2(1 - \beta)s_A/n)^i}{n(1 - 2s_A/n)^{i-1}}\right)\right] \notag\\
    &+ \sum_{i = 1}^{w}\left[\frac{d^i}{(n/2)^{i-1}i}\left(-1 + \frac{1 - 2s_A/n - 2s_B/n + 4(1 - \beta)s_As_B/n^2}{1 - 2s_B/n}\right.\right.\notag\\&+ \left.\left.\frac{1 - 2s_A/n - 2s_B/n + 4(1 - \beta)s_As_B/n^2}{1 - 2s_A/n}\right)\right] + O(k_A + k_B)\notag\\
    &= O\left(\frac{\lambda d^2(s_A + s_B)}{n}\right).\notag
\end{align}
Recall that $\lambda  = O(\eps^6 / \log n)$ and $k_A + k_B \leq n\log d$, so \[
\zeta = \exp\left[O\left(\eps^6(k'_A + k'_B)\log n\right)\right] = \exp\left[O\left(\eps^5k'_A\log\frac{dn}{2k_A} + \eps^5k'_B\log\frac{dn}{2k_B}\right)\right].
\]
Thus, again, we can simplify~\eqref{eq: simplified main term 2} to 
\[
(d/n)^{k_A + k_B}\exp\left[O\left(\eps^3k'_A\log\frac{dn}{2k_A} + \eps^3k'_B\log\frac{dn}{2k_B}\right)\right],
\]
which completes the proof.
\qed

\section{Small torsos: proofs of \cref{lem: Small torso counting small k} and \cref{lem: Small Torso Counting}}\label{sec: Small Torso}
\subsection{Proof of \cref{lem: Small torso counting small k}}\label{sec: small torso simple}
As before, let $a = |A|$, $b = |B|$, $k'_A = k_{A, r_A - 1}$, and $k'_B = k_{B, r_B - 1}$ and assume $(A,B)$ satisfies \eqref{eq:AB}. We bound $|\mathcal T_{k_A, r_A, k_B, r_B}|$ in a similar manner to the methods used in \cref{sec: Janson}. That is, we first fix a $(k_A, r_A, A, D_A)$-graph $G_A$ and a $(k_B, r_B, B, D_B)$-graph $G_B$ with maximum degree at most $d/2$ as defect graphs and then upper bound the number of ways to complete these graphs to be $d$-regular without creating any triangles by adding edges between $A$ and $B$. To bound the number of choices for cross-edges that do not create triangles, we require that there is a relatively small set $S$ of vertices in $B$ that have large defect degree, and in fact $S$ is incident to almost all defect edges in $B$. Thus, we consider the notion of spines that was introduced in \cite{osthus2003densities}.

\begin{definition}\label{def: spine}
    Given a graph $G$ let $T$ be a $D$-torso of $G$. A set $S \subseteq V(G)$ is a spine of $G$ if every edge in $E(G) \setminus E(T)$ is incident to a vertex in $S$. 
\end{definition}
As was shown in~Proposition 11 of \cite{osthus2003densities}, every $(k_B, \bar r_B, B, D_B)$-graph has a spine with size at most $s := 2k'_B/D_B$. 

Recall that by \cref{cor: counting defects}, there are at most 
\begin{equation}\label{eq: GAGB}
    \exp\left[(k_A + k_B)\log\frac{n}{d} + (1 + \eps^2)\left(k'_A\log\frac{dn}{2k_A} +k'_B\log\frac{dn}{2k_B}\right)\right]
\end{equation}
choices for $G_A$ and $G_B$. Also, there are at most 
\begin{equation}\label{eq: bounding spines 1}
    sb^{s} = \exp\left[O(k'_Bd^{-1}\log n + \log(k'_B/d))\right] = \exp\left[o(k_B)\right]
\end{equation}
spines of $G_B$ with size at most $s$.

Fix a triangle-free $(k_A, r_A, A, D_A)$-graph $G_A$ and a triangle-free $(k_B, r_B, B, D_B)$-graph $G_B$ and let $S$ be a spine of $G_B$ with size at most $s$. Let $\mathcal H_{G_A, G_B, S}$ be the set of $d$-regular graphs $H$ where $H[A] = G_A$, $H[B] = G_B$, and no edge in $G_B$ is contained in a triangle in $H$. Similarly to $S_A$ and $S_B$ in the main proofs in \cref{sec: Janson}, $\mathcal H_{G_A, G_B, S}$ does not depend on $S$, but it is useful to fix $S$ when counting the number of ways to add cross-edges without creating triangles. However, here $S$ is much smaller, and so summing over all choices of $S$ is negligible. Thus,
\begin{equation}\label{eq: small torso 1 sum}
    |\mathcal T_{k_A, r_A, k_B, r_B}| \leq \sum_{G_A, G_B, S}|\mathcal H_{G_A, G_B, S}|.
\end{equation}
Let $\mathcal H = \mathcal H_{G_A, G_B, S}$. 
Similarly to the main arguments in \cref{sec: Janson}, we bound $|\mathcal H |$ by counting the number of ways to construct a member $H\in\mathcal{H}$ by adding edges into $H$ in multiple steps that we describe below.

\begin{itemize}
    
    \item[1.](1st stage of choosing cross-edges for $S'$) Let 
    \[
    D' = \frac{\varepsilon^7d}{16\log\log\left(\frac{dn}{2k_B}\right)}.
    \]
    Let $S' = \{u \in S \mathrel : d_{G_B}(u) \geq D'\}$. We will have $|S'| \geq 1$, which is justified below in \cref{clm: small torso subset}. For each $u \in S'$, choose a set $A_u$ of $d_{G_B}(u)$ vertices in $A$ and add edges $\{ux \mathrel : x \in  A_u\}$ to $H$. Moreover, ensure that no vertex in $H$ has degree more than $d$.
    \item[2.](2nd stage of choosing cross-edges for $S'$) For each $u \in S'$, choose a set $\mathcal U$ of $d - 2d_{G_B}(u)$ vertices in $A$ and add edges $\{ux \mathrel : x \in \mathcal U\}$ to $H$. Moreover, ensure that $H$ is simple (i.e. no edge that was already included in $H$ is chosen to be added) and no vertex in $H$ has degree more than $d$. Let $m$ denote the number of edges that were added in steps 1 and 2. Note that $m$ is determined by $G_A,G_B$ and $S$.
    \item[3.](degree completion) Let $B' = B \setminus S'$. Complete $H$ to be a (simple) $d$-regular graph by including $dn/2 - k_A - k_B - m$ edges in $A \times B'$ such that no edge in $G_B$ is contained in a triangle in $H$.
\end{itemize}

A visual representation of this procedure is given in Figure~\ref{fig: small torso}.

    \begin{figure}[h]
        \centering{
        \resizebox{0.7\textwidth}{!}{\fontsize{30pt}{30pt}\selectfont
\begingroup%
  \makeatletter%
  \providecommand\color[2][]{%
    \errmessage{(Inkscape) Color is used for the text in Inkscape, but the package 'color.sty' is not loaded}%
    \renewcommand\color[2][]{}%
  }%
  \providecommand\transparent[1]{%
    \errmessage{(Inkscape) Transparency is used (non-zero) for the text in Inkscape, but the package 'transparent.sty' is not loaded}%
    \renewcommand\transparent[1]{}%
  }%
  \providecommand\rotatebox[2]{#2}%
  \newcommand*\fsize{\dimexpr\f@size pt\relax}%
  \newcommand*\lineheight[1]{\fontsize{\fsize}{#1\fsize}\selectfont}%
  \ifx\svgwidth\undefined%
    \setlength{\unitlength}{522.43325061bp}%
    \ifx\svgscale\undefined%
      \relax%
    \else%
      \setlength{\unitlength}{\unitlength * \real{\svgscale}}%
    \fi%
  \else%
    \setlength{\unitlength}{\svgwidth}%
  \fi%
  \global\let\svgwidth\undefined%
  \global\let\svgscale\undefined%
  \makeatother%
  \begin{picture}(1,1.06662731)%
    \lineheight{1}%
    \setlength\tabcolsep{0pt}%
    \put(0,0){\includegraphics[width=\unitlength,page=1]{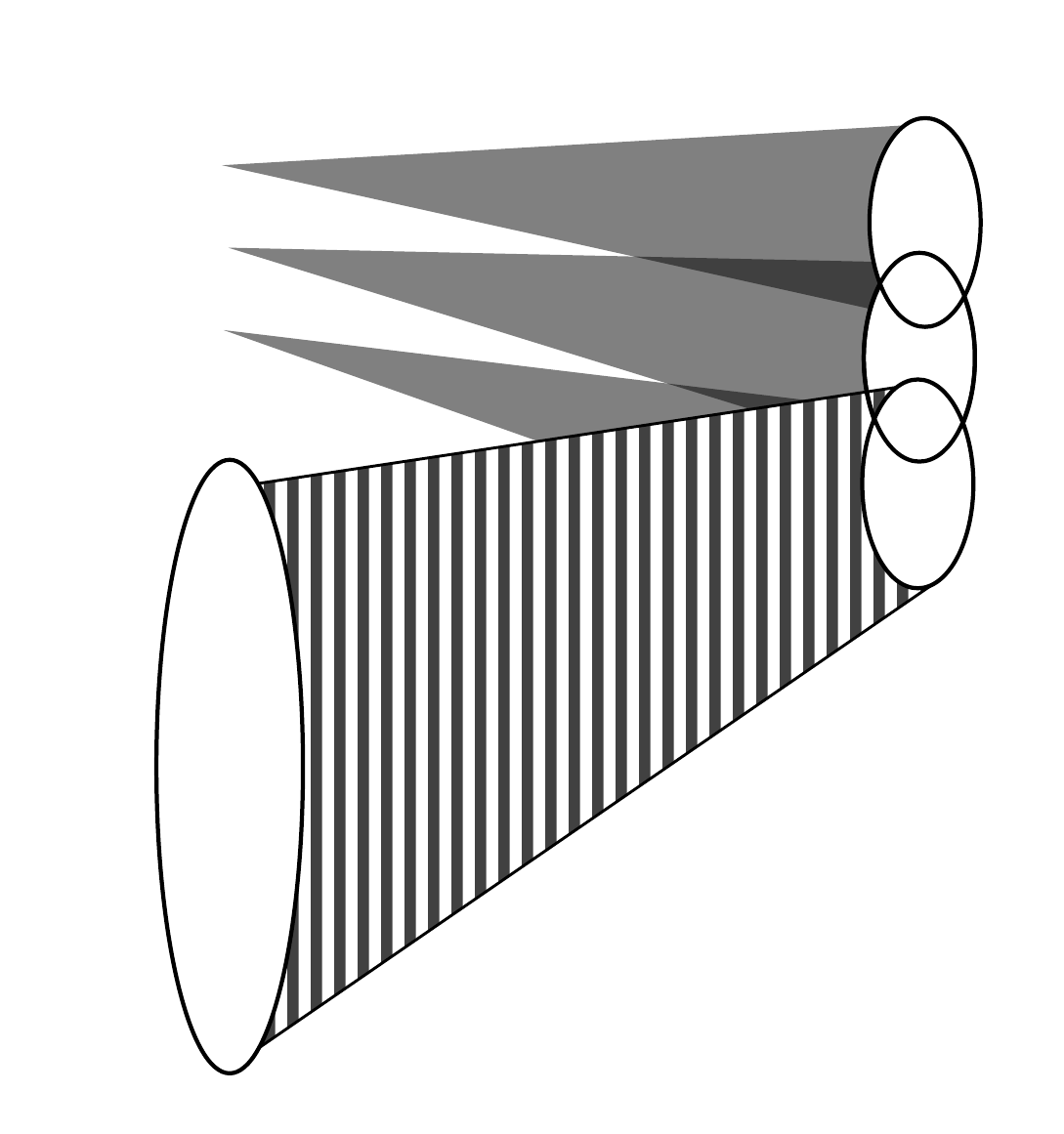}}%
    \put(0.0339417,0.8110937){\color[rgb]{0,0,0}\makebox(0,0)[lt]{\lineheight{1.25}\smash{\begin{tabular}[t]{l}$S'$\end{tabular}}}}%
    \put(0,0){\includegraphics[width=\unitlength,page=2]{expansion_thesis.pdf}}%
    \put(0.19194133,1.00776983){\color[rgb]{0,0,0}\makebox(0,0)[lt]{\lineheight{1.25}\smash{\begin{tabular}[t]{l}$A$\end{tabular}}}}%
    \put(0,0){\includegraphics[width=\unitlength,page=3]{expansion_thesis.pdf}}%
    \put(0.84346956,1.00593417){\color[rgb]{0,0,0}\makebox(0,0)[lt]{\lineheight{1.25}\smash{\begin{tabular}[t]{l}$B$\end{tabular}}}}%
    \put(0,0){\includegraphics[width=\unitlength,page=4]{expansion_thesis.pdf}}%
  \end{picture}%
\endgroup%
}}
        
        \caption[Procedure for constructing graphs in \cref{lem: Small torso counting small k}]{Representation of above procedure. The shaded cross-edges are added in steps 1 and 2 of the procedure, where some special subset is chosen in step 1. All other cross-edges, including the dashed cross-edges, are added in step 3 of the procedure. The key fact here is that many of the special neighbourhoods chosen in step 1 will have large cross-neighbourhoods chosen in step 3, indicated by the striped cross-edges.}
        \label{fig: small torso}
\end{figure}

\begin{claim}\label{clm: small torso subset}
    For any given $(k_B, \bar r_B, B, D_B)$-graph $H_B$, there exists a set $S' \subseteq B$ where every vertex in $S'$ is incident to at least $D'$ edges in $H_B$, and the total number of edges in $H_B$ that are incident to at least one vertex in $S'$ is at least $(1 - \varepsilon)k_B$.
\end{claim}
\begin{proofclaim}
    Fix a $D_B$-torso $T_B$ of $H_B$. As mentioned before, $H_B$ has a spine $S_B$ with $|S_B| \leq s$. Note that $k_B \geq |D_B|$, as otherwise $k_B = k'_B$, contradicting our assumption that $r_B = \overbar r_B$. Let $D''$ be the average degree of vertices in $S_B$ in $H_B$. Then, 
    \begin{align*}
        D''|S_B| \geq k_B - k'_B
    \end{align*}
    implying that 
    \begin{align}
        D'' &\geq D_{B} \left(\frac{k_B}{2k'_B} - \frac{1}{2}\right) \notag \\
        &= \frac{\varepsilon^4d}{2\log\left(\frac{dn}{2k_B}\right)}\left(\frac{(1 - \varepsilon)\varepsilon^2\log\left(\frac{dn}{2k_B}\right)}{8\log\log\left(\frac{dn}{2k_B}\right)} - \frac{1}{2}\right) \notag \\
        &\geq \frac{\varepsilon^7d}{16\log\log\left(\frac{dn}{2k_B}\right)}.\notag
    \end{align}
    Thus, since the average degree of vertices in $S_B$ is at least $D'$, some vertex has degree at least $D'$ in every $(k_B, r_B, B, D_B)$-graph.

    Let $S'_B \subseteq S_B$ denote the set of vertices in $S_B$ that have degree at least $D'$. Vertices in $S_B\setminus S'_B$ are incident to at most $|S_B|D' \leq 2k'_BD'/D_B \leq \frac{\varepsilon}{2}k_B $ edges in $G_B$. Also, since $r_B = \overbar r_{B}$, $k'_B=o(k_B)$. Thus, vertices in $S$ are incident to at least $(1 - o(1))k_B$ edges in $G_B$, so vertices in $S'$ must be incident to at least $(1 - \varepsilon/2 - o(1))k_B > (1 - \varepsilon)k_B$ edges in $G_B$.
    
    \end{proofclaim}

    We now approximate the number of choices in each step. 
    Recall that $G_A,G_B$ and $S$ determine both $S'$ and $m$. Let $s'=|S'|$ which is also determined by $G_A,G_B$ and $S$.
    The
number of choices for  steps 1 and 2 together is obviously bounded by 
\begin{equation}\label{eq: steps 1 and 2 1}
        \binom{as'}{m}.
    \end{equation} 
 Also by the observation that $S'\subseteq S$ where $|S|\le s$, and the construction of the $m$
 edges in the first two steps, we obtain
\begin{equation}\label{eq: mS' relaxed 1}
       1 \leq s' \leq s \quad\text{and} \quad  \max\{0,ds' - 2k_B\} \leq m \leq ds'.
    \end{equation}

    Bounding the number of choices in step 3 is more difficult, and requires us to consider choices in step 1 that satisfy an extra condition (P) defined below, and the choices without condition (P) separately. We show that the number of choices in step 1 where (P) fails is sufficiently small and then consider only choices where (P) holds. To define (P), we first introduce the following parameters.
    \begin{equation}\label{eq: y',y''}
        y' = \frac{a}{\log^{4000/\eps^4}\frac{dn}{2k_B}}, \quad y'' = \min\{k_B/16, y'\}.
    \end{equation}
\begin{definition}\label{def: well-spread}
    A collection of sets $\{A_i \subseteq A\}_{i \in \mathcal I}$ is \textit{well-spread} if for all $\mathcal J \subseteq \mathcal I$ with $\sum_{j \in \mathcal J}|A_j| \geq k_B/8$, there exists $\mathcal J' \subseteq \mathcal J$ and a collection $\{A'_j\}_{j \in J'}$ of disjoint sets where for all $j \in \mathcal J'$, $A'_j \subseteq A_j$, $|A'_j| \geq |A_j|/2$, and $\sum_{j \in \mathcal J'}|A'_j| \geq y''/2$.
\end{definition}
    
    Notice that here, $y'' = k_B/16$ by the assumptions of the lemma. We use the definition of well-spread in the proof of \cref{lem: Small Torso Counting} as well, and so define it more generally here. Now, we define condition (P) as follows.
    \begin{enumerate}
        \item[(P)] $\{A_u\}_{u \in S'}$ is well-spread.
    \end{enumerate}
    Let $m' = dn/2 - k_A - k_B - m$. Together, the following two claims help us to bound $|\mathcal H|$.

    \begin{claim}\label{clm: P fails}
        The number of choices in steps 1-3 such that (P) fails is at most

\[
        \exp\left[-12\eps^{-2}k'_B\log\frac{dn}{2k_B}\right]\max_{s', m}\left\{ \binom{as'}{m}\binom{b - s'}{m'/a}^a\binom{a}{m'/(b - s')}^{b-s'}\binom{a(b - s')}{m'}^{-1} \right\},
        \]
where the maximization is over all $(s',m)$ that satisfies~\eqref{eq: mS' relaxed 1}.
        
    \end{claim}

    \begin{claim}\label{clm: P holds}
        The number of choices in steps 1-3 such that (P) holds is at most 


\[
        \exp\left[-\eps^{-3}k'_B\log\frac{dn}{2k_B}\right]\max_{s', m}\left\{ \binom{as'}{m}\binom{b - s'}{m'/a}^a\binom{a}{m'/(b - s')}^{b-s'}\binom{a(b - s')}{m'}^{-1} \right\}.
        \]
where the maximization is over all $(s',m)$ that satisfies~\eqref{eq: mS' relaxed 1}.

    \end{claim}

    Obviously the bound in~\cref{clm: P holds} is negligible compared to the bound in \cref{clm: P fails} whenever $\eps < 1/12$. Thus, adding these two bounds together yields the following upper bound on $|\mathcal H|$.
    \begin{equation}
        |\mathcal H| \leq \exp\left[-11\eps^{-2}k'_B\log\frac{dn}{2k_B}\right]\max_{s', m}\left\{ \binom{as'}{m}\binom{b - s'}{m'/a}^a\binom{a}{m'/(b - s')}^{b-s'}\binom{a(b - s')}{m'}^{-1} \right\}.
    \end{equation}
    Thus, by~\eqref{eq: small torso 1 sum}, we have 
    \begin{align}
    |\mathcal T_{k_A, r_A, k_B, r_B}| \leq& \sum_{G_A, G_B, S}\max_{s', m}\left\{ \binom{as'}{m}\binom{b - s'}{m'/a}^a\binom{a}{m'/(b - s')}^{b-s'}\binom{a(b - s')}{m'}^{-1} \right\}\notag\\
    &\times \exp\left[-11\eps^{-2}k'_B\log\frac{dn}{2k_B}\right].\label{eq: small torso 1 T}
    \end{align}
    Simplifying this leads to the following claim, which completes the proof.
    \begin{claim}\label{clm: small torso final 1}
        \[
        \frac{|\mathcal T_{k_A, r_A, k_B, r_B}|}{|\mathcal G_{n,d}^\text{bip}|}\leq\frac{1}{\binom{n}{n/2}}\exp\left[ -2\left( k'_A\log\frac{dn}{2k_A} + k'_B\log\frac{dn}{2k_B}\right) \right].
        \]
    \end{claim}
\qed

\cref{clm: P holds} is proved in \cref{sec: expansion}. \cref{clm: P fails} and \cref{clm: small torso final 1} are proved below in this subsection.

\medskip

\noindent\textit{Proof of \cref{clm: P fails}.}
Let $G$ be a random bipartite graph with bipartition $(A,S')$ obtained by choosing sets $A_u \subseteq A$ for each $u \in S'$ with $|A_u| = d_{G_B}(u)$ uniformly and independently at random and including edges $\{ux \mathrel x \in A_u\}$. Let $\mathcal E$ be the event that $\{A_u\}_{u \in S'}$ is not well-spread. Then, for some $(s',m')$ that satisfy~\eqref{eq: mS' relaxed 1},~\eqref{eq: steps 1 and 2 1} gives that there are at most 
        \begin{equation}\label{eq: step 1 and 2 P}
        \binom{as'}{m}\P[\mathcal E]
        \end{equation}
        choices in step 1. Moreover, using \cref{cor: max sequence} to crudely bound the number of choices in step 3 gives that there are at most 
        \begin{equation}\label{eq: step 3 P}
            O\left(\binom{b - s'}{m'/a}^a\binom{a}{m'/(b - s')}^{b-s'}\binom{a(b - s')}{m'}^{-1}\right)
        \end{equation}
        choices in step 3. Combining~\eqref{eq: step 1 and 2 P} and~\eqref{eq: step 3 P} gives that there are at most 
        \[
        \max_{s', m}\left\{ \binom{as'}{m}\binom{b - s'}{m'/a}^a\binom{a}{m'/(b - s')}^{b-s'}\binom{a(b - s')}{m'}^{-1} \right\}O(\P[\mathcal E])
        \]
        total choices. Hence, it suffices to show that 
        \begin{equation}
            \P[\mathcal E] \leq \exp\left[-13\eps^{-2}k'_B\log\frac{dn}{2k_B}\right].\notag
        \end{equation}

    Fix $S^* \subseteq S'$ where $\sum_{u \in S^*}d'_u \geq k_B/8$. Then, Corollary~24 in \cite{osthus2003densities} implies that there does not exist $S^{**} \subseteq S^*$ and sets $A'_u \subseteq A_u$ for all $u \in S^{**}$ where $|A'_u| \geq |A_u|/2$ and $\sum_{u \in S^{**}}|A'_u| \geq k/32$ with probability at most 
    \[
    \exp\left[-(1 +\varepsilon^2)14\varepsilon^{-2}k'_{B}\log\frac{dn}{2k_B}\right].
    \]
    We then use the union bound to bound the probability that any $S^*$ has the above property. That is, the probability that there exists any set $S^*$ such that there do not exist $S^{**} \subseteq S^*$ and sets $A'_u \subseteq u$ for all $u \in S^{**}$ where $|A'_u| \geq |A_u|/2$ and $\sum_{u \in S^{**}}|A'_u| \geq k/32$ is bounded by 
    \[
    2^{s}\exp\left[-(1 +\varepsilon^2)14\varepsilon^{-2}k'_{B}\log\frac{dn}{2k_B}\right] \leq \exp\left[-13\varepsilon^{-2}k'_B\log\frac{dn}{2k_B}\right],
    \]
    as desired, where the inequality holds because $|S| = o(k_B) = o\left(k'_B\log\frac{dn}{2k_B}\right)$.
    \qed

    \medskip

\noindent\textit{Proof of \cref{clm: small torso final 1}.}
By~\eqref{eq: GAGB},~\eqref{eq: bounding spines 1}, and~\eqref{eq: small torso 1 T}, we have
\begin{align}
    |\mathcal T_{k_A, r_A, k_B, r_B}| \leq& \max_{s', m}\left\{ \binom{as'}{m}\binom{b - s'}{m'/a}^a\binom{a}{m'/(b - s')}^{b-s'}\binom{a(b - s')}{m'}^{-1} \right\}\notag\\
    &\times \exp\left[(k_A + k_B)\log\frac{n}{d} -11\eps^{-2}k'_B\log\frac{dn}{2k_B} + O(k_A + k_B)\right].\notag
\end{align}
Then, by~\eqref{eq:reg bip count2}, we have 
\begin{align}
    \frac{|\mathcal T_{k_A, r_A, k_B, r_B}|}{|\mathcal G_{n,d}^\text{bip}|} \leq& \max_{s', m}\left\{\frac{ \binom{as'}{m}\binom{b - s'}{m'/a}^a\binom{a}{m'/(b - s')}^{b-s'}\binom{n^2/4}{dn/2}}{\binom{a(b - s')}{m'}\binom{n/2}{d}^n\binom{n}{n/2} }\right\}\notag\\
    &\times \exp\left[(k_A + k_B)\log\frac{n}{d} -11\eps^{-2}k'_B\log\frac{dn}{2k_B} + O(k_A + k_B)\right].\label{eq: T/G}
\end{align}
The following claim evaluates the maximization in a similar manner to \cref{clm: s_B = 0} and \cref{clm: s_B > 0}.

\begin{claim}\label{clm: small torso max}
Under the conditions of \cref{lem: Small torso counting small k} or \cref{lem: Small Torso Counting},
    \[
    \max_{s', m}\left\{\frac{ \binom{as'}{m}\binom{b - s'}{m'/a}^a\binom{a}{m'/(b - s')}^{b-s'}\binom{n^2/4}{dn/2}}{\binom{a(b - s')}{m'}\binom{n/2}{d}^n }\right\} \leq \exp\left[-(k_A + k_B)\log\frac{n}{d} + O(k_A + k_B)\right].
    \]
\end{claim}
By assumption of the lemma, $2k'_A\log\frac{dn}{2k_A} \leq 4\eps^{-2}k'_B\log\frac{dn}{2k_B}$, and so 
\[
\exp\left[-11\eps^{-2}k'_B\log\frac{dn}{2k_B}\right] \leq \exp\left[-2\left(k'_A\log\frac{dn}{2k_A} + k'_B\log\frac{dn}{2k_B}\right)\right].
\]
Thus, combining~\eqref{eq: T/G} and \cref{clm: small torso max} gives the desired result. \cref{clm: small torso max} is proved in \cref{sec: small torso max}.
\qed

\subsection{Proof of \cref{lem: Small Torso Counting}}\label{sec: small torso proof 1}
Again, let $a = |A|$, $b = |B|$, $k'_A = k_{A, r_A - 1}$, and $k'_B = k_{B, r_B - 1}$ and assume $(A,B)$ satisfies \eqref{eq:AB}. This proof follows a similar structure to the proof of \cref{lem: Small torso counting small k}, but here we consider choosing $G_B$ to avoid triangles rather than choosing the cross-edges to avoid triangles. Since the cross-edges are dependent on the degree sequence of the defect graphs, we have to consider all possible degree sequences for $G_B$. Hence, a natural way to bound $|\mathcal T_{k_A, r_A, k_B, r_B}|$ would be to first fix a $(k_A, r_A, A, D_A)$-graph $G_A$ and a degree sequence $\mathbf d'$ on $B$ (which will be the degree sequence for the defect edges in $B$), then estimate the number of bipartite graphs with bipartition $(A,B)$ on degree sequence $(d-d_{G_A}(1),d - d_{G_A}(2),\ldots, d - d_{G_A}(n)) - \mathbf d'$, and then estimate the number of choices for a $(k_B, r_B, B, D_B)$-graph $G_B$ with degree sequence ${\mathbf d}'$ such that no edge in $G_B$ is contained in a triangle in the resulting graph. However, getting a sufficiently good bound for the number of choices for $G_B$ given ${\mathbf d}'$ is difficult. As a result, we count the choices for $G_B$ based on the spine of $G_B$ (see \cref{def: spine}).

Instead, we bound $|\mathcal T_{k_A, r_A, k_B, r_B}|$ by first fixing a $(k_A, r_A, A, D_A)$-graph $G_A$ with maximum degree at most $d/2$, and then fixing a subset of vertices $S\subseteq B$ as well as a degree sequence $\mathbf d' \leq (d/2, d/2, \ldots, d/2)$ on $B$, then estimating the number of bipartite graphs with bipartition $(A,B)$ on degree sequence $(d-d_{G_A}(1),d - d_{G_A}(2),\ldots, d - d_{G_A}(n)) - \mathbf d'$, and finally estimating the number of choices for a $(k_B, r_B, B, D_B)$-graph $G_B$ with degree sequence ${\mathbf d}'$ such that $S$ is a spine of $G_B$, and no edge in $G_B$ is contained in a triangle in the resulting graph.

Observe that by \cref{cor: counting defects}, there are at most 
\begin{equation}\label{eq: GA}
    \exp\left[k_A\log\frac{n}{d} + (1 + \eps^2)k'_A\log\frac{dn}{2k_A}\right]
\end{equation}
choices for $G_A$. 
Given a set $S \subseteq B$ with $|S|$ small, the claim below bounds the number of degree sequences ${\mathbf d}'$ such that there is a $(k_B, r_B, B, D_B)$-graph $G_B$ with degree sequence $\mathbf{d}'$ and $S$ is a spine of $G_B$. Recall that every $(k_B, r_B, B, D_B)$-graph has a spine with size at most $s := 2k'_B/D_B$.

\begin{claim}\label{clm: num deg sequences}
    For fixed $S \subseteq B$ with $|S| \leq s$, there are at most 
    \[
   \binom{k_B-k'_B + s - 1}{s-1}\binom{k_B + k'_B + b - 1}{b-1} \leq {  \exp\left[(1 + o(1))k'_B\log\frac{dn}{2k_B}\right]}
    \]
    degree sequences on $B$ that admit a $(k_B, r_B, B, D_B)$-graph with spine $S$.
\end{claim}

Next, we fix $(G_A, \mathbf d', S)$ where $S \subseteq B$ and $|S| \leq s:=2k'_B/D_B$. Let $\mathcal H_{G_A, \mathbf d', S}$ be the set of $d$-regular graphs $H$ on vertex set $[n]$ where $H[A] = G_A$, $H[B]$ is a $(k_B, r_B, B, D_B)$-graph with degree sequence $\mathbf d'$ where $S$ is a spine of $H[B]$, and no edge in $H[B]$ is contained in a triangle in $H$.  Then,

\begin{equation}\label{eq: T leq H small torso}
    |\mathcal T_{k_A, r_A, k_B, r_B}| \leq \sum_{G_A, \mathbf d', S} |\mathcal H_{G_A, \mathbf d', S}|.
\end{equation}

Similarly to the main arguments in \cref{sec: Janson}, we bound $|\mathcal H |$ where $\mathcal H :=\mathcal H_{G_A, \mathbf d', S}$ by counting the number of ways to construct a member $H\in\mathcal{H}$ by adding edges into $H$ in multiple steps that we describe below.

\begin{itemize}
    
    \item[1.](1st stage of choosing cross edges for $S'$) Let 
    \[
    D' = \frac{\varepsilon^7d}{16\log\log\left(\frac{dn}{2k_B}\right)}.
    \]
    Let $S' = \{u \in S \mathrel : d'_u \geq D'\}$. Note that $\sum_{u \in S'}d'_u \geq (1 - \eps)k_B$ by \cref{clm: small torso subset}. For each $u \in S'$, choose a set $A_u$ of $d'_u$ vertices in $A$ and add edges $\{ux \mathrel : x \in  A_u\}$ to $H$. Moreover, ensure that no vertex in $H$ has degree more than $d$.
    \item[2.](2nd stage of choosing cross-edges for $S'$) For each $u \in S'$, choose a set $\mathcal U$ of $d - 2d'(u)$ vertices in $A$ and add edges $\{ux \mathrel : x \in \mathcal U\}$ to $H$. Moreover, ensure that $H$ is simple (i.e. no edge that was already included in $H$ is chosen to be added) and no vertex in $H$ has degree more than $d$. Let $m$ denote the number of edges added in steps 1 and 2.
    \item[3.](degree completion) Let $B' = B \setminus S'$. Complete $H$ to be a (simple) graph on degree sequence $(d,\ldots, d) - \mathbf d'$ and vertex set $[n]$ by including $dn/2 - k_A - k_B - m$ edges in $A \times B'$.
    \item[4.](choosing $G[B]$) Generate a $(k_B, r_B, B, D_B)$-graph $G_B$ on degree sequence $\mathbf d'$ where $S$ is a spine of $G_B$ and no edge of $G_B$ is contained in a triangle when $G_B$ is included in $H$. Include the edges of $G_B$ in $H$.
\end{itemize}

Let $s' := |S'|$.
    As in the proof of \cref{lem: Small torso counting small k}, observe that
    \begin{equation}\label{eq: mS' relaxed}
       1 \leq s' \leq s \quad\text{and } \quad  \max\{0, ds' - 2k_B\} \leq m \leq ds'.
    \end{equation}
Next, we crudely bound the number of choices for edges in steps 1 and 2 combined by 
\begin{equation}\label{eq: steps 1 and 2}
    \binom{as'}{m}.
\end{equation} 
Let $m' = dn/2 - k_A - k_B - m$. By \cref{cor: max sequence}, there are at most 

\begin{equation}\label{eq: step 3}
    O\left(\binom{|B'|}{m'/a}^{a}\binom{a}{m'/|B'|}^{|B'|}\binom{a|B'|}{m'}^{-1}\right)
\end{equation}
choices in step 3. Bounding the number of choices in step 4 requires counting the choices of $G_B$ satisfying the following three properties.
\begin{enumerate}
    \item[(P1)] $G_B$ has degree sequence ${\mathbf d}'$;
    \item[(P2)] $S$ is a spine of $G_B$;
    \item[(P3)] no edge in $G_B$ is contained in a triangle.
\end{enumerate}
Bounding graphs with properties (P1)-(P3) is quite complicated. To do so we impose two more conditions (P4) and (P5) on the choices in steps 1 to 3 above, and we show that the number of choices such that (P4) or (P5) fails (in steps 1--3) and (P1)--(P3) are satisfied (in step 4) is sufficiently small. 

Recall the definition of a well-spread collection of sets in \cref{def: well-spread}. Here, observe that $y'' = y'$. We now define (P4) similarly to (P) in \cref{sec: small torso proof 1}.

\begin{enumerate}
    \item[(P4)] $\{A_u\}_{u\in S'}$ is well-spread. 
\end{enumerate}

To define (P5), we specify one additional parameter and one new definition.
\begin{equation}
    y = \frac{b}{\log^{400/\varepsilon^4}\frac{dn}{2k_B}}
\end{equation}
Let $H_3$ be the set of edges added in step 3.
\begin{definition}
    A set $A_u \subseteq A$ is \textit{expanding} if $|N_{H_3}(A_u)| \geq b - y$.
\end{definition}
Let $S'' = \{u \in S' \mathrel : A_u \text{ is expanding}\}$. 

\begin{enumerate}
    \item[(P5)] $\sum_{u \in S''}d'_u \geq k_B/8$. 
\end{enumerate}
The following two claims show that very few graphs in $H$ fail to satisfy (P4) or (P5).
\begin{claim}\label{clm: P4}
    The number of choices of steps 1--4 such that (P1)--(P3) are satisfied and (P4) fails is at most
    \[
    \exp\left[k_B\log\frac{n}{d}-11\eps^{-2}k'_B\log\frac{dn}{2k_B}\right]\max_{s', m}\left\{\binom{as'}{m}\binom{b - s'}{m'/a}^a\binom{a}{m'/(b - s')}^{b - s'}\binom{a(b - s')}{m'}^{-1} \right\}.
    \]
\end{claim}

\begin{claim}\label{clm: P5}
        The number of choices of steps 1--4 such that (P1)--(P4) are satisfied and (P5) fails is at most
        \[
    \exp\left[k_B\log\frac{n}{d}-7\eps^{-2}k'_B\log\frac{dn}{2k_B}\right]\max_{s', m}\left\{\binom{as'}{m}\binom{b - s'}{m'/a}^a\binom{a}{m'/(b - s')}^{b - s'}\binom{a(b - s')}{m'}^{-1} \right\}.
    \]
\end{claim}

Now we bound the number of choices of steps 1--4 such that (P1)--(P5) hold simultaneously. 

\begin{claim}\label{clm: P3}
    The number of choices of steps 1--4 such that (P1)--(P5) are satisfied is at most
    \[
    \exp\left[k_B\log\frac{n}{d} - 10\eps^{-2}k'_B\log\frac{dn}{2k_B}\right]\max_{s', m}\left\{\binom{as'}{m}\binom{b - s'}{m'/a}^a\binom{a}{m'/(b - s')}^{b - s'}\binom{a(b - s')}{m'}^{-1} \right\}.
    \]
\end{claim}

 Combining \cref{clm: P4}, \cref{clm: P5}, and \cref{clm: P3} gives that $|\mathcal H|$ is bounded from above by 
\begin{equation}
     \exp\left[k_B\log\frac{n}{d}-6\eps^{-2}k_B'\log\frac{dn}{2k_B}\right]\max_{s', m}\left\{\binom{as'}{m}\binom{b - s'}{m'/a}^a\binom{a}{m'/(b - s')}^{b - s'}\binom{a(b - s')}{m'}^{-1} \right\}.\notag
\end{equation}
Thus, by~\eqref{eq: T leq H small torso}, we have 
\begin{align}
    |\mathcal T_{k_A, r_A, k_B, r_B}| \leq \sum_{G_A,\mathbf d', S} &\max_{s', m}\left\{\binom{as'}{m}\binom{b - s'}{m'/a}^a\binom{a}{m'/(b - s')}^{b - s'}\binom{a(b - s')}{m'}^{-1} \right\}\notag \\
    &\times \exp\left[k_B\log\frac{n}{d}-6\eps^{-2}k_B'\log\frac{dn}{2k_B}\right].\label{eq: small torso T}
\end{align}
Simplifying~\eqref{eq: small torso T} gives the following claim, which completes the proof.
\begin{claim}\label{clm: small torso simplification}
    \[
    \frac{|\mathcal T_{k_A, r_A, k_B, r_B}|}{|\mathcal G^{\text{bip}}_{n,d}|} \leq \frac{1}{\binom{n}{n/2}} \exp\left[ -2\left( k'_A\log\frac{dn}{2k_A} + k'_B\log\frac{dn}{2k_B}\right) \right].
    \]
\end{claim}
\cref{clm: P5} is proven in \cref{sec: expansion}. \cref{clm: num deg sequences}, \cref{clm: P4}, and \cref{clm: small torso simplification} are proved below in this subsection. 
\qed

\medskip 

\noindent\textit{Proof of \cref{clm: num deg sequences}.}
The existence of a spine of $G_B$ with size at most $s$ is proved in Lemma~11 in \cite{osthus2003densities}. It remains to bound the number of degree sequences that admit a $(k_B, r_B, B, D_B)$-graph with spine $S$. The total degree of vertices in $S$ will be at least $k_B - k'_B$ by the definition of a spine, so we bound the number of degree sequences by first distributing $k_B - k'_B$ indistinguishable ``balls'' among $|S|$ ``bins'' and then distributing the remaining $k_B + k'_B$ indistinguishable ``balls'' among $a$ ``bins''. Thus, there are at most 
\begin{equation}\label{eq: stars and bars}
    \binom{k_B-k'_B + s - 1}{k_B-k'_B}\binom{k_B + k'_B + b - 1}{k_B + k'_B}
\end{equation}
degree sequences that admit a $(k_B, r_B, B, D_B)$-graph with spine $S$. It remains to simplify~\eqref{eq: stars and bars}. Observe that $ \binom{k_B-k'_B + s - 1}{k_B-k'_B}\binom{k_B + k'_B + b - 1}{k_B + k'_B}$ is bounded from above by 
\begin{align*}
    &\left(\frac{s-1}{k_B-k'_B}\right)^{k_B - k'_B}\left(\frac{b-1}{k_B + k'_B}\right)^{k_B + k'_B}\exp[O(k_B)]\\
    &= \exp\left[k_B\log\frac{(b-1)(s-1)}{k^2_B - k'^2_B} +k'_B\log\frac{(b-1)(k_B - k'_B)}{(s-1)(k_B + k'_B)} + O(k_B)\right]\\
    &= \exp\left[k_B\log\frac{nk'_B}{D_Bk_B^2} + k'_B\log\frac{nD_B}{k_B} + O(k_B)\right]\\
    &\leq \exp\left[(1 + o(1))k'_B\log\frac{dn}{2k_B}\right],
\end{align*}
where the final inequality holds because $k_B/16 \geq a\log^{-4000/\eps^{4}}\frac{dn}{2k_B}$ implies that $D_Bk_B \geq n$.

\qed

\smallskip
\noindent\textit{Proof of \cref{clm: P4}.}
Fix some $(s',m)$ that satisfy~\eqref{eq: mS' relaxed}. This proof is very similar to the proof of \cref{clm: P fails}. In fact, the same argument immediately gives that there are at most 
\begin{equation}\label{eq: step 1 P4}
    \binom{as'}{m}\exp\left[-13\eps^2k'_B\right]
\end{equation}
choices in steps 1 and 2 that lead to (P4) failing. Also, observe that there are at most 
\begin{equation}\label{eq: step 3 P4}
    O\left(\binom{b-s'}{m'/a}^a\binom{a}{m'/(b - s')}^{b - s'}\binom{a(b-s')}{m'}^{-1}\right)
\end{equation}
choices in step 3 by \cref{cor: max sequence}. Finally, by \cref{cor: counting defects}, there are at most 
\begin{equation}\label{eq: step 4 P4}
    \exp\left[k_B\log\frac{n}{d} +(1 + \eps^2)k'_B\log\frac{dn}{2k_B}\right]
\end{equation}
choices in step 4. Hence, by~\eqref{eq: step 1 P4},~\eqref{eq: step 3 P4}, and~\eqref{eq: step 4 P4}, there are at most 
\begin{align*}
    \exp\left[k_B\log\frac{n}{d}-11\eps^{-2}k'_B\log\frac{dn}{2k_B}\right]\max_{s', m}\left\{\binom{as'}{m}\binom{b - s'}{m'/a}^a\binom{a}{m'/(b - s')}^{b - s'}\binom{a(b - s')}{m'}^{-1} \right\}
\end{align*}
graphs in $\mathcal H$ that do not satisfy (P4), completing the proof.

\qed

\medskip

\noindent\textit{Proof of \cref{clm: small torso simplification}.} This proof is very similar to the proof of \cref{clm: small torso final 1}. By \cref{clm: num deg sequences} and~\eqref{eq: small torso T}, we have 
\begin{align*}
    |\mathcal T_{k_A, r_A, k_B, r_B}| \leq& \max_{s', m}\left\{\binom{as'}{m}\binom{b - s'}{m'/a}^a\binom{a}{m'/(b - s')}^{b - s'}\binom{a(b - s')}{m'}^{-1} \right\}\notag \\
    &\times \exp\left[(k_A + k_B)\log\frac{n}{d}-5\eps^{-2}k_B'\log\frac{dn}{2k_B}\right].
\end{align*}
 That is, by~\eqref{eq:reg bip count2}, 
\begin{align}
    \frac{|\mathcal T_{k_A, r_A, k_B, r_B}|}{|\mathcal G_{n,d}^\text{bip}|} \leq& \max_{s', m}\left\{\frac{ \binom{as'}{m}\binom{b - s'}{m'/a}^a\binom{a}{m'/(b - s')}^{b-s'}\binom{n^2/4}{dn/2}}{\binom{a(b - s')}{m'}\binom{n/2}{d}^n\binom{n}{n/2} }\right\}\notag\\
    &\times \exp\left[(k_A + k_B)\log\frac{n}{d} -5\eps^{-2}k'_B\log\frac{dn}{2k_B}\right].\notag
\end{align}
\cref{clm: small torso max} then gives that 
\begin{align}
    \frac{|\mathcal T_{k_A, r_A, k_B, r_B}|}{|\mathcal G_{n,d}^\text{bip}|} \leq \frac{1}{\binom{n}{n/2}}\exp\left[-5\eps^{-2}k'_B\log\frac{dn}{2k_B} + O(k_A + k_B)\right].\label{eq: just fix exponential part}
\end{align}
Since $k'_A\log\frac{dn}{2k_A} \leq 2\eps^{-2}k'_B\log\frac{dn}{2k_B}$, the exponential term in~\eqref{eq: just fix exponential part} is bounded from above by 
\[
\exp\left[-2\left(k'_A \log\frac{dn}{2k_A } + k'_B\log\frac{dn}{2k_B}\right)\right],
\]
which completes the proof.

\qed

\subsection{Proofs of \cref{clm: P holds} and \cref{clm: P5}}\label{sec: expansion}
In both of these proofs, we first reduce the problem to showing that a randomly chosen bipartite graph with a certain fixed degree sequence (that is close to regular) avoids a large set of forbidden edges with sufficiently low probability. Hence, we provide the lemma below, which is proved in \cref{sec: subswitching}.
\begin{lem}\label{lem: switching} Assume \eqref{Assumptions}.  Fix a constant $0 < \alpha < 1/e$ where 
\begin{equation}\label{eq: alpha condition}
    \alpha^2 - 6\alpha^4 - (4  + \alpha^{-4})\mu_0 > \alpha^3.
\end{equation}  Recall that here, $\mu_0 > 0$ is a sufficiently small constant. Let $U$ and $V$ be disjoint vertex sets where $|U|, |V| = \frac{n}{2}(1 + o(1))$ and $X \subseteq U \times V$. Suppose  
\begin{itemize}
    \item[(1)] $d|X| = \omega(n)$;
    \item[(2)] for all $u \in U$, $|\{uw \mathrel : uw \in X\}| \leq \mu_0n$;
    \item[(3)] for all $v \in V$, $|\{wv \mathrel : wv \in X\}| \leq \mu_0n$;
    \item[(4)] there exists $V' \subseteq V$ with $|V'| \geq |V| - 2\mu_0n$ such that for all $v \in V'$, $uv \notin X$ for any $u \in U$.
\end{itemize}   
Let $\mathbf d^*$ be a degree sequence on $U \cup V$ with maximum degree at most $d$ where $\sum_{u \in U \cup V}d^*_u \geq (1 - o(1))dn$. Moreover, assume that for all $u \in U \cup V$ such that $u$ is incident to an edge in $X$, $d^*_u \geq \alpha d$. Let $G$ be a graph chosen uniformly at random from all bipartite graphs with bipartition $(U,V)$ and degree sequence $\mathbf d^*$. Then, 
    \[
    \P[E(G) \cap X = \emptyset]\leq \exp\left[ -\frac{\alpha^3d|X|}{3n}\right].
    \]
\end{lem}

\noindent\textit{Proof of \cref{clm: P holds}.} 
Fix some $(s', m)$ that satisfy~\eqref{eq: mS' relaxed 1}. Fix a graph $H$ that is obtained by steps 1 and 2 of the procedure. Then, there exists some degree sequence $\mathbf d^*$ dependent on $H$ such that the edges chosen in step 3 follow $d^*$. Let $G$ be a random graph chosen uniformly from all graphs with bipartition $(A, B')$ and degree sequence $d^*$. Let $\mathcal E$ be the event that no edge of $G_B$ is contained in a triangle in $H \cup G$. Recall that there are at most \begin{equation}
    O\left(\binom{|B'|}{m'/a}^a\binom{a}{m'/|B'|}^{|B'|}\binom{a|B'|}{m'}^{-1}\right)\notag
\end{equation}
graphs with bipartition $(A,B')$ and degree sequence $d^*$ by \cref{cor: max sequence}. Hence, by \eqref{eq: steps 1 and 2 1}, the number of choices in steps 1--3 is at most 
\begin{equation}
    \max_{s', m}\left\{ \binom{as'}{m}\binom{b - s'}{m'/a}^a\binom{a}{m'/(b - s')}^{b-s'}\binom{a(b - s')}{m'}^{-1} \right\}\max_{H}\{O(\P[\mathcal E])\},
\end{equation}
where the first maximization is over all $(s', m)$ that satisfy~\eqref{eq: mS' relaxed 1} and the second is over all graphs $H$ that can be obtained by the first two steps of the procedure. Thus, it suffices to show that 
\begin{equation}\label{eq: desired prob small torso small k}
    \P[\mathcal E] \leq \exp\left[-\eps^{-4}k'_B\log\frac{dn}{2k_B}\right].
\end{equation}
To do so, we first find a set of well behaved edges whose inclusion in $G$ would guarantee a triangle in $H \cup G$ and bound the probability that $G$ avoids all of these edges using \cref{lem: switching}.

Since $\{A_u\}_{u \in S'}$ is well-spread, there exists a set $S^* \subseteq S'$ and pairwise disjoint sets $\{A_u'\}_{u \in S^*}$ such that 

\begin{equation}\label{eq: well-spread A's}
    A_u' \subseteq A_u, \quad |A_u'| \geq |A_u|/2,\quad  \text{and } \sum_{u \in S^*}|A'_u| \geq k_B/16.
\end{equation}
Let $A^*$ be the set of vertices in $A$ that have degree less than $d/4$ in $d^*$. Since $G_A$ has maximum degree at most $d/2$, each vertex in $A^*$ must have degree at least $d/4$ in $H$. We then have that $d|A^*|/4 \leq d|S'|$, and so $|A^*| \leq 4s$. That is, $|A^*| = o(k_B)$. Hence, by potentially deleting vertices from each $A'_u$,~\eqref{eq: well-spread A's} gives that for each $u \in S^*$, there exists $A^*_u \subseteq A'_u$ such that 
\begin{align}
    \{A^*_u\}_{u \in S^{**}} \text{ are pairwise disjoint}&, \quad \sum_{u \in S^{**}}|A^*_u| \geq k_B(1 - o(1))/16,\notag \\ &\text{and for all $v \in \bigcup_{u \in S^{**}}A^*_u$,  } d^*_v \geq d/4. \label{eq: A* properties}
\end{align}
Let 
\begin{equation}\notag
    X = \bigcup_{u \in S^*}A^*_u \times (N_{G_B}(u) \cap B').
\end{equation}
Thus, 
\[
\P[\mathcal E] \leq \P[E(G) \cap X = \emptyset],
\]
and so we consider bounding $\P[E(G) \cap X = \emptyset]$ instead. We now show that all conditions of \cref{lem: switching} are satisfied when $U = B'$, $V = A$. Note that $s' = o(n)$, so~\eqref{eq:AB} gives that $|A|,|B'| = \frac{n}{2}(1 + o(1))$. 

Note that $|X| \geq \frac{k_BD'}{16}(1 + o(1))$. Then, 
\begin{align*}
    d|X|/n \geq C\frac{d^2k_B}{n\log\log\frac{dn}{2k_B}} \geq C'k_B\frac{\log n}{\log\log\frac{dn}{2k_B}} \gg 1,
\end{align*}
for some positive constants $C$ and $C'$. Thus, (1) is satisfied.

Moreover, observe that each vertex in $X \cap A$ is incident to at most $d \leq \mu_0n$ edges in $X$ since $G_B$ has maximum degree at most $d$ and by the pairwise disjoint property in~\eqref{eq: A* properties}, so (3) is satisfied. Likewise, each vertex in $B'$ is incident to at most $\max_{u \in S^*}\{|A^*_u|\} \leq d$ edges in $X$, so (2) is satisfied. Also, there are at most $k_B = o(n)$ vertices in $A$ that are incident to an edge in $X$, so (4) is satisfied. Finally, by~\eqref{eq: A* properties} and since every vertex in $B'$ has degree at least $d/2$ in $G$, we are able to apply \cref{lem: switching} with $\alpha = 1/4$ (notice that \eqref{eq: alpha condition} holds when $\mu_0$ is sufficiently small). That is, 
\begin{align*}
    \P[E(G) \cap X = \emptyset] &\leq \exp\left[-\frac{dD'k_B}{192n}(1 - o(1))\right]\\
    &= \exp\left[-C\frac{d^2k_B}{n\log\log\frac{dn}{2k_B}}\right]\\
    &\leq \exp\left[-Ck_B\frac{\log n}{\log\log\frac{dn}{2k_B}}\right]\\
    &\leq \exp\left[-\eps^{-7}k_B\log\log\frac{dn}{2k_B}\right]\\
    &\leq \exp\left[ -\eps^{-4}k'_B\log\frac{dn}{2k_B}\right],
\end{align*}
where $C$ is some positive constant (that depends on $\eps$). Thus,~\eqref{eq: desired prob small torso small k} holds, and so the proof is complete.
\qed

\medskip
\noindent\textit{Proof of \cref{clm: P5}.}
Again, fix $(s',m)$ that satisfies~\eqref{eq: mS' relaxed}. As in the proof of \cref{clm: P holds}, let $H$ be a graph obtained from steps 1 and 2 of the procedure and let $\mathbf d^*$ be the degree sequence of the edges to be added in step 3. Let $G$ be a random bipartite graph with bipartition $(A, B')$ and degree sequence $\mathbf d^*$ chosen uniformly at random. Let $\mathcal E$ be the event that $\sum_{u \in S''}d'_u \leq k_B/8$. By \cref{cor: max sequence}, there are at most 
\[
O\left(\binom{b-s'}{m'/a}^a\binom{a}{m'/(b - s')}^{b - s'}\binom{a(b - s')}{m'}^{-1}\right)
\]
bipartite graphs with bipartition $(A, B')$ and degree sequence $d^*$, so there are at most 
\begin{equation}\label{eq: P4 step 3}
    \binom{b-s'}{m'/a}^a\binom{a}{m'/(b - s')}^{b - s'}\binom{a(b - s')}{m'}^{-1}O(\P(\mathcal E))
\end{equation}
choices in step 3. Also, there are at most 
\begin{equation}
    \exp\left[k_B\log\frac{n}{d} + (1 + \eps^2)k'_B\log\frac{dn}{2k_B}\right]
\end{equation}
choices in step 4 by \cref{cor: counting defects}. Hence, by~\eqref{eq: steps 1 and 2}, there are at most 
\begin{align}
    \max_{s',m}&\left\{\binom{b-s'}{m'/a}^a\binom{a}{m'/(b - s')}^{b - s'}\binom{a(b - s')}{m'}^{-1}\right\}\max_{H}\{O(\P(\mathcal E))\}\notag\\
    &\times \exp\left[k_B\log\frac{n}{d} + (1 + \eps^2)k'_B\log\frac{dn}{2k_B}\right]\notag
\end{align}
graphs in $\mathcal H$ for which (P5) fails. Thus, it suffices to show that 
\begin{equation}\label{eq: desired prob P5}
    \P[\mathcal E] \leq \exp\left[-9\eps^{-2}k'_B\log\frac{dn}{2k_B}\right].
\end{equation}
Recall that $\sum_{u \in S'}d'_u \geq (1 - \eps)k_B$ by \cref{clm: small torso subset}. Let $S'''\subseteq S'$ be the set of vertices $u\in S'$ where $A_u$ does not expand. Observe that if $\sum_{u \in S''}d'_u < k_B/8$, then $\sum_{u \in S'''}d'_u \geq k_B/8$. Thus, it suffices to show that for all sets $S^* \subseteq S'$ with $\sum_{u \in S^*}d'_u \geq k_B/8$, there exists  $v \in S^*$ such that $A_v$ expands. In fact, expansion is a upwards closed property, so it suffices to show that $A_v$ contains a subset that expands. 

Fix $S^* \subseteq S'$ with $\sum_{u \in S^*}d'_u \geq k_B/8$. Let $\mathcal E_{S^*}$ be the event that for all $u \in S^*$, $u$ does not expand. Since $\{A_u\}_{u \in S'}$ is well-spread, there exists $S^{**} \subseteq S^*$ and pairwise disjoint sets $\{A'_u\}_{u \in S^{**}}$ such that
\begin{equation}
    A'_u \subseteq A_u, \quad |A'_u| \geq |A_u|/2, \quad \text{and }\sum_{u \in S^{**}}|A'_u| \geq y'/2.
\end{equation}
We may further assume that 
\begin{equation}\label{eq: S** not too big}
    \sum_{u \in S^{**}}|A'_u| \leq y'/2 + d \leq 2\mu_0 n
\end{equation}
by potentially removing vertices from $S^{**}$.

Recall that $s' \leq 2k'_B/D_B < O\left(\frac{n\log d\log\log}{d}\right) < y$, so $y > s'$. Then, for any $v \in S^{**}$, $A_v$ expands if there is an edge induced by $A_v \cup B_v$ in $G$ for every $B_v \subseteq B'$ with $|B_v| = y-s'$. Let $\mathcal F(\{B_u\}_{u \in S^{**}})$ be the event that $A_u \cup B_u$ induces no edges in $G$ for all $u \in S^{**}$. 
\begin{equation}\label{eq: prob S*}
    \P[\mathcal E_{S^*}] \leq \binom{b - s'}{y - s'}^{|S^{**}|}\max_{\{B_u\}_{u \in S^{**}}}\P[\mathcal F(\{B_u\}_{u \in S^{**}})].
\end{equation}
Observe that $d|S^{**}|\geq y'$, so $|S^{**}|\geq y'/d$. Thus,
\begin{equation}\label{eq: bounding Bu}
    \binom{b - s'}{b - y}^{|S^{**}|} \leq \binom{b}{y}^{y'/d} \leq \left(\frac{eb}{y}\right)^{yy'/d}.
\end{equation}
Now, fix sets $\{B_u \subseteq B'\}_{u \in S^{**}}$ such that $|B_u| = y - s'$ for all $u \in S^{**}$.  We upper bound $\P[\mathcal F(\{B_u\}_{u \in S^{**}})]$ by finding a large, well-behaved set of forbidden edges in $G$ and using \cref{lem: switching} to upper bound the probability that none of those edges are included in $G$. It is natural to take $\bigcup_{u \in S^{**}}A'_u \times B_u$ to be those forbidden edges, but as in the proof of \cref{clm: P holds} there may be vertices in $A$ that are incident to a forbidden edge that also have small degree in $\mathbf d^*$. But again, there will be very few of these vertices. Let $A^*$ be the set of vertices in $A$ that have degree less than $d/4$ in $\mathbf d^*$. Then, since $G_A$ has maximum degree at most $d/2$, each vertex in $A^*$ has degree at least $d/4$ in $H$. Thus, as in the proof of \cref{clm: P holds}, $|A^*| = o(k_B)$. Hence, taking $A^*_u = A'_u \setminus A^*$ for all $u \in S^{**}$, we have that 
\begin{align}
    \{A^*_u \neq \emptyset\}_{u \in S^{**}} \text{ are pairwise disjoint},& \quad \sum_{u \in S^{**}}|A^*_u| \geq y'(1 - o(1))/2, \quad \notag\\
    &\text{and $d^*_v \geq d/4$ for all $v \in \bigcup_{u \in S^{**}}A^*_u$.} \label{eq: A* properties 2}
\end{align}
Let $X = \bigcup_{u \in S^{**}}A^*_u \times B_u$. Next, we show that all conditions of \cref{lem: switching} are satisfied with $U = B'$, $V = A$, and $\alpha = 1/4$. Observe that $d|X| = \Omega(dyy') = \omega(n)$, so (1) is satisfied. Also, every vertex in $B$ is incident to at most $y = o(n)$ edges in $X$, and so (2) holds. Likewise, every vertex in $A$ is incident to at most $d \leq \mu_0n$ edges in $X$, so (3) holds. Finally, by~\eqref{eq: S** not too big}, we have $|\bigcup_{u \in S^{**}}A^*_u| \leq o(n) + d \leq 2\mu_0n$. Since no other vertex in $A$ is incident to an edge in $X$, (4) holds as well. Finally, all conditions on $\alpha$ hold for $\alpha = 1/4$ by~\eqref{eq: A* properties 2} and since \eqref{eq: alpha condition} holds when $\mu_0$ is sufficiently small. Thus, 
\begin{equation}\notag
    \P[\mathcal F(\{B_u\}_{u \in S^{**}})]\leq\P[E(G) \cap X = \emptyset] \leq \exp\left[-\frac{d|X|}{192n}\right].
\end{equation}

Thus, by~\eqref{eq: prob S*} and~\eqref{eq: bounding Bu}, 
\begin{align}
    \P[\mathcal E_{S^*}] &\leq \exp\left[\frac{yy'}{d}\log\frac{eb}{y} - \frac{dyy'}{192n}(1 + o(1))\right]\notag\\
    &\leq\exp\left[\frac{dyy'}{n}\left(\frac{3}{(1 + \eps)^24\log n} + \frac{3\log\log^{400/\eps^4}\frac{dn}{2k_B}}{(1 + \eps)^24\log n} - (1 + o(1))/192\right)\right]\label{eq:1}\\
    &\leq \exp\left[-\frac{dn(1 + o(1))}{192\log^{4400/\eps^4}\frac{dn}{2k_B}}\right]\notag\\
    &\leq \exp\left[-10\eps^{-2}k'_B\log\frac{dn}{2k_B}\right]\label{eq: Final S^* bound},
\end{align}
where~\eqref{eq:1} holds since $d^2 \geq (1 + \eps)^2\frac{3}{4} n\log n$ and~\eqref{eq: Final S^* bound} holds because $k_B \leq n\log d$ implies $dn \geq dk_B/\log d$. Thus, by the union bound, 
\[
\P[\mathcal E] \leq 2^s\exp\left[-10\eps^{-2}k'_B\log\frac{dn}{2k_B}\right] \leq \exp\left[-9\eps^{-2}k'_B\log\frac{dn}{2k_B}\right],
\]
where the second inequality holds because $s = o(k_B)$. Thus,~\eqref{eq: desired prob P5} holds, and so the proof is complete.
\qed

\subsection{Proof of~\cref{clm: P3}}
Fix $(s', m)$ that satisfies~\eqref{eq: mS' relaxed}. For any $H$ that can be constructed via the first 3 steps of the procedure, let $N_H$ denote the number of choices in step 4. By~\eqref{eq: steps 1 and 2} and~\eqref{eq: step 3}, there are at most 
\begin{equation}\notag
    \max_{s', m}\left\{\binom{as'}{m}\binom{b - s'}{m'/a}^a\binom{a}{m'/(b - s')}^{b - s'}\binom{a(b - s')}{m'}^{-1} \right\}\max_{H}\{N_H\},
\end{equation}
graphs in $\mathcal H$ that satisfy (P1)--(P5), where the second maximization is over all $H$ that can be obtained by steps 1--3 of the procedure that satisfy (P4) and (P5).
Hence, it suffices to show that 
\begin{equation}\label{eq: desried count}
    N_H \leq \exp\left[k_B\log\frac{n}{d} - 10\eps^{-2}k'_B\log\frac{dn}{2k_B}\right].
\end{equation}
Here, we obtain an upper bound on $N_H$ by counting defect graphs that may or may not satisfy (P1). For all $u \in S''$, let $B_u = \{v \in B' \mathrel : vw \notin E(H) \text{ for any $w \in A_u$}\}$. Then, since each $A_u$ expands, each $B_u$ has size at most $y - s'$ (recall that this is positive by the calculation below~\eqref{eq: S** not too big}). Thus, if $H \cup G_B$ has no triangle containing an edge in $G_B$, it must be the case that $N_{G_B}(u) \subseteq S' \cup B_u$ for all $u \in S''$. We show that there are few $(k_B, r_B, B, D_B)$-graphs $G_B$ where $S$ is a spine of $G_B$ that satisfy this property. In fact, it is shown in Equation (45) and (46) of \cite{osthus2003densities} that there are at most 
\begin{equation}\label{eq: osthus U}
    b^{2s + 2}\binom{sy}{k_B/8}\binom{bs}{7k_B/8 - k'_B}\binom{\binom{b}{2}}{k'_B}
\end{equation}
such graphs. It remains to simplify~\eqref{eq: osthus U}. Also, by Proposition 11 in \cite{osthus2003densities}, there are at most 
\begin{equation}\notag
    b^s\binom{bs}{k_B - k'_B}\binom{\binom{b}{2}}{k'_B}    
\end{equation}
$(k_B, r_B,B, D_B)$-graphs. That is, by \cref{cor: counting defects}, 
\begin{equation}
    b^s\binom{bs}{k_B - k'_B}\binom{\binom{b}{2}}{k'_B} \leq \exp\left[k_B\log\frac{n}{d} + (1 + \eps^2)k'_B\log\frac{dn}{2k_B}\right].
\end{equation}
Thus, we have
{\allowdisplaybreaks \begin{align}
    \text{\eqref{eq: osthus U}} &\leq b^{s + 2}\frac{\binom{sy}{k_B/8}\binom{bs}{7k_B/8 - k'_B}}{\binom{bs}{k_B- k'_B}}\exp\left[k_B\log\frac
    nd + (1 + \eps^2)k'_B\log\frac{dn}{2k_B}\right]\notag\\
    &\leq \left(\frac{sy}{k_B/8}\right)^{k_B/8}\left(\frac{bs}{7k_B/8 - k'_B}\right)^{7k_B/8 - k'_B}\left(\frac{k_B - k'_B}{bs}\right)^{k_B - k'_B}\notag\\&\hspace{2cm}\times\exp\left[k_B\log\frac{n}{d} + (1 + \eps^2)k'_B\log\frac{dn}{2k_B} + O(k_B)\right]\notag\\
    &= \left(\frac{y}{b}\right)^{k_B/8}\exp\left[k_B\log\frac{n}{d} + (1 + \eps^2)k'_B\log\frac{dn}{2k_B} + O(k_B)\right]\notag\\
    &= \exp\left[k_B \log\frac{n}{d}-\frac{50k_B}{\eps^4}\log\log\frac{dn}{2k_B} + (1 + \eps^2)k'_B\log\frac{dn}{2k_B}\right]\notag\\
    &\leq \exp\left[k_B\log\frac{n}{d} - 10\eps^{-2}k'_B\log\frac{dn}{2k_B}\right],\notag
\end{align}}
where the final inequality holds by~\eqref{eq: r1}. This completes the proof.

\qed

\subsection{Switching: proof of \cref{lem: switching}}\label{sec: subswitching}

\noindent\textit{Proof of \cref{lem: switching}.}
    Let $n' = |U| + |V|$. For any $\chi \subseteq X$, let $\mathcal G_{\chi}$ denote the set of bipartite graphs $G$ with bipartition $(U,V)$ on degree sequence $\mathbf d^*$ where $E(G) \cap X = \chi$. For all $0 \leq i \leq |X|$, let $\mathcal G'_i = \bigcup_{\chi \in \binom{X}{i}} \mathcal G_\chi$ (where $\binom{X}{i}$ is the set of subsets of $X$ with size $i$). One natural approach to this proof is to first show that 
    \begin{equation}\label{eq: switching natural}
        \frac{|\mathcal G'_{i}|}{|\mathcal G'_{i-1}|} \geq q
    \end{equation}
    for some $q = q(i,n)$, and then to upper bound $|\mathcal G'_0|/\sum_{i = 0}^{|X|}|\mathcal G'_i|$. To establish~\eqref{eq: switching natural}, though, we must consider only sets $\chi$ whose maximum degree is not too large (we may conveniently consider $\chi$ as a graph). For each $0 \leq i \leq |X|$, let $\mathcal X_i \subseteq \binom{X}{i}$ be the set of graphs $\chi \subseteq \binom{X}{i}$ where every vertex in $U \cup V$ is incident to at most $\mu_0d/\alpha^4$ edges in $\chi$. Then, let $\mathcal G_i = \bigcup_{\chi \in \mathcal X_i}\mathcal G_{\chi}$. The following claim allows us to consider only graphs $\chi$ who have small maximum degree.
    
    \begin{claim}\label{clm: chi degree}
        \[
        \sum_{i = 0}^{|X|}|\mathcal G'_i| = (1 + o(1))\sum_{i = 0}^{|X|}|\mathcal G_i|.
        \]
    \end{claim}

    Now, we establish a switching result akin to~\eqref{eq: switching natural} with $\mathcal G_i$ rather than $\mathcal G'_i$.
    
    \begin{claim}\label{clm: switching}
        For all $0 < i \leq \frac{3d}{n}|X|$, 
        \begin{equation}
        \frac{|\mathcal G_i|}{|\mathcal G_{i-1}|} \geq \frac{2\alpha^3d|X|}{in'}\left(1 - \frac{i}{|X|}\right).\notag
    \end{equation}
    \end{claim}
    Then, since 
    \begin{equation}
        \frac{|\mathcal G_i|}{|\mathcal G_0|} = \prod_{j = 1}^{i}\frac{|\mathcal G_{j}|}{|\mathcal G_{j-1}|}, \notag
    \end{equation}
    \cref{clm: switching} gives 
    \begin{equation}\label{eq: d0/di}
        \frac{|\mathcal G_i|}{|\mathcal G_0|} \geq \left(\frac{2\alpha^3d|X|(1 - i/|X|)}{n'}\right)^{i}\frac{1}{i!}.
    \end{equation}
    Therefore, by \cref{clm: chi degree} and since $\mathcal G'_0 = \mathcal G_0$,
    \begin{align}
       \frac{|\mathcal G_0'|}{\sum_{i = 0}^{|X|}|\mathcal G'_i|} = \frac{|\mathcal G_0|}{(1 + o(1))\sum_{i = 0}^{|X|}|\mathcal G_i|} &\leq \frac{|\mathcal G_0|}{\sum_{i = 0}^{3d|X|/n}|\mathcal G_i|}(1 + o(1)) \label{eq: abused sum 1}\\
        &= \left[ \sum_{i = 0}^{ 3d|X|/n}  \frac{|\mathcal G_i|}{|\mathcal G_0|}\right]^{-1}(1 + o(1))\label{eq: abused sum 2}\\
        &\leq \left[ \sum_{i = 0}^{3d|X|/n} \left(\frac{2\alpha^3d|X|\left(1 - 3d/n\right)}{n'}\right)^i\frac{1}{i!} \right]^{-1}(1 + o(1)),\label{eq: abused sum 3}
    \end{align}
    where the summations in~\eqref{eq: abused sum 1},~\eqref{eq: abused sum 2}, and~\eqref{eq: abused sum 3} are taken to $\lfloor 3d|X|/n\rfloor$ if $3d|X|/n$ is not an integer. Taylor's theorem gives that for any $y > 0$ and integer $\ell > 0$, 
    \[
    e^y \leq \sum_{i = 0}^\ell\frac{y^i}{i!} + e^y\frac{y^{\ell + 1}}{(\ell + 1)!}.
    \]
    Thus, we have that $ \sum_{i = 0}^{3d|X|/n}\left(\frac{2\alpha^3d|X|(1 - 3d/n)}{n'}\right)^i\frac{1}{i!} $ is bounded from below by
    \begin{align}
       &\exp\left[\frac{2\alpha^3d|X|(1 - 3d/n)}{n'}\right]\left[1 - \left(\frac{2\alpha^3d|X|(1 - 3d/n)}{n'}\right)^{3d|X|/n + 1}\frac{1}{(3d|X|/n + 1)!}\right]\notag\\
        &\geq \exp\left[\frac{\alpha^3d|X|}{n'}\right]\left[1 - \left(\frac{2\alpha^3e}{3}\right)^{3d|X|/n}O(d|X|/n)\right]\label{eq: summation manipulation 2}\\
        &= \exp\left[\frac{\alpha^3d|X|}{n'} - \left(\frac{2\alpha^3e}{3}\right)^{\Omega(d|X|/n)}\right]\notag\\
        &\geq \exp\left[ \frac{\alpha^3d|X|}{2n'}\right],\label{eq: summation maniulation 3}
    \end{align}
    where~\eqref{eq: summation manipulation 2} holds because $(3d|X|/n)! \leq \left[3d|X|/en\right]^{3d|X|/n}$ and ~\eqref{eq: summation maniulation 3} holds because $d|X|  = \omega(n)$ and $\frac{\alpha^3e}{3} < 1$ (since $\alpha < 1/e$) imply that 
    \[
    \frac{\alpha^3d|X|}{2n'} \geq \left(\frac{2\alpha^3e}{3}\right)^{\Omega(d|X|/n)}.
    \]
    Thus, combining~\eqref{eq: abused sum 3} and~\eqref{eq: summation maniulation 3} gives that 
    \begin{equation}\notag
        \frac{|\mathcal G'_0|}{\sum_{i = 1}^{|X|}|\mathcal G'_i|} \leq \exp\left[-\frac{\alpha^3d|X|}{2n'}\right] \leq \exp\left[-\frac{\alpha^3d|X|}{3n}\right],
    \end{equation}
    (because $n' = n(1 + o(1))$) as desired.
\qed
\smallskip 

    \noindent\textit{Proof of \cref{clm: chi degree}}
        Let $\mathcal H$ be the set of graphs with bipartition $(U,V)$ and degree sequence $\mathbf d^*$. Let $H$ be a graph chosen uniformly at random from $\mathcal H$. Let $\mathcal E$ be the event that no vertex in $H$ is incident to at least $\mu_0\alpha^{-4}d$ edges in $E(H) \cap X$. Then, it suffices to show that $\P[\mathcal E] = 1 - o(1)$. To prove this, we first show that the probability that a fixed vertex in $H$ is incident to at least $\mu_0\alpha^{-4}d$ edges in $E(H) \cap X$ is very small, and then use the union bound to bound the probability that any vertex in $H$ is incident to at least $\mu_0\alpha^{-4}d$ edges in $E(H) \cap X$. Fix $x \in U \cup V$ and let $\mathcal E_x$ be the event that $x$ is incident to at least $\mu_0\alpha^{-4}d$ edges in $E(H) \cap X$. We first show that
        \begin{equation}\label{eq: desired prob chi degree}
            \P[\mathcal E_x] \leq \exp[-\Omega(d)]
        \end{equation}
        using a switching argument. Then, since $\exp[-\Omega(d)]= o(1/n)$, the union bound gives that $\P[\mathcal E] = o(1)$, as desired. Hence, it suffices to show~\eqref{eq: desired prob chi degree}. Recall that every vertex in $U \cup V$ is incident to at most $\mu_0n$ edges in $X$ by assumption. For $0 \leq i \leq \mu_0n$, let $\mathcal H_i\subseteq \mathcal H$ be the set of graphs $G$ in $\mathcal H$ such that $x$ is incident to exactly $i$ edges in $E(G) \cap X$. Moreover, let $h_i = |\mathcal H_i|$. For each $0  < i \leq \mu_0n$, let $\sim$ be a relation on $\mathcal H_i \times \mathcal H_{i-1}$ where $G \sim G'$ if there exist distinct vertices $y,z,w \in (U\cup V)\setminus\{x\}$ such that
        \begin{enumerate}
            \item[(i)] $xy \in E(G) \cap X$;
            \item[(ii)] $yz, xw \notin E(G) \cup X$;
            \item[(iii)] $zw \in E(G)$;
            \item[(iv)] $G' = G - xy - zw + yz + xw$.
        \end{enumerate}

        \begin{figure}[h]
        \centering{
        \resizebox{0.5\textwidth}{!}{
        \includegraphics[]{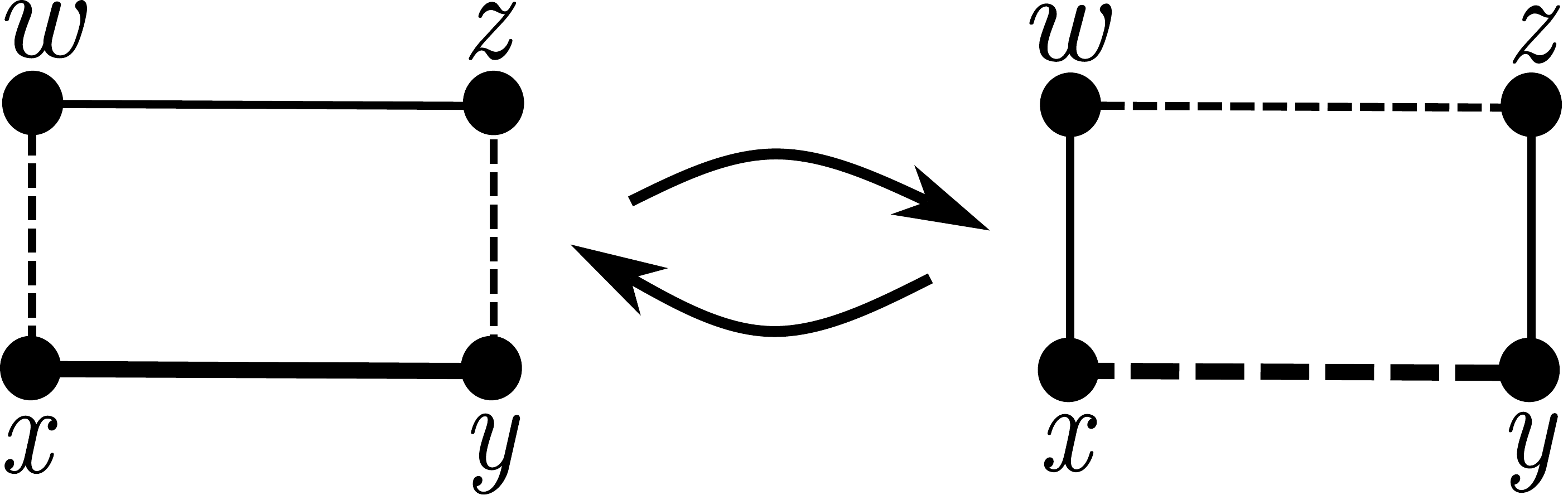}}}
        
        \caption[Switching away forbidden edges incident to ``bad'' vertices.]{Switching away forbidden edges incident to ``bad'' vertices. The bold edge is a member of $X$ and $xz \notin X$. All other edges may be in $X$ or not in $X$.}
        \label{fig: switching_1}
\end{figure}
        
        Let $\mathscr H$ be the bipartite graph with bipartition $(\mathcal H_i, \mathcal H_{i-1})$ with edges determined by $\sim$. Let $\underline{d}_i$ be the minimum degree of a graph in $\mathcal H_i$ in $\mathscr H$ and $\bar d_{i-1}$ be the maximum degree of a graph in $\mathcal H_{i-1}$ in $\mathscr H$. Then, $\underline{d}_ih_i \leq \bar d_{i-1}h_{i-1}$, and so 
        \begin{equation}\label{eq: switching h to d}
            h_i/h_{i-1} \leq \bar{d}_{i-1}/\underline{d}_i.
        \end{equation}
        Notice that each edge in $E(G) \cap X$ is incident to a vertex in $V \setminus V'$, and so there are at most $2\mu_0dn$ edges in $E(G) \cap X$. Next, we show that
        \begin{equation}\label{eq: switching deg d_i}
            \underline{d}_{i} \geq idn/2(1 + o(1)) - 2i\mu_0dn - 2id^2 \geq \frac{idn}{2}(1 - 9\mu_0).
        \end{equation}
The number of choices for $y$ satisfying condition (i) is exactly $i$. The number of choices for $(z,w)$ satisfying (iii) is asymptotic to $dn/2$ by the assumption $\sum_{u \in U \cup V}d^*_u \geq (1 - o(1))dn$, and the fact that $z$ needs to be in the same part as $x$ in the bipartition due to (iv). This gives a first approximation of $\underline{d}_i$ by the inclusion-exclusion principle. To obtain a lower bound for it, we subtract the number of choices of $(z,w)$ such that condition (ii) is violated from $i(dn/2)(1+o(1))$. Given $y$, there are at most $\mu_0 n+d$ choices for $z$ such that $yz\in E(G)\cup X$ and the same bound holds for the choices for $w$ such that $xw\in E(G)\cup X$. Hence, the number of choices for $(w,z)$ violating (ii) is at most $2i(\mu_0 n+d)d$. Subtracting this upper bound from the first approximation of $\underline{d}_i$ gives~\eqref{eq: switching deg d_i}. The second inequality in~\eqref{eq: switching deg d_i} holds because $d \leq \mu_0 n$. On the other hand, for any $G' \in \mathcal H_{i-1}$, we can upper bound $\bar d_{i-1}$ by counting choices of $y,z,w$ where $xy \in X$, $yz,xw \in E(G')$. That is,
        \begin{equation}\label{eq: switching deg d_i-1}
            \bar d_{i-1} \leq \mu_0nd^2
        \end{equation}
        since there are at most $\mu_0n$ edges in $X \setminus E(G)$ incident to $x$ and both $x$ and $y$ have at most $d$ neighbours in $G'$. 
        Let $\hat\imath=6\mu_0d$ and observe that $\alpha^{-4}>6$ by the assumption that $\alpha<1/e$ and thus $\alpha^{-4}\mu_0 d-\hat\imath=\Omega(d)$. Thus, by~\eqref{eq: switching h to d},~\eqref{eq: switching deg d_i}, and~\eqref{eq: switching deg d_i-1}, 
        \begin{equation}
            h_i/h_{i-1} \leq \frac{2\mu_0d}{i(1 - 9\mu_0)} \leq \frac{1}{2},\quad \text{for all $i\ge \hat\imath$}.
        \end{equation}
Thus,
        \begin{align*}
            \P[\mathcal E_x] =\sum_{i = \mu_0\alpha^{-4}d}^{\mu_0n}\frac{h_i}{|\mathcal H|} \leq \sum_{i = \mu_0\alpha^{-4}d}^{\mu_0n}\frac{h_i}{h_{\hat\imath}}\leq \sum_{i = \mu_0\alpha^{-4}d}^{\mu_0n}(1/2)^{i - \hat\imath} = \exp[-\Omega(d)].
        \end{align*}
      Hence,~\eqref{eq: desired prob chi degree} holds, and so the proof is complete.
    \qed

\medskip
\noindent\textit{Proof of \cref{clm: switching}.}
     We prove this using another switching argument. Fix $1 \leq i \leq |X|$ and consider a relation $\sim$ on $\mathcal G_i \times \mathcal G_{i-1}$ where $G' \sim H'$ if and only if there exist distinct vertices $w,y\in U$, $x,z \in V$ such that:
    \begin{itemize}
        \item[(i)] $wx \in E(G') \cap X$;
        \item[(ii)] $yz \in E(G') \setminus X$;
        \item[(iii)] $wz \notin E(G') \cup X$;
        \item[(iv)] $yx \notin E(G') \cup X$;
        \item[(v)] $G' -wx- yz + wz+ yx = H'$. 
    \end{itemize}

    \begin{figure}[h]
        \centering{
        \resizebox{0.5\textwidth}{!}{
        \includegraphics[]{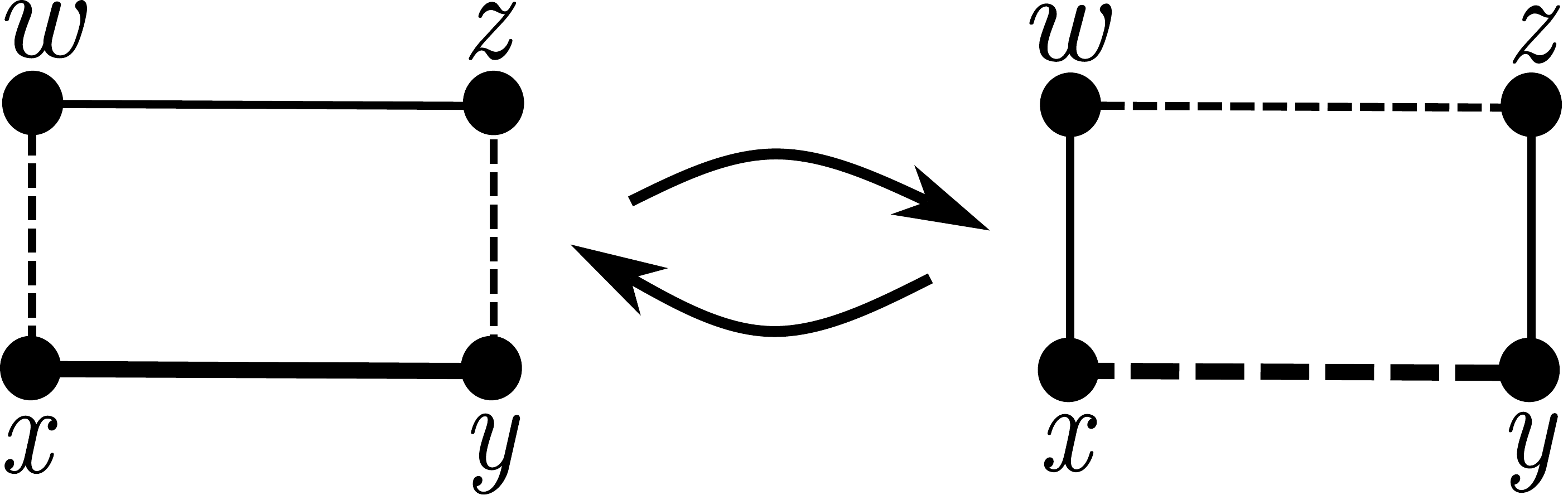}}}
        
        \caption[Switching away forbidden edges.]{Switching away forbidden edges. The bold edge is a member of $X$ and all other edges must not be members of $X$.}
        \label{fig: switching 2}
\end{figure}
    
    Let $\mathscr G$ be a bipartite graph with bipartition $(\mathcal G_i,\mathcal G_{i-1})$ and edges determined by $\sim$. We obtain a lower bound on $\frac{|\mathcal G_i|}{|\mathcal G_{i-1}|}$ by bounding the average degree in $\mathscr G$ of vertices in $\mathcal G_i$ and vertices in $\mathcal G_{i-1}$. More precisely, letting $f_i = \frac{1}{|\mathcal G_i|}\sum_{G \in \mathcal G_i}d_\mathscr G(G)$ and $r_{i-1} = \frac{1}{|\mathcal G_{i-1}|}\sum_{H \in \mathcal G_{i-1}}d_\mathscr G(H)$, we have $\frac{|\mathcal G_i|}{|\mathcal G_{i-1}|} = \frac{r_{i-1}}{f_i}$. Fix $G \in \mathcal G_i$. Then, 
    \begin{equation}
        d_\mathscr G(G) = \left|\left\{\{w,x,y,z\} \in \binom{V(G)}{4} \mathrel : \text{(i), (ii), (iii), and (iv) hold}\right\}\right|. \notag
    \end{equation}
    By first choosing an edge $wx \in E(G) \cap X$ and then choosing any other edge in $G$, we have that 
    \begin{equation}
        d_{\mathscr G}(G) \leq i\frac{dn'}{2}, \notag
    \end{equation}
    since the maximum degree of $G$ is at most $d$. This implies that
    \begin{equation}\label{eq: di upper}
        f_i \leq i\frac{dn'}{2}.
    \end{equation}
    Next we obtain a lower bound on $r_{i-1}$.  
    \begin{claim}\label{clm: switching expectation}
        \[
 r_{i-1} \geq (|X| - i)\alpha^3d^2.
        \]
    \end{claim}
By (\ref{eq: di upper}) and \cref{clm: switching expectation}, we have 
    \begin{equation*} 
        \frac{|\mathcal G_i|}{|\mathcal G_{i-1}|}=\frac{r_{i-1}}{f_i} \geq \frac{2\alpha^3d|X|}{in'}\left(1 - \frac{i-1}{|X|}\right),
    \end{equation*}
    as desired.
    \qed

\medskip
\noindent\textit{Proof of \cref{clm: switching expectation}}
Fix $\chi \in \mathcal {X}_{i-1}$. Let $H$ be chosen uniformly at random from $\mathcal {G}_\chi$.
Then, it suffices to show that
$ \mathbb E[d_{\mathscr G}(H)] \ge (|X| - i + 1)\alpha^3d^2$. 
       To prove this, we obtain a lower bound on the expected number of choices of switchings that produce $H$. Since we only need a lower bound, we consider a subset of switchings defined below. Consider the relation $\sim'$ on ${\mathcal G}_i\times {\mathcal G}_{i-1}$ where $G' \sim' H'$ if and only if there exist distinct vertices $w,y\in U$, $x,z \in V$ such that:
        \begin{itemize}
            \item[(I)] $wx \in X \setminus E(H')$;
            \item[(II)] $z \in V'$;
            \item[(III)] $yz \notin E(H')$;
            \item[(IV)] $wz \in E(H')$;
            \item[(V)] $yx \in E(H')\setminus X$;
            \item[(VI)] $H' - wz- yx + wx+ yz = G'$.
        \end{itemize}
Note that $z\in V'$ implies that $yz,wz\notin X$ by definition of $V'$. Hence, $\sim'$ is the restriction of $\sim$ by requiring $z\in V'$. Namely, for any $H'\in \mathcal{G}_{i-1}$, if $G'\sim' H'$ then $G'\sim H'$. 
Let $\mathscr{G}'$ be the bipartite graph on $(\mathcal{G}_i,\mathcal{G}_{i-1})$ with respect to the new relation $\sim'$. Then immediately we have $ \mathbb E[d_{\mathscr G}(H)] \ge  \mathbb E[d_{\mathscr G'}(H)] $.
Observe also that for any $H'\in {\mathcal G}_{i-1}$, a choice of $(x,y,z,w)$ satisfying (I')--(V') does not necessarily produce $G'\in {\mathcal G}_i$ since the switching of the edges from $H'$ to $G'$ may cause the maximum degree of $G'\cap X$ to exceed $\alpha^{-4}\mu_0 d$, and so we need to be careful when lower bounding the number of choices satisfying (I')--(V') where the resulting graph $G'$ produced by the switching is in ${\mathcal G}_{i}$.
By definition of ${\mathscr G}'$,

       
    \begin{align}
        d_{\mathscr{G}'}(H) 
        &= \left|\left\{\{w,x,y,z\} \in \binom{V(H)}{4} \mathrel : \text{(I)--(VI) hold for some $G'\in {\mathcal G}_{i}$}\right\}\right|.\notag\\
        &\geq |S_1| - |S_2| - |S_3| - |S_4| - |S_5|\label{eq: dR(H)}
    \end{align}
    where 
    \begin{itemize}
        \item $S_1 = \left\{\{w,x,y,z\} \in \binom{V(H)}{4} \mathrel : \text{(I) and (IV) hold and $yx \in E(H)$}\right\}$;
        \item $S_2 = \left\{\{w,x,y,z\} \in S_1 \mathrel : yz \in E(H)\right\}$;
        \item $S_3 = \left\{ \{w,x,y,z\} \in S_1 \mathrel : yx \in X \right\}$;
        \item $S_4 = \left\{ \left\{w,x,y,z\right\} \in S_1 \mathrel : z \notin V' \right\}$;
        \item $S_5 = \{\{w,x,y,z\} \in S_1 \mathrel : H' - wz - yx + wx + yz \notin \mathcal G_i\}$.
    \end{itemize}
    Observe that for any $u, v \in V(H)$ such that $uv \notin X$, we have 
    \begin{equation}\label{eq: edge prob}
        \P[uv \in E(H)] \leq \frac{2d}{n}(1 +O(\mu_0))
    \end{equation}
    by applying \cref{thm: bip edge probs} with $H_1 = \chi$ and $H_2 = X \setminus \chi$ (notice that $\alpha^{-4}\mu_0$ is still sufficiently small and satisfies the assumptions of $\mu_0$ in \cref{thm: bip edge probs}). We then have
    \begin{align}
        \mathbb E|S_1| &\geq \alpha^2(|X| - i + 1)d^2, \label{eq: ineq 1}\\
        \mathbb E |S_2| &\leq (|X| - i + 1)\frac{2d^3}{n}(1 + O(\alpha^{-4}\mu_0)) \leq 2\mu_0(|X| - i)d^2(1 + O(\alpha^{-4}\mu_0)),\label{eq: ineq 2}\\
        \mathbb E|S_3| &\leq \alpha^{-4}\mu_0(|X| - i + 1)d^2, \label{eq: ineq 3}\\
        \mathbb E |S_4| &\leq (|X| - i + 1)d\left(\frac{2d(|V| - |V'|)}{n}(1 + O(\alpha^{-4}\mu_0))\right) \leq 2\mu_0(|X| - i + 1)d^2(1 + O(\alpha^{-4}\mu_0)),\label{eq: ineq 4}
    \end{align}
    where~\eqref{eq: ineq 1} holds since $wx \in X$ implies $d^*_w, d^*_x \geq \alpha d$,~\eqref{eq: ineq 2} holds since $d/n \leq \mu_0$,~\eqref{eq: ineq 3} holds because there are at most $\alpha^{-4}\mu_0d$ edges in $\chi$ incident to $x$, and~\eqref{eq: ineq 4} holds since the expected number of vertices in $V \setminus V'$ adjacent to $w$ is $2d(|V| - |V'|)(1 + O(\mu_0))/n$ by~\eqref{eq: edge prob}.
    The following claim bounds $|S_5|$ deterministically for any $H'$.

    \begin{claim}\label{clm: S5}
        \[
        |S_5| \leq 6\alpha^4d^2|X| = 6\alpha^4d^2(|X| - i + 1)(1 + o(1)).
        \]
    \end{claim}

    Thus, taking the expectation of (\ref{eq: dR(H)}) and $\mu_0$ to be sufficiently small, we have 
    \begin{equation}
        \mathbb E [d_{\mathscr G}(H)]\ge \mathbb E [d_{\mathscr G'}(H)] \geq (|X| - i + 1)(\alpha^2 - 6\alpha^4 - (4 + \alpha^{-4})\mu_0)d^2(1 + O(\mu_0)) \geq \alpha^3(|X| - i)d^2,
    \end{equation}
    where the third inequality holds by~\eqref{eq: alpha condition}. This completes the proof of this claim.
    \qed

    \medskip
    
    \noindent\textit{Proof of \cref{clm: S5}.}
        Fix $H'$. A vertex in $U \times V$ is called bad if it is incident to $\mu_0\alpha^{-4}d$ edges in $\chi$. An edge in $X$ is called bad if it is incident to a bad vertex. Note that $H' - wz - yx + wx + yz \notin \mathcal G_i$ only if $wx$ is bad. We will show that there are at most $6\alpha^4|X|$ bad edges in $X$. By choosing $wx$ and then $yz$, it follows that $|S_5| \leq 6\alpha^4d^2|X|$.  We prove that there are at most $6\alpha^4|X|$ bad edges by double counting the number of bad vertices using the number of bad edges and $|\chi|$. Let $q_e$ denote the number of bad edges in $H'$ and $q_v$ denote the number of bad vertices in $H'$. Note that each bad edge is incident to a bad vertex and each bad vertex is incident to at most $\mu_0n$ bad edges. Thus, 
        \begin{equation}\label{eq: q_v q_e}
            q_e \leq \mu_0nq_v.
        \end{equation}
        Also, each bad vertex is incident to $\mu_0\alpha^{-4}d$ edges in $\chi$. Thus, in total, the bad vertices are incident to at least $\mu_0dq_v/2\alpha^4$ edges in $\chi$, so
        \begin{equation}\label{eq: q_v chi}
            \mu_0dq_v/2\alpha^4 \leq i-1.
        \end{equation}
        Combining~\eqref{eq: q_v q_e} and~\eqref{eq: q_v chi} then gives that 
        \begin{align*}
            q_e/n \leq 2\alpha^4(i-1)/d,
        \end{align*}
        and so 
        \begin{equation}\notag
            q_e \leq \frac{2\alpha^4n(i-1)}{d} \leq 6\alpha^4|X|,
        \end{equation}
        as desired.

\qed

\subsection{Proof of \cref{clm: small torso max}}\label{sec: small torso max}
As in the proof of \cref{clm: s_B = 0} and \cref{clm: s_B > 0}, we begin by expanding each binomial coefficient using \cref{cor: choose} (except for $\binom{as'}{m}$, which we upper bound by $(eas'/m)^{m}$). Observe that
\begin{align}
    &\binom{as'}{m} \leq \left(\frac{eas'}{m}\right)^{m} = \left(\frac{n}{2d}\right)^{m}\exp[O(k_A + k_B)] \notag\\
    &\binom{b - s'}{m'/a}^a = \frac{(n/2)^{m'}(1 - m'/(ab - as'))^{-ab + as' +m' -a/2}}{(2\pi)^{a/2}d^{m' + a/2}(1 + a/12m' + O(a^2/m'^2))^{a}}\exp[O(k_A + k_B)]\notag\\
    &\binom{a}{m'/(b - s')}^{b - s'} = \frac{(n/2)^{m'}(1 - m'/(ab - as'))^{-ab + as' +m' -b/2 + s'/2}}{(2\pi)^{b/2}d^{m' + b/2 - s'/2}(1 + (b-s')/12m' + O(b^2/m'^2))^{b-s'}}\exp[O(k_A + k_B)]\notag\\
    &\binom{a(b-s')}{m'} = \frac{(n/2)^{m'  - 1/2}}{d^{m' + 1/2}(1 - m'/(ab - as'))^{ab - as' - m'}}\exp[O(k_A + k_B)]\notag\\
    &\binom{n^2/4}{dn/2}   = \frac{(n/2)^{dn}}{(dn/2)^{dn/2 + 1/2}(1 - 2d/n)^{n^2/4 - dn/2}}\exp[O(1)]\notag\\
     &\binom{n/2}{d}^n = \frac{(n/2)^{dn}}{(2\pi)^{n/2}d^{dn + n/2}(1 + 1/12d + O(1/d^2))^{n}(1 - 2d/n)^{n^2/2 - dn + n/2}}\exp[O(1)].\notag
\end{align}

Moreover, observe that 
\[
(1 + a/12m' + O(a^2/m'^2))^a(1 + (b-s')/12m' + O(b^2/m'^2))^{b - s'} = (1 + 1/12d)^n\exp[O(k_A + k_B)]
\]
and 
\[
(1 - m'/(ab - as'))^{ab - as' - m' + a/2 + b/2 - s'/2} = (1 - 2d/n)^{n^2/4 - dn/2 + n/2}\exp[O(k_A + k_B)].
\]
Thus, for any $(s', m)$ that satisfy~\eqref{eq: mS' relaxed} (or equivalently~\eqref{eq: mS' relaxed 1})
\begin{equation}\notag
    \frac{ \binom{as'}{m}\binom{b - s'}{m'/a}^a\binom{a}{m'/(b - s')}^{b-s'}\binom{n^2/4}{dn/2}}{\binom{a(b - s')}{m'}\binom{n/2}{d}^n} \leq (d/n)^{k_A + k_B}d^{s_A}\exp[O(k_A + k_B)].
\end{equation}
Observe that $d^{s_A} = \exp[s_A\log d] = \exp\left[O\left(\frac{k'_B\log d}{d}\right)\right] = \exp[O(k_B)]$. The desired result follows immediately.
\qed

\bibliographystyle{plain}

\bibliography{Greg.bib}

\newpage

\appendix
\section{Appendix}\label{sec: Appendix}
In this appendix, we establish a number of approximations that are useful for calculations throughout the paper. To do so, it is useful to define binomial coefficients using the gamma function. 

\begin{definition}
    For $z > 0$, the gamma function is defined as follows:
    \[
    \Gamma(z) = \int_{0}^\infty t^{z-1}e^{-t}dt.
    \]
\end{definition}
In particular, when $z$ is a non-negative integer, $\Gamma(z + 1) = z!$. Thus we may define binomial coefficients as follows:
\begin{definition}
    For any $0 \leq y \leq x$, 
    \[
    \binom{x}{y} = \frac{\Gamma(x+1)}{\Gamma(y+1)\Gamma(x - y + 1)}.
    \]
\end{definition}

The following is Stirling's approximation of the gamma function. 

\begin{prop}\label{prop: stirling} \cite{Stirling1730} For any $a \rightarrow \infty$,
    \[
    \Gamma(a+1) = (2\pi)^{1/2}a^{a + 1/2}e^{-a}(1 + 1/12a + O(1/a^2)).
    \]
\end{prop}
The corollary below can be obtained by applying \cref{prop: stirling} to each term in the binomial coefficient.
\begin{cor}\label{cor: choose}
    \[
    \binom{x}{y} = \frac{x^{y}}{\sqrt{2\pi}\cdot y^{y + 1/2}\left(1 + \frac{1}{12y} + O(1/y^2)\right)\left(1 - \frac{y}{x}\right)^{x - y + 1/2}}\exp[O(1/x)]
    \]
\end{cor}

We now establish some results that allow us to simplify messy binomial coefficients.
For any integer $a \geq 0$ and any $x \geq a$, let $x_a := x(x-1)\cdots(x-a+1)$. 

\begin{prop}\label{prop: choose top}
Let $x = x(n) = \omega(1)$, $y = y(n) = o(1)$, and $z = z(n) = o(x) \in \{1, 2, \ldots, x(1-y) - 1\}$ be such that $xy$ is an integer. Then,
    \begin{equation*}
         \binom{x(1 - y)}{z} \geq \binom{x}{z}\exp\left[-yz - \frac{yz^2}{2x}(1 + o(1))- O(y^2z)\right].
    \end{equation*}
\end{prop}
\begin{proof}
    Observe that 
    \begin{align}
        \frac{\binom{x(1-y)}{z}}{\binom{x}{z}}&= \frac{(x - xy)!}{x!} \cdot \frac{(x - z)!}{(x - xy - z)!}\notag\\
        &= \frac{(x - z)_{xy}}{x_{xy}}. \label{eq: choose top 1}
        \end{align}
         Thus, \eqref{eq: choose top 1} gives 
         \begin{align*}
             \frac{\binom{x(1-y)}{z}}{\binom{x}{z}}  &\geq \left(\frac{x(1 - y) + 1 - z}{x(1 - y) + 1}\right)^{xy} \\
             &= \left(1 - \frac{z}{x(1 - y) + 1}\right)^{xy}\\
             &= \exp\left[-\frac{yz}{1 - y + 1/x} - \frac{yz^2}{x} (1 + o(1))\right]\\
             &= \exp\left[-yz - \frac{yz^2}{2x}(1 + o(1))- O(y^2z)\right],
         \end{align*}
        
    as desired.
\end{proof}
\begin{prop}\label{prop: choose bot}
    Let $x = x(n) = \omega(1)$, $y = y(n) = o(1)$, and $z = z(n) \in \{1, 2, \ldots, x\}$ be such that $zy$ is an integer and $x - z(1 - y) \geq z$. Then,
    \begin{align*}
        \binom{x}{z(1 - y)} \geq\binom{x}{z}\left(\frac{z}{x}\right)^{yz}\exp\left[\left( \frac{yz^2}{x}- y^2z\right)(1 + o(1))\right].
    \end{align*}
        
\end{prop}
\begin{proof}
    Observe that 
    \begin{align}
        \frac{\binom{x}{z(1-y)}}{\binom{x}{z}}&= \frac{z!}{(z - yz)!} \cdot \frac{(x - z)!}{(x - z + yz)!} \notag\\
        &= \frac{z_{yz}}{(x - z + yz)_{yz}}\label{eq: choose bottom}.
    \end{align}
    Thus, \eqref{eq: choose bottom} gives 
    \begin{align*}
        \frac{\binom{x}{z(1-y)}}{\binom{x}{z}} &\geq \left(\frac{z - yz + 1}{x - z + 1}\right)^{yz}\\
        &= \left(\frac{z}{x}\right)^{yz}\left(\frac{1 - y + 1/z}{1 - z/x + 1/x}\right)^{yz}\\
        &= \left(\frac{z}{x}\right)^{yz}\left(1 - y + 1/z\right)^{yz}\left(1 + \frac{z-1}{x} + O(z^2/x^2)\right)^{yz}\\
        &= \left(\frac z x \right)^{yz}\exp\left[\left(\frac{yz^2}{x} - y^2z\right)(1 + o(1))\right],
    \end{align*}
    as desired.
\end{proof}

    \begin{prop}\label{prop: choose approx ub top}
        For any $x = x(n),y = y(n) = o(1),z = z(n)$ where $x(1 - y)/2 \geq z$, we have 
        \[
        \binom{x(1 - y)}{z} = \binom{x}{z}\exp[O(yz + 1/x)].
        \]
    \end{prop}
    \begin{proofclaim}
        We use \cref{cor: choose} to approximate $\binom{x(1 - y)}{z}$. This gives 
        \begin{align*}
            \binom{x(1 - y)}{z} &= \frac{x^z(1 - y)^z}{z!\big(1 - z/(x(1-y))\big)^{x - xy - z + 1/2}}\exp[O(1/x)]\\
            &= \frac{x^z}{z!(1 - z/x)^{x - z + 1/2}}\exp[O(yz + 1/x)]\\
            &= \binom{x}{z}\exp[O(yz + 1/x)].
        \end{align*}
    \end{proofclaim}

\begin{prop}\label{prop: choose approx ub bot}
    For any $x = x(n), y = y(n) = o(1), z = z(n)$ where $x/2 \geq z$, we have 
    \[
    \binom{x}{z(1 - y)}= \binom{x}{z}\left(\frac{z}{x}\right)^{yz}\exp[O(yz)].
    \]
\end{prop}
\begin{proof}
    We use \cref{cor: choose} to approximate $\binom{x}{z(1 - y)}$. This gives 
        \begin{align*}
            \binom{x}{z(1-y)} &= \frac{x^{z - yz}}{z^{z - yz}\big(1 - z(1-y)/x\big)^{x - z + yz+ 1/2}}\exp[O(1/x)]\\
            &= \frac{x^z}{z!(1 - z/x)^{x - z + 1/2}}\left(\frac{z}{x}\right)^{yz}\exp[O(yz + 1/x)]\\
            &= \binom{x}{z}\left(\frac{z}{x}\right)^{yz}\exp[O(yz + 1/x)].
        \end{align*}
\end{proof}

We next provide a proposition that allows us to upper bound the product of a number of binomial coefficients by the product of their averages (in some sense of average). To do so, it is convenient to use the beta function.

\begin{definition}\label{def: Beta}
    For any $x,y > 0$, The beta function is 
    \[
    B(x,y) = \int_{0}^1t^{x - 1}(1 - t)^{y  - 1}dt.
    \]
\end{definition}
It is known that for $0 \leq  y \leq x$, 
\begin{equation}\label{eq: Binom Beta}
    \binom{x}{y} = \frac{1}{(x+1)B(y+1, x - y + 1)}.
\end{equation}

\begin{prop}\label{prop: binom average}
    Suppose $0 < y \leq x$ and for positive numbers $y_1, \ldots, y_r$, we have $y = \frac{y_1 + \cdots + y_r}{r}$. Then, 
    \[
        \binom{x}{y}^r \geq \prod_{i = 1}^r \binom{x}{y_i}.
    \]
\end{prop}
\begin{proof}
    We prove this by induction on $r$. The base case holds trivially when $r = 1$, so suppose $r > 1$ and the statement holds for all positive integers $r' < r$. Let $y' = \frac{y_1 + \cdots + y_{r-1}}{r-1}$. Then, we have 
    \begin{align*}
        \prod_{i = 1}^r \binom{x}{y_i} \leq \binom{x}{y'}^{r-1}\binom{x}{y_r}.
    \end{align*}
    
    H\"older's inequality states that for positive values $p,q$ such that $1/p + 1/q = 1$ and real valued functions $f$ and $g$, 
    \begin{equation}\label{eq: Holder}
        \int_{0}^1|f(t)g(t)|dt \leq \left(\int_{0}^1|f(t)|^pdt\right)^{1/p}\left(\int_{0}^1|g(t)|^qdt\right)^{1/q}.
    \end{equation}
    By first applying \eqref{eq: Binom Beta}, then applying \cref{def: Beta}, and finally applying \eqref{eq: Holder} with $p = \frac{r}{r-1}$, $q = r$, $f(t) = t^{y'(r-1)/r}(1-t)^{(x - y')(r-1)/r}$, and $g(t)= t^{y_r/r}(1-t)^{x/r - y_r/r}$, we have
    \begin{align*}
        \binom{x}{y'}^{r-1}\binom{x}{y_r} &= \frac{1}{(x+1)^r[B(y' + 1, x - y' + 1)]^{r-1}B(y_r + 1, x - y_r + 1)} \\
        &= \left[(x+1)^r\left(\int_0^1t^{y'}(1-t)^{x - y'}dt\right)^{r-1}\int_0^1t^{y_r}(1-t)^{x-y_r}dt\right]^{-1}\\
        &\leq \left[(x+1)\int_0^1t^y(1-t)^{x-y}dt\right]^{-r}\\
        &= \frac{1}{(x+1)^q[B(y+1, x - y+1)]^q} = \binom{x}{y}^r.
    \end{align*}
\end{proof}

Finally, we establish a result that allows us to expand many of the terms obtained approximating binomial coefficients using \cref{cor: choose}.

\begin{prop}\label{prop: expand log}
For $x = x(n)$, $y = y(n)$, and $z = z(n)$ with $x> y$, there exists $w = w(z,n)$ such that for any integer $w > w_z$, 
    \[
    (1 - y/x)^{x - y} = \exp\left[-y + \sum_{i = 2}^{w}\frac{y^i}{x^{i-1}i(i-1)} + O(z)\right]
    \]
\end{prop}
\begin{proof}
    Observe that 
    \[
    (1 - y/x)^{x - y} = \exp\left[(x - y)\log(1 - y/x)\right].
    \]
    Then, by taking a sufficiently precise Taylor polynomial of $\log(1 - y/x)$, we have that there exists $w_z = w_z(n)$ such that for all $w > w_z$,
    \[
    \log(1 - y/x) = -\sum_{i = 1}^w\frac{y^i}{x^ii} + O(z/x).
    \]
    Thus, 
    \begin{align}
        (1 - y/x)^{x - y} &= \exp\left[-(x - y)\sum_{i = 1}^w\frac{y^i}{x^ii} + O(z)\right]\notag\\
        &= \exp\left[-y + \sum_{i = 1}^w\left(\frac{y^{i+1}}{x^ii} - \frac{xy^{i+1}}{x^{i+1}(i+1)}\right) + O(z)\right]\notag\\
        &= \exp\left[-y + \sum_{i = 2}^w\frac{y^i}{x^{i-1}i(i-1)} + O(z)\right]\notag,
    \end{align}
    as desired.
\end{proof}

\section*{Acknowledgments}
ChatGPT 5.6 Sol was used in the proofreading of this paper. It discovered a number of typos and two minor mathematical errors which we corrected without assistance. 

The authors take full responsibility for the accuracy and content of this paper.

\end{document}